\documentclass[a4paper, 11pt, reqno]{amsart}

\usepackage{amsmath, amssymb, amsthm, mathrsfs, hyperref, color, enumerate, tikz, tikz-cd, mathtools, mathdots, calc, mathabx, float}
\usepackage[capitalize]{cleveref}
\usetikzlibrary{cd, calc, matrix, shapes, decorations, decorations.pathreplacing, intersections}
\usetikzlibrary{decorations.markings}

\usepackage{booktabs, array}
\newcolumntype{C}{>{$}c<{$}}
\newcommand{\obj}[2]{{{}^{#2}K_{#1}}}

\renewcommand{\P}{\mathbb{P}}
\newcommand{\R}{\mathbb{R}}
\newcommand{\C}{\mathbb{C}}

\newcommand{\x}{\vec{x}}
\newcommand{\y}{\vec{y}}
\usepackage{lscape}
\usepackage{adjustbox}
\usetikzlibrary{shapes}
\newcommand{\Z}{\mathbb{Z}}

\newcommand{\mf}{\mathrm{mf}}
\newcommand{\Cone}{\operatorname{Cone}}

\newcommand{\Tw}{\operatorname{Tw}}

\newcommand{\gr}{\operatorname{gr}}

\newcommand{\Hom}{\operatorname{Hom}}
\newcommand{\Ext}{\operatorname{Ext}}

\renewcommand{\hom}{\operatorname{hom}}
\newcommand{\Spec}{\operatorname{Spec}}

\newcommand{\diff}{\mathrm{d}}
\newcommand{\w}{\mathbf{w}}

\newcommand{\HH}{\mathrm{H}}
\newcommand{\wt}{\widecheck{\mathbf{w}}}
\newcommand{\wttilde}{\widecheck{\widetilde{\w}}}
\newcommand{\eps}{\varepsilon}

\newcommand{\xt}{\check{x}}
\newcommand{\yt}{\check{y}}
\newcommand{\fact}{\widecheck{w}}
\newcommand{\vc}[2]{{{}^{#2}V_{#1}}}
\newcommand{\vcpr}[2]{{{}^{#2}V^\mathrm{pr}_{#1}}}
\theoremstyle{plain}
\newtheorem{mthm}{Theorem}

\newtheorem{mconj}{Conjecture}
\newtheorem{thm}{Theorem}[section]
\newtheorem{lem}[thm]{Lemma}
\newtheorem{prop}[thm]{Proposition}
\newtheorem{cor}[thm]{Corollary}
\newtheorem{mcor}{Corollary}
\theoremstyle{remark}
\newtheorem{ex}[thm]{Example}
\newtheorem{rmk}[thm]{Remark}
\definecolor{color1}{HTML}{FF00FF}

\title[Homological B-H-H mirror symmetry for curve singularities]{Homological Berglund--H\"ubsch--Henningson mirror symmetry for curve singularities}

\author{Matthew Habermann}
\email{m.habermann@lancaster.ac.uk}
\address{Department of Mathematics and Statistics,
	University of Lancaster,
	Bailrigg LA1 4YW}
\begin{document}

\begin{abstract}
	In this article, we establish homological Berglund--H\"ubsch mirror symmetry for curve singularities where the A--model incorporates equivariance, otherwise known as homological Berglund--H\"ubsch--Henningson mirror symmetry, including for certain deformations of categories. More precisely, we prove a conjecture of Futaki and Ueda \cite[Conjecture 6.1]{FutakiUedaD} which posits that the equivariance in the A--model can be incorporated by pulling back the superpotential to the total space of the corresponding crepant resolution. Along the way, we show that the B--model category of matrix factorisations has a tilting object whose length is the dimension of the state space of the FJRW A--model, a result which might be of independent interest for its implications in the Landau--Ginzburg analogue of Dubrovin's conjecture.
\end{abstract}
\maketitle
\section{Introduction}
Given a polynomial $f:\C^n\rightarrow \C$ with an isolated singularity at the origin, one can naturally associate two categories to it; one by studying the singularity defined by $f$ symplectically (the A--model), and the other by studying it algebraically (the B--model). The first category is called the Fukaya--Seidel category of $f$, $\mathcal{FS}(f)$, and is a categorification of the space of vanishing cycles. On the algebro-geometric side, one can study the (dg-)category of matrix factorisations, denoted by $\mathrm{mf}(\C^n,f)$. In the case where $f$ is equivariant with respect to the action of an abelian group $G$ which contains $\C^*$, one can also study the $G$-\emph{equivariant} matrix factorisations, which is denoted by $\mathrm{mf}(\C^n, G,f)$. \\

Whilst there is no general prediction about how the Fukaya--Seidel category and the category of matrix factorisations (equivariant or otherwise) of $f$ should be related, homological mirror symmetry predicts that certain singularities should arise in pairs such that the Fukaya--Seidel category of one singularity matches (in a sense to be made precise shortly) the category of matrix factorisations of its partner, and vice-versa. Crucially for us, this pairing of singularities should respect equivariant structures. The main result of this article confirms this prediction in the case of invertible curve singularities, where we incorporate equivariance into the A--model by following a suggestion of Futaki and Ueda in \cite{FutakiUedaD}. \\
To define an invertible polynomial, consider an $n\times n$ matrix $A$, invertible over $\mathbb{Q},$ with non-negative integer entries $a_{ij}$. To such a matrix, one can associate a polynomial $\w\in\C[x_1,\dots, x_n]$
\begin{align}\label{Polydefn}
\w=\sum_{i=1}^n\prod_{j=1}^nx_{j}^{a_{ij}}.
\end{align}
Note that $\w$ is quasi-homogeneous, with weight system given by $(d_0,d_1,\dots,d_n;h)$, where 
\begin{align*}
\w(t^{d_1}x_1,\dots, t^{d_n}x_n)=t^h\w(x_1,\dots,x_n)
\end{align*}
for any $t\in \C^*$, and $d_0:=h-\sum_{i=1}^nd_i$. Whilst $d_i>0$ for $i>0$, there is a trichotomy of cases depending on $d_0$, and we say that $\w$ is log Fano, log Calabi--Yau or log general type if $d_0$ is negative, zero or positive, respectively.\\

In \cite{BerglundHuebsch}, the authors define the transpose of $\w$ by associating to $A^T$ a polynomial $\wt$ and a weight system $(\widecheck{d}_0,\dots,\widecheck{d}_n;\widecheck{h})$; we call this the Berglund--H\"ubsch transpose. We call the polynomial $\w$ \emph{invertible} if it is of the form \eqref{Polydefn} for some $A$ and both $\w$ and $\wt$ define isolated singularities at the origin. \\

Recall that for $f\in\C[x_1,\dots, x_n]$ and $g\in \C[y_1,\dots, y_m]$, their Thom--Sebastiani sum is defined as 
\begin{align}\label{TSSum}
f\boxplus g=f\otimes 1+1\otimes g\in \C[x_1,\dots, x_n,y_1,\dots, y_m].
\end{align} 
A corollary of Kreuzer--Skarke's classification of quasi-homogeneous polynomials, \cite{KreuzerSkarke}, is that any invertible polynomial can be decoupled into the Thom--Sebastiani sum of \textit{atomic} polynomials of the following three types:
\begin{itemize}
	\item Fermat: $\w=x_1^{p_1}$,
	\item Loop: $\w=x_1^{p_1}x_2+x_2^{p_2}x_3+\dots+x_n^{p_n}x_1$,
	\item Chain: $\w=x_1^{p_1}x_2+x_2^{p_2}x_3+\dots+x_n^{p_n}$.
\end{itemize}
The Thom--Sebastiani sums of polynomials of Fermat type are also called Brieskorn--Pham. 
\begin{ex}
	Consider the matrix $A=\begin{pmatrix}3 & 0 & 0 \\
	1 & 3 & 0\\
	0 & 0 & 2
	\end{pmatrix}$. The corresponding polynomial is the $E_7$ singularity, and is the Thom--Sebastiani sum of $x^3+xy^3$ and $z^2$. In fact, ADE singularities in normal form and in all dimensions are examples of invertible polynomials.
\end{ex}

A key piece of data which is associated to an invertible polynomial $\w$ is its \emph{maximal symmetry group}
\begin{align}\label{maxSymGp}
\Gamma_{\w}:=\{(t_1,\dots,t_n,t_{n+1})\in(\C^*)^{n+1}| \ t_1^{a_{1,1}}\dots t_n^{a_{1,n}}=\dots =t_1^{a_{n,1}}\dots t_n^{a_{n,n}}=t_{n+1} \}.
\end{align}
In other words, by setting $\chi_\w$ to be the character given by projection onto $t_{n+1}$, and where $\Gamma_{\w}$ acts on $\C^n$ in the obvious way, $\w$ is $\Gamma_\w$-equivariant of degree $\chi_\w$. Note that $t_{n+1}$ is completely determined by the other $t_i$, and so we consider $\Gamma_\w\subseteq(\C^*)^n$. Since $\w$ is quasi-homogeneous, there is a natural inclusion
\begin{align}\label{groupInclusion}
\begin{split}
\varphi: \C^*&\rightarrow \Gamma_\w\\
t&\mapsto (t^{d_1},\dots,t^{d_n}). 
\end{split}
\end{align}
We call any subgroup of $\Gamma_\w$ containing the image of $\varphi$ \emph{admissible}. In particular, observe that any admissible subgroup is of finite index. The B--model of the homological mirror symmetry prediction we consider will be a pair $(\w,\Gamma)$, where $\Gamma$ is an admissible subgroup. In the case where $\Gamma=\Gamma_{\w}$, we call $\w$ (and the corresponding transpose $\wt$) \emph{maximally graded}, and drop $\Gamma_{\w}$ from the notation. \\

Before moving on, we briefly discuss the better-understood maximally graded case. In this setting, homological Berglund--H\"ubsch mirror symmetry is the following\footnote{All categories here are appropriately derived.}.
\begin{mconj}[\cite{UedaSimple, Takahashi, LekiliUeda}]\label{HomologicalBHHMSConjecture}
Let $\w$ be a maximally graded invertible polynomial and $\wt$ its Berglund--H\"ubsch transpose. Then, there is a quasi-equivalence of $\Z$-graded pre-triangulated $A_\infty$-categories over $\C$
\begin{align*}
\mathrm{mf}(\C^n, \Gamma_\w,\w)\simeq \mathcal{FS}(\wt).
\end{align*}
\end{mconj}
In the above, the B--model is the category of $\Gamma_\w$-equivariant matrix factorisations of $\w$, whilst $\mathcal{FS}(\wt)$ is the Fukaya--Seidel category of $\wt$ defined in \cite{SeidelBook}. The conjecture has been established for Brieskorn--Pham polynomials and Thom--Sebastiani sums of singularities of type $A$ and $D$ in \cite{FutakiUedaBP} and \cite{FutakiUedaD}, respectively, as well as for all invertible curve singularities in \cite{HabermannSmith}. Moreover, strong evidence in the case of chain polynomials was provided in \cite{PolishVarol}, where it was shown that the B--model satisfies a certain recursion relation for directed $A_\infty$-categories, and the corresponding argument on the A--side was sketched in detail. \\
For any admissible group $\Gamma\subseteq\Gamma_\w$, the category $\mathrm{mf}(\C^n, \Gamma,\w)$ is theoretically unproblematic (although it is still difficult to compute in practice), and so it is natural to ask which symplectically defined category should be its mirror, extending \cref{HomologicalBHHMSConjecture} to allow for non-maximally graded symmetry groups. To this end, we introduce the notion of Berglund--Henningson dual group, defined as
\begin{align}\label{DualGpDefn}
\widecheck{\overline{\Gamma}}:=\Hom(\Gamma_{\w}/\Gamma,\C^*).
\end{align} 
By construction, this is a subgroup of $\ker \chi_{\wt}\subseteq \mathrm{SL}(n,\C)$, the group of symmetries with respect to which $\wt$ is \emph{invariant}. This generalisation of invertible polynomials to include different symmetry groups was initially studied in \cite{BergHenn} and further explored in \cite{Krawitz}. By \cite[Proposition 3]{EbelingTakahashi}, the definition given above of $\widecheck{\overline{\Gamma}}$ is a reformulation\footnote{Strictly speaking, the cited result pertains to the closely related maximal group of symmetries which keeps $\w$ invariant, although the resulting quotient, and therefore dual groups, are the same.} of \cite[Definition 15]{Krawitz}. The appropriate A--model should then be an `orbifold Fukaya--Seidel category' associated to the Landau--Ginzburg model $\wt:[\C^n/\widecheck{\overline{\Gamma}}]\rightarrow \C$.  \\
In general, the definition of an orbifold Fukaya--Seidel category is a difficult problem (\cite[Problem 3]{FutakiUedaProceedings}), even in the case of a Lefschetz fibration whose fibres are preserved by the group action, as studied in \cite[Section 13.1]{ChoHong}. More problematic still is, if one begins with a hypersurface singularity, $\wt$, one must, at least in the original formulation of \cite{SeidelBook}, Morsify $\wt$ in order to construct the Fukaya--Seidel category. Whilst Morsifications abound, these will, in general in the case where $\wt$ is equivariant with respect to a group action, not respect this equivariance -- cf. \cref{NoEqMorseEx}. \\
Motivated by extending \cref{HomologicalBHHMSConjecture} to allow for non-maximally graded groups on the B--side,  recent work of Cho, Choa and Jeong in \cite{CCJ} constructs a new $\Z/2$-graded $A_\infty$-category which they take as a definition of the orbifold Fukaya--Seidel category and conjecture that it agrees with the standard definition in the non-equivariant case (\cite[Conjecture 1.2]{CCJ}). This is a robust definition in that it avoids the need to Morsify and is applicable to any polynomial which is not of log general type, including in higher dimensions. On the other hand, all invertible curve singularities, with the exception of $x^2+y^2$, are of log general type, so this work is not immediately applicable to curves; the extension to include curves was recently studied in \cite{CCJ24}. In contrast to these techniques, the main result of this article concerns an alternative approach in two variables suggested by Futaki and Ueda in \cite{FutakiUedaD}.\\

Restricting to the case of curves, Futaki and Ueda suggest a model for the orbifold Fukaya--Seidel category, which is defined as follows. Firstly, observe that $\widecheck{\overline{\Gamma}}\simeq \mu_{\ell}$ for $\ell=[\Gamma_{\w}:\Gamma]$ and that this group acts diagonally on $\C^2$ by $\xi\cdot(\xt,\yt)=(\xi\xt,\xi^{-1}\yt)$. That it is a finite group follows from the admissibility assumption. In the case at hand, it is also immediate that it is an abelian group. In combination with the general fact (for example, in \cite{Krawitz} or \cite[Section 2]{EbelingTakahashi}) that the dual group of an admissible subgroup is in $\mathrm{SL}(n,\C)$ ($n=2$ for us), we can see that the action of the dual group is as above.  Since $\widecheck{\overline{\Gamma}}\subseteq \ker\chi_{\wt}$, we have that $\wt$ descends to a map
\begin{align*}
\widecheck{\overline{\w}}:X\rightarrow \C,
\end{align*}
where $X=\C^2/\mu_\ell$ is the $A_{\ell-1}$ singularity. Then, let 
\begin{align*}
\pi:\widetilde{X}\rightarrow X
\end{align*}
be the crepant resolution of $X$ and define
\begin{align*}
\wttilde:=\pi^*\widecheck{\overline{\w}}:\widetilde{X}\rightarrow \C.
\end{align*}
We will discuss the issue of symplectic form in \cref{CrepResolutionSection}, although note that $\widetilde{X}$ will no longer be exact since the curves comprising the exceptional divisor are holomorphic. The slogan of this approach is that we have traded equivariance for non-exactness. We denote the Fukaya--Seidel category associated to $\wttilde:\widetilde{X}\rightarrow \C$ as $\mathcal{FS}(\wttilde)$, although note that, in addition to the non-exactness, the construction of such a category has several subtleties in comparison to the maximally graded case, to be discussed momentarily.  \\

With notation as above, Futaki--Ueda make the following conjecture: 
\begin{mconj}[{\cite[Conjecture 6.1]{FutakiUedaD}}]\label{FUHomologicalBHHMSConjecture}
	Let $\w$ be an invertible polynomial in two variables with admissible group of symmetries $\Gamma$ and $\wt$ its Berglund--H\"ubsch transpose with dual grading group $\widecheck{\overline{\Gamma}}$. Then, there is a quasi-equivalence of $\Z$-graded pre-triangulated $A_\infty$-categories over $\C$
	\begin{align*}
	\mathrm{mf}(\C^2, \Gamma,\w)\simeq \mathcal{FS}(\wttilde).
	\end{align*}
\end{mconj}
Regarding the definition of $\mathcal{FS}(\wttilde)$, observe that the resonant Morsification $\wt_{\varepsilon}=\wt-\varepsilon xy$ in \cite{HabermannSmith} preserves the $\widecheck{\overline{\Gamma}}$-invariance, and so descends to a perturbation of  $\widecheck{\overline{\w}}_{\varepsilon}:X\rightarrow \C$. We will see later that $\wttilde_{\varepsilon}=\pi^*\widecheck{\overline{\w}}_{\varepsilon}$ has only non-degenerate critical points, and we define $\mathcal{FS}(\wttilde)$ to be the Fukaya--Seidel category associated to $\wttilde_{\varepsilon}$. In \cref{EqMorsificationSection}, we justify this definition. \\
Using this construction of the Fukaya--Seidel category, our main result is a proof of this conjecture, including for certain deformations of categories. We notate the deformation of the Fukaya--Seidel category of $\wttilde$ by a B--field $B\in H^2(\widetilde{X};\C)$ as $\mathcal{FS}(\wttilde;B)$, dropping it from the notation when $B=0$. On the B--side, we study the $\Gamma$-equivariant matrix factorisations of $\w_{\vec{\varepsilon}}=\w+\sum_{j=1}^{\ell-1}\varepsilon_i g_i$, where the $g_i$ are the elements of $\mathrm{Jac}\ \w$ which have the same $\Gamma$-degree as $\w$, each $\varepsilon_i\in \C$, and $\w_{\vec{\varepsilon}}$ only has a unique isolated singularity at the origin. 
\begin{mthm}\label{mainTheorem}
	\cref{FUHomologicalBHHMSConjecture} holds for all invertible curve singularities and any admissible symmetry group. Moreover, for any $\Gamma$-equivariant deformation $\w_{\vec{\varepsilon}}$ of $\w$ which has an isolated singularity at the origin, there exists a B--field $B\in H^2(\widetilde{X};\C)$ such that there is a quasi-equivalence of $\Z$-graded pre-triangulated $A_\infty$ categories over $\C$
	\begin{align*}
	\mathrm{mf}(\C^2, \Gamma,\w_{\vec{\varepsilon}})\simeq \mathcal{FS}(\wttilde;B).
	\end{align*}
\end{mthm}
\begin{rmk}
It should be noted that we do not study all deformations of $\w$ which remain $\Gamma$-equivariant since we are demanding that $\w_{\vec{\varepsilon}}$ only has an isolated singularity; however, the deformations we do consider is a Zariski-open set in $\Spec\C[g_1,\dots,g_{N}]$, where $N=\dim_\C\mathrm{HH}^2(\C^2,\Gamma,\w)$. 
\end{rmk}
\begin{rmk}
In \cref{BDeformations}, we will see that the relevant categories in the maximally graded case are rigid, and so we emphasise that the statement regarding deformations is a new result in comparison to this setting.
\end{rmk}
With a model for the equivariant Fukaya--Seidel category in hand, our strategy of proof of \cref{mainTheorem} follows the well-established method of finding and matching generating objects on both sides, including in the deformed cases. We discuss this further in \cref{StrategyOfProof}, but for now, we note that the direct sum of the generating objects in question forms a \emph{tilting object}. Recall that an object $\mathcal{E}$ in a pre-triangulated $A_\infty$/ dg- category is tilting\footnote{We follow the notation convention that morphisms beginning with upper-case letters are at the level of cohomology, whilst those beginning with lower-case letters are chain level morphisms.}  if $\mathrm{End}^i(\mathcal{E})=0$ and $\hom^\bullet(\mathcal{E},X)=0$ implies $X\simeq 0$ for any other object $X$. In the case where a tilting object is the direct sum of exceptional objects, one refers to its length as the length of the full, strong, exceptional collection which these objects form. \\

A famous result of Seidel \cite[Theorem 18.24]{SeidelBook} establishes that the Fukaya--Seidel category has a full (not necessarily strong) exceptional collection given by Lagrangian thimbles. Combining this with \cref{HomologicalBHHMSConjecture} yields the prediction that $\mathrm{mf}(\C^n,\Gamma_{\w},\w)$ should have a full exceptional collection of length $\mu(\wt)$, the Milnor number of the \emph{transpose polynomial}. Strengthening this, Lekili and Ueda conjectured that $\mathrm{mf}(\C^n,\Gamma_{\w},\w)$ should have a tilting object of length $\mu(\wt)$ (\cite[Conjecture 1.3]{LekiliUeda}). Slightly stronger still is the prediction of Hirano and Ouchi that this tilting object should come from a full and \emph{strong} exceptional collection (\cite[Conjecture 1.4]{HiranoOuchi}). This was shown for invertible polynomials in $n\leq 3$ variables in \cite{Kravets} and for all chain polynomials in \cite{HiranoOuchi}. Moreover, it was established in \cite{FKK} that the category of matrix factorisations for any maximally graded invertible polynomial has a full exceptional collection of the predicted length, and that this collection is strong under an additional (and mild) Gorenstein hypothesis. \\
Importantly, Seidel's result is formulated in the setting where the total space is exact and the Landau--Ginzburg model is not equivariant. Therefore, the evidence for predicting a full (strong) exceptional collection of a given length becomes weaker in the non-maximally graded setting. Nevertheless, the construction of a tilting object is a key element of our proof of \cref{mainTheorem}, and we establish the following: 

\begin{mcor}\label{TiltingCor}
	Let $\w$ be an invertible polynomial in two variables with admissible group of symmetries $\Gamma$ of index $\ell$. Then, $\mf(\C^2,\Gamma,\w)$, as well as any deformation considered in \cref{mainTheorem}, has a tilting object of length 
	\begin{align*}
	\frac{\mu(\wt)-1}{\ell}+\ell. 
	\end{align*}
\end{mcor}
The length of this tilting object can be seen as being predicted by the Dubrovin conjecture for Landau--Ginzburg models. This is strongly analogous to the original Dubrovin conjecture \cite{Dubrovin}, and compares the algebraic structure of the open string B--model with the analytic structure of the closed-string A--model of the \emph{same} Landau--Ginzburg model. The relevant closed-string A--model is studied in FJRW theory \cite{FJR}, and should be thought of as a Landau--Ginzburg analogue of Gromov--Witten theory. The input for the theory is as follows. Let $\w$ be an invertible polynomial with $\Gamma$ an admissible group of symmetries as before. Then, we consider $\overline{\Gamma}=\Gamma\cap \ker\chi$, where $\chi=\chi_\w|_{\Gamma}$ (recall that $\chi_\w$ was the character given by projecting onto $t_{n+1}$ in \eqref{maxSymGp}). Concretely, $\overline{\Gamma}_\w$ is given by setting $t_{n+1}=1$ in the presentation of $\Gamma_{\w}$ in \eqref{maxSymGp}, and $\overline{\Gamma}\subseteq\overline{\Gamma}_\w$ is the corresponding subgroup. We call this the admissible group of \emph{finite} symmetries; however, particularly when FJRW theory is the focus, it is also common to refer to $\overline{\Gamma}$ as just being admissible. The state space of the theory is then defined as
\begin{align*}
\mathscr{H}_{\w,{\overline{\Gamma}}}:=\bigoplus_{\gamma\in{\overline{\Gamma}}}\mathscr{H}_{\gamma}^{{\overline{\Gamma}}},
\end{align*}
where $\mathscr{H}_\gamma=H^{\dim_\C S_\gamma}(\C^{\dim_\C S_\gamma},\Lambda_\gamma;\C)$ for $S_\gamma=\Spec \C[x_1,\dots,x_n]^\gamma$ the fixed locus of $\gamma$, $\w|_{S_\gamma}$ is the restriction of $\wt$ to this fixed locus, and $\Lambda_\gamma$ is the stop in $S_\gamma$ determined by $\w|_{S_{\gamma}}$. Just as in Gromov--Witten theory, $\mathscr{H}_{\w,{\overline{\Gamma}}}$ can be equipped with a Frobenius manifold structure. Then, Dubrovin's conjecture in the Landau--Ginzburg setting is the following. 
\begin{mconj}[Landau--Ginzburg Dubrovin conjecture]\label{LGDubrovin}
	Let $\w$ be an invertible polynomial with admissible group of finite symmetries $\overline{\Gamma}$. Then, the Frobenius manifold structure of $\mathscr{H}_{\w,\overline{\Gamma}}$ is generically semisimple if and only if $\mf(\C^n,\Gamma,\w)$ admits a full exceptional collection $\{E_1,\dots, E_N\}$ of length $N=\dim_\C\mathscr{H}_{\w,\overline{\Gamma}}$.
\end{mconj}
For $\w$ and $\Gamma$ as in \cref{mainTheorem}, a short computation shows that
\begin{align}
\dim_\C\mathscr{H}_{\w,\overline{\Gamma}}= \frac{\mu(\wt)-1}{\ell}+\ell,
\end{align}
where $\ell=[\Gamma_{\w}:\Gamma]$ and $\wt$ is the transpose polynomial. Therefore, for the cases studied in this paper, \cref{LGDubrovin} reduces to establishing semisimplicity of the corresponding FJRW A--model. This was recently done in \cite{FHS} for all but two Brieskorn--Pham curve singularities with $\overline{\Gamma}$ corresponding to $\Gamma\simeq \mathrm{im}\ \varphi\simeq\C^*$, for $\varphi$ as in \eqref{groupInclusion}, as well as simple curve singularities.\\
It should be noted that, whilst not trivial, \cref{LGDubrovin} becomes less interesting in the maximally graded case in any number of variables. This is because the corresponding closed-string B--model, Saito's theory of primitive forms, is known to always be semisimple. Therefore, \cref{LGDubrovin} is an immediate consequence of closed-string mirror symmetry, as established in \cite{HLSW}. In the non-maximally graded cases, there is no well-defined closed-string B--model, let alone mirror symmetry results, from which one can deduce (instances of) \cref{LGDubrovin}.
\\

We end this introduction with the observation that \cref{TiltingCor} also relates to previous work at the intersection of algebraic geometry and representation theory, independently of homological mirror symmetry. Namely, all invertible polynomials in two variables yield $L$-graded Gorenstein rings of Krull dimension one with \emph{non-positive} Gorenstein parameter (see \cref{MCMSubsection} for a definition and discussion). The existence of a tilting object for $\Z$-graded Gorenstein rings of Krull dimension one with non-positive Gorenstein parameter was established in \cite[Theorem 1.2]{BIY}, and our corollary can be seen as a partial generalisation of \cite[Theorem 2.1]{BIY}, which specifically studies $\Z$-graded hypersurface singularities $\mathbf{k}[x,y]/(f)$ where $|x|=|y|=1$. In \textit{loc. cit.}, $\mathbf{k}$ is any field, although our result is only proven over $\C$. 
\subsection{Strategy of proof}\label{StrategyOfProof}
Broadly speaking, our strategy of proof follows that of \cite{HabermannSmith}. Namely, we handle the cases of loop, chain and Brieskorn--Pham invertible polynomials separately. We begin with the loop case and study the undeformed category $\mathrm{mf}(\C^2,\Gamma,\w)$, showing that it has a tilting object, which we call $\mathcal{E}$. In order to study the loop A--model, a key element of our proof is the fact that the resonant Morsifications which were utilised in \emph{loc. cit.} are in fact $\widecheck{\overline{\Gamma}}$-invariant, meaning that it descends to $X$ and can then be pulled back to $\widetilde{X}$, resulting in the function $\wttilde_\varepsilon:\widetilde{X}\rightarrow \C$ which has Morse critical points, and whose smooth fibre is the quotient of the Milnor fibre of $\wt$ by $\widecheck{\overline{\Gamma}}$. We discuss subtleties related to this in \cref{EqMorsificationSection}. With this set-up, our approach reduces to the familiar strategy, albeit in the non-exact setting. We discuss this issue in \cref{NonExactFSSection}. \\
We describe the directed category $\mathcal{A}_{\widecheck{\overline{\Gamma}}}$ associated to vanishing cycles in the smooth fibre of $\wttilde_\varepsilon$. We then establish mirror symmetry by showing that, at the level of cohomology, the endomorphism algebra of the direct sum of the objects in $\mathcal{A}_{\widecheck{\overline{\Gamma}}}$ matches the endomorphism algebra of $\mathcal{E}$ before appealing to formality to establish the corresponding chain-level statement. With the loop case handled, we then deal with the chain and Brieskorn--Pham cases, explaining the alterations required in comparison to the loop case. This establishes the first statement of \cref{mainTheorem} about the undeformed categories.\\
In order to extend this statement to the deformed categories, we compute the second and third Hochschild cohomology groups of $\mf(\C^2,\Gamma,\w)$ using a result of \cite{BallardFaveroKatzarkov} in order to identify every (infinite order) deformation of this category with the category of $\Gamma$-equivariant matrix factorisations of a deformation of $\w$. For the A--model, we add a non-unital B--field in order to deform the Floer products in the desired way. The proof of \cref{mainTheorem} for the deformed categories follows by explicitly matching the deformation parameters on both sides.
\subsection{Structure of the paper}
As in \cite{HabermannSmith}, we begin by considering the loop case, first studying the B--model in \cref{LoopBModel}, before moving on to the A--model. We discuss the subtleties regarding choice of Morsification and non-exactness of the total space in \cref{EqMorsificationSection} and \cref{NonExactFSSection}, respectively, with the rest of \cref{LoopAModel} devoted to the study of $\mathcal{FS}(\wttilde)$ for loop polynomials. The section ends with a proof of \cref{mainTheorem} in this undeformed case. \\
We then consider the chain case in Sections \ref{ChainBmodel} and \ref{ChainAmodel}, and Brieskorn--Pham polynomials in \cref{BPPolynomials}, although it should be emphasised that all technical details, both in general and as they differ from the maximally graded case are already present in the loop case; the chain and Brieskorn--Pham cases are essentially simplifications. \\
In \cref{DeformationsSection}, we study deformations of categories, completing the proof of \cref{mainTheorem} and \cref{TiltingCor}.
\subsection{Acknowledgements}
The author would like to thank Jack Smith for his collaboration in \cite{HabermannSmith}, of which this work is an outgrowth, for valuable comments on an earlier draft of this note, and also patiently explaining some technical points about B--fields in Floer theory. He would also like to thank the anonymous referee for careful reading of this paper, their numerous suggestions for improvement, and encouraging the author to extend the main result to include deformations. The author gratefully acknowledges support from the University of Hamburg and the Deutsche Forschungsgemeinschaft under Germany's Excellence Strategy -- EXC 2121 ``Quantum Universe" -- 390833306.
\section{Loop B-model}
\label{LoopBModel}
In this section, our main goal is to study the dg-category of $\Gamma$-equivariant matrix factorisations, $\mf(\C^2,\Gamma,\w)$, where $\w=x^py+y^qx$ and $\Gamma\subseteq \Gamma_\w$ is an admissible subgroup of the maximal grading group of index $\ell\leq d=\gcd(p-1,q-1)$. Without loss of generality, we will assume $p\geq q\geq 2$. This last constraint of $q\geq 2$ comes from the fact that the case of $q=1$ has multiple isolated singularities, and the Milnor fibre for the one at the origin is equivalent to the Milnor fibre of $x^2+y^2$. As in the maximally graded case, we will encode this equivariance by an $L$-grading, where $L$ is a rank one abelian group with cyclic torsion. \\

Let\footnote{This notation is consistent with Futaki--Ueda \cite{FutakiUedaD}, but opposite to that of Dyckerhoff \cite{Dyckerhoff}.} $S=\C[x,y]$ and $R=S/(\w)$. Recall that $\Gamma_\w$-equivariance corresponds to an $L_\w$ grading on $R$, where $L_\w$ is the character group of $\Gamma_\w$. This group is generated by $|x|=\x$, $|y|=\y$ and $|\w|=\vec{c}$, modulo the relations
\begin{align*}
p\x+\y=\x+q\y=\vec{c},
\end{align*}
which is equivalent to the group $\Z^2$ quotiented by the subgroup generated by $(p-1,1-q)$. Correspondingly, for $\Gamma\subseteq\Gamma_{\w}$ of index $\ell$,  
\begin{align*}
L\simeq \Z^2/(\frac{p-1}{\ell},\frac{1-q}{\ell})\simeq\Z\oplus\Z/(\frac{d}{\ell}).
\end{align*}
We will consider this as the group generated by $\x$ and $\y$ modulo the relation $\frac{p-1}{\ell}\x=\frac{q-1}{\ell}\y$, and note that $L/\Z\vec{c}\simeq\Z/\left(\frac{pq-1}{\ell}\right)$. Before moving on, we make a brief detour to remind the reader of the relevant commutative algebra notions which we will require. 
\subsection{Recollections on maximal Cohen--Macaulay modules}\label{MCMSubsection}
For an $L$-graded ring $R$ as above, or as in \cref{ChainBmodel,BPPolynomials}, recall that graded Gorenstein means that there is an isomorphism 
\begin{align*}
\mathrm{RHom}_{\gr-R}(\C,R(-\alpha))\simeq \C[-n],
\end{align*}
where $\alpha$ is the Gorenstein parameter, $n$ is the Krull dimension of $R$, and $M(l)$ for an $L$-graded $R$ (or $S$) module is an internal grading shift such that $M(l)_i=M_{i+l}$. The Gorenstein parameter $\alpha$ is \emph{negative} (resp. zero, positive) if its projection to the $\Z$-factor of $L$ is negative (resp. zero, positive). For an invertible polynomial, this value is $-d_0$, which is negative for all invertible polynomials other than $x^2+y^2$, where it is zero.\\
Later, we will utilise Serre duality for maximal Cohen--Macaulay (MCM) modules which follows from Auslander--Reiten duality (\cite{AusReit}). Recall that, when $R$ is Gorenstein, a finitely-generated $R$-module is MCM if 
\begin{align}\label{MCMExtVanishing}
\Ext_{\gr-R}^i(M,R)=0\ \text{ for all }i>0,
\end{align}
and the \emph{stabilised} category of maximal Cohen--Macaulay modules, denoted by $\underline{\mathrm{MCM}}(R)$, has the same objects as $\mathrm{MCM}(R)$ (i.e. finitely generated $R$ modules such that \eqref{MCMExtVanishing} holds), but the morphisms are given by
\begin{align*}
\underline{\Hom}_{\gr-R}(M,N):=\Hom_{\gr-R}(M,N)/\mathcal{P}(M,N),
\end{align*} 
where $\mathcal{P}(M,N)$ are those morphisms which factor through some free $R$-module. For any $M,N\in \mathrm{MCM}(R)$ such that $\Hom_{\gr-R}(M,N)$ is finite dimensional there is then Serre duality
\begin{align}\label{SerreDuality}
\underline{\Hom}_{\gr-R}(M,N)\simeq \underline{\Hom}_{\mathrm{gr}-R}(N,M(-\alpha)[n-1])^\vee,
\end{align}
where $(-)^\vee$ is the $\C$-linear dual and $n$ is again the Krull dimension of $R$. Note that, in our case, the fact that $\w$ has an \emph{isolated} singularity means that all Hom-spaces are finite dimensional. For a proof of Serre duality in the $\Z$-grading case, we refer to \cite[Corollary 3.5]{IyaTak}, with the $L$-graded case we need being a straightforward generalisation. \\

By an $L$-graded matrix factorisation of rank $r$, $K^\bullet$, we mean a two periodic sequence of $L$-graded, rank $r$, free $S$-modules 
\begin{align}\label{MFStabilisationDef}
K^\bullet = ( \cdots \rightarrow K^i \xrightarrow{k^i} K^{i+1} \xrightarrow{k^{i+1}} K^{i+2} \rightarrow \cdots ),
\end{align}
where $k^{i+1}\circ k^{i}=\w \cdot\mathrm{Id}_{r}$, and graded two-periodicity means a fixed choice of $K^\bullet(\vec{c})\simeq K^\bullet[2]$. This presentation of a matrix factorisation naturally arises from the process of \emph{stabilisation}; see, for example, \cite[Sections 2.1 and 2.2]{Dyckerhoff}. Namely, one starts with a finitely generated $R$-module $M$, and considers an $R$-free resolution. By a famous result of Eisenbud (\cite{Eisenbud}), any such resolution eventually becomes two periodic (i.e. stabilises). To get a matrix factorisation, we extend this two periodic complex indefinitely to the right and then replace the $R$ modules in the complex by the corresponding free $S$ modules. The inverse of this process is given by taking the cokernel of the map $k^0$. In fact (\cite[Prop. 3.5, Theorem 3.10]{Orlov09}), there is an equivalence of categories
\begin{align*}
\mathrm{coker}:\mathrm{HMF}(\C^2,\Gamma,\w)\xrightarrow{\sim}\underline{\mathrm{MCM}}(R),
\end{align*}
where $\mathrm{HMF}(\C^2,\Gamma,\w)$ is the homotopy category of $\mathrm{mf}(\C^2,\Gamma,\w)$. Observe that, if $K^\bullet$ is a matrix factorisation, then we have
\begin{align}\label{GradingShiftFunctor}
K^\bullet[n]\mapsto \mathrm{coker}((-1)^nk^n)=M[n]
\end{align}
under this equivalence. We will use this fact frequently, particularly in the proofs of generation statements.\\
A different perspective on the category of matrix factorisations is provided by a famous result of Buchweitz and Orlov (\cite{Buchweitz}, \cite[Theorem 39]{Orlov09}), which, generalised to stacks in \cite{PolishVainStacks} and applied to our setting, shows that
\begin{align*}
D^b_{\mathrm{sg}}([\w^{-1}(0)/\Gamma])\simeq \mathrm{HMF}(\C^2,\Gamma,\w).
\end{align*}
In the above, the category on the left is the derived category of singularities of the stack $[\w^{-1}(0)/\Gamma]$, which is defined as the quotient of the derived category of coherent sheaves on $[\w^{-1}(0)/\Gamma]$ by the subcategory of perfect complexes. Moreover, $L$-graded $R$-modules naturally correspond to sheaves on $[\w^{-1}(0)/\Gamma]$, yielding an equivalence
\begin{align*}
D^b_{\mathrm{sg}}(\gr-R)\simeq D^b_{\mathrm{sg}}([\w^{-1}(0)/\Gamma]).
\end{align*}
Justified by the above equivalences, we will frequently switch between talking about graded $R$-modules, sheaves on $[\w^{-1}(0)/\Gamma]$ and matrix factorisations. Moreover, observe that \eqref{SerreDuality} and these equivalences gives Serre functors on each of the categories we consider. 
\subsection{Graded matrix factorisations for loop polynomials}
In this subsection, we introduce and study the basic objects from which we will build our tilting object. These are broadly similar to those of \cite[Section 2.1]{HabermannSmith}; however, direct computation as in \textit{loc. cit.} would be intractable, since the ranks of the matrix factorisations we need to consider increases with the index of $\Gamma\subseteq \Gamma_{\w}$. The main technical difference is therefore the utilisation of Serre duality in order to facilitate computation. In particular, we will see that the quiver \cref{figLoopQuiver} with $\ell=1$ corresponds on-the-nose to the \cite[Fig. 1]{HabermannSmith}. \\
From the equivalences in \cref{MCMSubsection}, the intuition is that $\mf(\C^2,\Gamma,\w)$ should correspond to structures sheaves of the irreducible components, sheaves supported at the origin (the singular point of $\w^{-1}(0)$), and $L$-graded shifts of these. This indeed turns out to be the case.\\

Analogously to \cite{HabermannSmith}, we introduce the notation $\w=xyw_1\dots w_\ell$, where 
\begin{align*}
w_r=x^{\frac{p-1}{\ell}}-e^{\frac{\pi\sqrt{-1}}{\ell}}\eta^ry^{\frac{q-1}{\ell}}
\end{align*}
for $\eta$ a fixed primitive $\ell^{\text{th}}$ root of unity. Continuing with notation introduced in \cite[Section 2.2]{HabermannSmith}, we define $w=w_1\dots w_\ell$. With this, there are $\ell+2$ matrix factorisations coming from ($\Gamma$-equivariantly) factoring $\w$. These correspond to 
\begin{align*}
\obj{x}{}^\bullet = ( \cdots \rightarrow S(-\vec{c}) \xrightarrow{yw} S(-\vec{x}) \xrightarrow{x} S \rightarrow \cdots ),
\\ \obj{y}{}^\bullet = ( \cdots \rightarrow S(-\vec{c}) \xrightarrow{xw} S(-\vec{y}) \xrightarrow{y} S \rightarrow \cdots ),
\end{align*}
as well as the matrix factorisations
\begin{align*}
\obj{{w_r}}{}^\bullet = ( \cdots \rightarrow S(-\vec{c}) \xrightarrow{\w/w_r} S(-\frac{p-1}{\ell}\x) \xrightarrow{w_r} S\rightarrow \cdots ).
\end{align*}

These matrix factorisations are the obvious generalisation of the maximally graded case, and, as in that case, we also consider the objects
\begin{align*}
\obj{x}{i}&=\obj{x}{}((i+1-p)\x)\\
\obj{y}{j}&=\obj{y}{}((i+1-q)\y)
\end{align*} 
where $i=p-\frac{p-1}{\ell},\dots,p-1$ and $j=q-\frac{q-1}{\ell},\dots, q-1$.\\

On the other hand, the generalisation of the matrix factorisations corresponding to sheaves supported at the origin require more technical considerations. Namely, for $p-\frac{p-1}{\ell}\leq i\leq p-1$, $1\leq j\leq q-1$, we set $k=\lfloor \frac{(j-1)\ell}{q-1}\rfloor$ and consider the ideal
\begin{align}\label{definingIdeal}
I_{i,j}=&\  (x^{i-(\ell-1-k)\frac{p-1}{\ell}})+\sum_{t=1}^{k}(x^{i-(\ell-k-1+t)\frac{p-1}{\ell}}y^{j-(k+1-t)\frac{q-1}{\ell}}) + (y^{j})\\
=&\ (x^{i-(\ell-1-k)\frac{p-1}{\ell}}, x^{i-(\ell-k)\frac{p-1}{\ell}}y^{j-k\frac{q-1}{\ell}}, \dots, x^{i-(\ell-1)\frac{p-1}{\ell}}y^{j-\frac{q-1}{\ell}},y^{j})
\end{align} 
and the $L$-graded $R$-modules
\begin{align*}
R((i+1)\x+(j+1)\y)/I_{i,j}.
\end{align*}
The corresponding matrix factorisation, which we denote by $\obj{0}{i,j}$, is given by stabilising this module. Namely, an $R$-free resolution can be built by beginning with
\begin{align*}
R(\vec{c}-(k+1)\frac{p-1}{\ell}\vec{x}+j\y)\oplus\bigoplus_{t=0}^{k-1}R(\vec{c})\oplus R((i+1)\x+\y)
\rightarrow{}
R((i+1)\x+(j+1)\y),
\end{align*}
where the map is multiplication by $\begin{pmatrix}
	x^{i-(\ell-1-k)\frac{p-1}{\ell}}& \dots& y^{j}
\end{pmatrix}$, yielding a rank $(k+2)$ matrix factorisation. Observe in particular that substituting $\ell=1$ into this presentation recovers the definition of $\obj{0}{i,j}$ in \cite{HabermannSmith}. The $k$ terms in the middle of the ideal in \eqref{definingIdeal} are all monomials of degree $i\vec{x}+j\vec{y}-(p-1)\vec{x}\ (=i\vec{x}+j\vec{y}-(q-1)\vec{y})$.
\begin{rmk}\label{ChoicesRMK}
It might seem objectionable that the symmetry which exists between $p$ and $q$ has been broken in the above ideals; however, we will later see that, for example, taking the ideals corresponding to $1\leq i\leq p-1$ and $q-\frac{q-1}{\ell}\leq j\leq q-1$, would lead to a tilting object whose endomorphism algebra corresponds to the same quiver-with-relations as the choice we have made -- cf. \cref{ChoicesRMK2}.
\end{rmk}

From this, it is straightforward to check that the maps defining the matrix factorisation are given in even degree by 
\begin{align}\label{EvenDifferential}
\diff_0=\begin{pmatrix}
y^{j-k\frac{q-1}{\ell}} & 0 & 0  & \dots & 0 &x^{p-i+\frac{p-1}{\ell}(\ell-1-k)}y\\
-x^{\frac{p-1}{\ell}}& y^{\frac{q-1}{\ell}} & 0 &\dots& 0 & 0\\
0 & -x^{\frac{p-1}{\ell}} & y^{\frac{q-1}{\ell}} & \dots &0 & 0\\
0 & 0 & -x^{\frac{p-1}{\ell}}&  \dots & 0& 0\\
\vdots &\vdots & \vdots   & \ddots& \vdots  & \vdots \\
0 & 0 & 0  & \hdots & y^{\frac{q-1}{\ell}} & 0\\
0 & 0 & 0  & \hdots& -x^{i-(\ell-1)\frac{p-1}{\ell}} & xy^{q-j}\\
\end{pmatrix},
\end{align}
and in odd degree by $\diff_1=\mathrm{Adj}(\diff_0)$, the adjugate matrix. Explicitly, we have that $\obj{0}{i,j}$ corresponds to the matrix factorisation 
 \begin{center}
	\begin{tikzcd}[row sep=4ex, column sep=5ex]
		S(\vec{c}-\frac{p-1}{\ell}\x) \ar[d, phantom, description, "\bigoplus"]  \ar[dd, phantom, description, "\cdots\hskip7ex\phantom{\hskip20ex\cdots}"] &S(\vec{c}-(k+1)\frac{p-1}{\ell}\x+j\y)  \ar[d, phantom, description, "\bigoplus"] & S(2\vec{c}-\frac{p-1}{\ell}\x)  \ar[d, phantom, description, "\bigoplus"] \ar[dd, phantom, description, "\cdots\hskip7ex\phantom{\hskip-40ex\cdots}"]
		\\ 
		\bigoplus_{t=0}^{k-1} S(\vec{c}-\frac{p-1}{\ell}\x) \ar[d, phantom, description, "\bigoplus"] \ar[r, shorten >=3ex, shorten <=3ex,"\diff_0"] &\bigoplus_{t=0}^{k-1}S(\vec{c}) \ar[d, phantom, description, "\bigoplus"] \ar[r, shorten >=3ex, shorten <=3ex,"\diff_1"] & \bigoplus_{t=0}^{k-1}S(2\vec{c}-\frac{p-1}{\ell}\x)\\
		S((i+1)\x+(j+1)\y-\vec{c})&S((i+1)\x+\y)&S((i+1)\x+(j+1)\y)
	\end{tikzcd}
\end{center}
where the rightmost term is in cohomological degree $0$ and $\diff_0$, $\diff_1$ map between the entire columns, not just he middle summands.\\

As in the maximally graded case, we are interested in a full subcategory $\mathcal{B}$ of $\mathrm{mf}(\C^2,\Gamma,\w)$ consisting of the objects described above. Namely, let $\mathcal{B}$ be the category consisting of the $\frac{pq-1}{\ell}+\ell$ objects 
\[
\{\obj{0}{i,j}, \obj{x}{i}[3],\obj{y}{q-\frac{q-1}{\ell}}[3],\dots, \obj{y}{q-1}[3], \obj{w_1}{}[3],\dots, \obj{w_\ell}{}[3]\}_{i=p-\frac{p-1}{\ell}, \dots, p-1;\ j=1, \dots, q-1}.
\]
The reason for the shifts is so that all morphisms turn out to have cohomological degree $0$. In the following sections we compute the cohomology level morphisms between the objects in this category. The reader willing to take these computations on faith may skip directly to \cref{LoopEndAlgebra} for the characterisation of $\mathcal{B}$ as a quiver algebra. 
\subsection{Morphisms between the $\obj{x}{}$'s, $\obj{y}{}$'s and $\obj{w_r}{}$'s}
As in \cite{HabermannSmith}, we leverage a result of Buchweitz \cite[Section 1.3, Remark (a)]{Buchweitz} which establishes the following: given $L$-graded $R$-modules $M$ and $M'$ with corresponding stabilisations $K$ and $K'$, we have
\begin{align*}
\Hom^\bullet_{\mathrm{HMF}(\C^2, \Gamma, \w)} (K, K') &\simeq \mathrm{H}^\bullet \left(\Hom_{\gr-R} (K \otimes_{S} R, M')\right)
\end{align*}
The $\Hom$ on the right-hand side is taken component-wise on the complex $K \otimes_S R$. We refer to \cite[Proposition 2.23]{Kravets} for a proof of the statement. Note that Buchweitz's original remark is missing a Gorenstein assumption, and, as demonstrated in \cite[Remark 2.24]{Kravets}, does not hold without it. 

\begin{rmk}
Strictly speaking, Buchweitz's result only applies for $\bullet\gg0$; however, since $[2]\simeq (\vec{c})$ in $\mathrm{HMF}(\C^2, \Gamma, \w)$, we can always achieve that the degree is high enough for the theorem to apply, and this does not affect the result of the calculation. 
\end{rmk}
For calculations regarding the modules $\obj{x}{}$, $\obj{y}{}$ and $\obj{w_r}{}$, the arguments carry over almost verbatim from the maximally graded case \cite[Sections 2.3 and 2.4]{HabermannSmith}, and we have: 

\begin{lem}
	In $\mathrm{HMF}(\C^2, \Gamma, \w)$, we have the following:
	\begin{enumerate}[(i)]
		\item \label{KIorthonormal} For any $i\in \Z$, the objects $\obj{x}{i},\dots, \obj{x}{i+\frac{p-1}{\ell}-1}$ are exceptional and pairwise orthogonal. 
		\item \label{KIIorthonormal} For any $j\in \Z$, the objects $\obj{y}{j},\dots, \obj{y}{j+\frac{q-1}{\ell}-1}$ are exceptional and pairwise orthogonal.
		\item \label{KIIIorthonormal} The objects $K_{w_{1}}, \dots, K_{w_{\ell}}$ are exceptional and pairwise orthogonal.
		\item \label{KIKIIorthonormal}For each $i=p-\frac{p-1}{\ell},\dots, p-1$, $j=q-\frac{q-1}{\ell},\dots, q-1$, $r=1,\dots, \ell$, $\obj{x}{i}$, $\obj{y}{j}$ and $\obj{w_r}{}$ are mutually orthogonal. \hfill\qed
	\end{enumerate}
\end{lem}

\subsection{Morphisms between the $\obj{w_r}{}$'s and $\obj{0}{}$'s}
\label{BMorphisms4}
The fact that $\Hom^\bullet(\obj{w_r}{}, \obj{0}{i,j})=0$ essentially follows from the arguments of \cite[Section 2.5]{HabermannSmith}. In the other direction $\Hom^\bullet(\obj{0}{i,j}, \obj{w_r}{})$ we argue as follows. Firstly, observe that, since $\w$ has an isolated singularity at the origin, morphisms of $R$-modules are finite dimensional, and Serre duality for the corresponding MCM modules applies. We then observe that 
 \begin{align}\label{ARDuality}
\dim_\C \Hom^\bullet(\obj{w_r}{},\obj{0}{i,j}(\vec{c}-\x-\y))=
 \begin{cases}
1\qquad\text{if }\bullet=-3\\
0\qquad\text{otherwise},
 \end{cases}
 \end{align}
 and so $\Hom^\bullet(\obj{0}{i,j}, \obj{w_r}{})$ is non-trivial only in cohomological degree three, where it is rank one; it is then straightforward to write down this non-trivial element. 
 \begin{lem}\label{KijKwrHoms}
 For each $r=1,\dots,\ell$, there is a single morphism between $\obj{0}{i,j}$ and $\obj{w_r}{}$ given by 
 \begin{align*}
 \Hom^{3}(\obj{0}{i,j},\obj{w_r}{})=\C\cdot\begin{pmatrix}
y^{(k+1)\frac{q-1}{\ell}-j}\\
e^{-\frac{\pi i}{\ell}}\eta^{-r}\\
(e^{-\frac{\pi i}{\ell}}\eta^{-r})^2\\
\vdots\\
(e^{-\frac{\pi i}{\ell}}\eta^{-r})^{k}\\
(e^{-\frac{\pi i}{\ell}}\eta^{-r})^{k+1}x^{p-1-i}
 \end{pmatrix},
 \end{align*}
 where $k=\lfloor \frac{(j-1)\ell}{q-1}\rfloor$ as before.
 \end{lem}
\begin{proof}
Observe that the complex computing the morphisms has differential given by $\diff_0^T$ in odd degree and $\diff_1^T=\mathrm{Adj}(\diff_0)^T$ in even degree, where $\diff_0$ is the map from \eqref{EvenDifferential}. It is then clear that the above vector defines an element in the kernel; the only non-trivial term to check is that it is in the kernel of the last row of $\diff_0^T$, although this follows from observing that this factors into $(\ell-1-k)$ polynomials, one of which is $w_r$. It is clear that this is not in the image of $\diff_1^T$, since the $k-1$ constant elements in the kernel cannot be in the image of a degree $\frac{p-1}{\ell}\x$-map. The fact that this spans the cohomology group follows from \eqref{ARDuality}.
\end{proof}
\begin{rmk}\label{MorphismsRMK}
It might be worth reminding the reader that a morphism of matrix factorisations is a pair of morphisms which are compatible with the differentials on both complexes. The Hom-space written in \cref{KijKwrHoms} uniquely determines such a pair of morphisms (up to scaling), although is, strictly, speaking, only half of the data. In this case, the transpose of the vector spanning this space is (up to scaling) the component of the morphism $\obj{0}{i,j}\rightarrow\obj{w_r}{}[3]$ in odd degrees. This is relevant when computing the composition of morphisms, as we will see in \cref{LoopEndAlgebra}.
\end{rmk}

\subsection{Morphisms between $\obj{x}{}$'s and $\obj{y}{}$'s and $\obj{0}{}$'s}
\label{BMorphisms5}

For each $I$ we have that $\Hom^\bullet (\obj{x}{I}, \obj{0}{i,j})$ vanishes, which can be computed in the same way as in the maximally graded case. In order to compute the morphisms in the other direction, we argue again by Serre duality. In particular, we have that

 \begin{align}\label{ARDuality2}
\dim_\C \Hom^\bullet(\obj{x}{I},\obj{0}{i,j}(\vec{c}-\x-\y))=
\begin{cases}
1\qquad\text{if }I=i,\ \bullet=-3\\
0\qquad\text{otherwise},
\end{cases}
\end{align}
and similarly there is only one morphism $\Hom^\bullet(\obj{x}{J},\obj{0}{i,j}(\vec{c}-\x-\y))$ in degree negative three when $J\equiv j\bmod \frac{q-1}{\ell}$, and is zero otherwise. Note that the complexes computing cohomology do not vanish identically in the cases where we claim the cohomology is zero, but the complexes are exact. 
\begin{lem}
	In $\mathrm{HMF}(\C^2, \Gamma, \w)$ there are no morphisms from $\obj{x}{I}$ to $\obj{0}{i,j}$.  There are no morphisms in the other direction unless $I=i$, in which case the morphism space is spanned by $(0,0,\dots,0, 1)$ in degree $3$.  Similarly, the only morphisms between $\obj{y}{J}$ and $\obj{0}{i,j}$ are from the latter to the former, and are given by $(1,0,\dots,0)$ in degree $3$ when $j\equiv J\bmod \frac{q-1}{\ell}$. 
\end{lem}
\begin{proof}
The proof follows the same strategy as the proof of \cref{KijKwrHoms}. Namely, the statement about the morphisms between $\obj{x}{I}$ and $\obj{0}{i,j}$ is immediate after observing that multiplication by $y^\bullet$ is exact, and the only morphism from $\obj{0}{i,j}$ to $\obj{x}{i}$ is in degree three. The statement about morphisms between $\obj{y}{J}$ and $\obj{0}{i,j}$ is proved similarly, although this time multiplication by $x^\bullet$ is exact. 
\end{proof}
\subsection{Morphisms between $\obj{0}{}$'s}
\label{BMorphisms6}
Computing morphisms $\obj{0}{i,j}\rightarrow\obj{0}{I,J}$ is analogous to the maximally graded case. Namely, morphisms are spanned by the module
\begin{align*}
\big(R/I_{I,J}\big)_{(I-i)\x+(J-j)\y}\ .
\end{align*} 
From this, it is immediate that there are no morphisms unless $J\geq j$; however, unlike in the maximally graded case, it is now possible to have $i>I$ since $\frac{p-1}{\ell}\x=\frac{q-1}{\ell}\y$.  Putting this together, we conclude:
\begin{lem}\label{StabOriginMorphisms}
For all $i\in\{p-\frac{p-1}{\ell}, \dots, p-1\}$ and $j\in\{1,\dots,q-1\}$, we have 

\begin{flalign*}
&&\Hom^{\bullet}(\obj{0}{i,j}, \obj{0}{I,J})\simeq &\begin{cases}
\text{span}_\C\{x^{I-i}y^{J-j}, \dots, x^{I-i+k\frac{p-1}{\ell}}y^{(J-j)\bmod \frac{q-1}{\ell}}\}& \text{if } I\geq i,\ J\geq j\\
& \text{and } \bullet=0\\
\text{span}_\C\{x^{I-i+\frac{p-1}{\ell}}y^{J-j-\frac{q-1}{\ell}},\dots, x^{I-i+k\frac{p-1}{\ell}}y^{(J-j)\bmod \frac{q-1}{\ell}}\}& \text{if }\ I< i,\\
&J\geq j+\frac{q-1}{\ell}\\
& \text{and }\bullet=0\\
0&\text{otherwise.}\hfill\qed
\end{cases} &
\end{flalign*}
\end{lem}
\subsection{The total endomorphism algebra of the basic objects.}
From the results of the previous sections, we see that the objects in $\mathcal{B}$ are all exceptional. In fact, the endomorphism algebra of the objects in $\mathcal{B}$ are presented as a quiver-with-relations.
 
\begin{thm}
	\label{LoopEndAlgebra}
	The cohomology-level total endomorphism algebra of the objects of $\mathcal{B}$ is the algebra of the quiver-with-relations described in \cref{figLoopQuiver}, with all arrows living in degree zero.  In particular, $\mathcal{B}$ is a $\Z$-graded $A_\infty$-category concentrated in degree $0$, so is intrinsically formal.

	\begin{figure}[ht]
		\centering
		\begin{tikzpicture}[blob/.style={circle, draw=black, fill=black, inner sep=0, minimum size=\blobsize}, arrow/.style={->, shorten >=6pt, shorten <=6pt}, scale = 0.9]
		\def\blobsize{1.2mm}

		\draw (2, 0) node[blob]{};
		\draw (3.25, 0) node{\ $\cdots$};
		\draw (4.5, 0) node[blob]{};
		\draw (6, 0) node[blob]{};
		\draw [
		thick,
		decoration={
			brace,
			raise=0.5cm
		},
		decorate
		] (1.8,5.25) -- (4.7,5.25)
		node [pos=0.5,anchor=south,yshift=0.55cm] {$\frac{q-1}{\ell}$};
		
			\draw [
		thick,
		decoration={
			brace,
			raise=0.5cm
		},
		decorate
		] (6.25,3.75) -- (6.25,-.25)
		node [pos=0.5,anchor=west,xshift=0.55cm] {$\frac{p-1}{\ell}$};

		\draw[arrow] (2, 0) -- (3, 0)  node[midway,below]{y};
		\draw[arrow] (3.5, 0) -- (4.5, 0) node[midway,below]{y};
		\draw[arrow] (4.5, 0) -- (6, 0) node[midway,below]{$b$};

		\draw[arrow] (2, 0) -- (2, 1)  node[midway,right]{x};
		\draw[arrow] (4.5, 0) -- (4.5, 1)  node[midway,right]{x};
		
		\draw (6,-1) node{$\obj{x}{i}[3]$};
		
		\draw (6.25,5.25) node{$\obj{w_r}{}[3]$};
		
		\begin{scope}[yshift=1cm]
		\draw (2, 0) node[blob]{};
		\draw (3.25, 0) node{\ $\cdots$};
		\draw (4.5, 0) node[blob]{};
		\draw (6, 0) node[blob]{};

		\draw[arrow] (2, 0) -- (3, 0) node[midway,above]{y};
		\draw[arrow] (3.5, 0) -- (4.5, 0) node[midway,above]{y};
		\draw[arrow] (4.5, 0) -- (6, 0) node[midway,below]{$b$};

		\draw[arrow] (2, 0) -- (2, 1)  node[midway,right]{x};
		\draw[arrow] (4.5, 0) -- (4.5, 1)  node[midway,right]{x};
		
		\end{scope}
		
		\begin{scope}[yshift=3.5cm]

		\draw (2, 0) node[blob]{};
		\draw (3.25, 0) node{\ $\cdots$};
		\draw (4.5, 0) node[blob]{};
		\draw (6, 0) node[blob]{};

		\draw[arrow] (2, 0) -- (3, 0) node[midway,above]{y};
		\draw[arrow] (3.5, 0) -- (4.5, 0) node[midway,above]{y};
		\draw[arrow] (4.5, 0) -- (6, 0) node[midway,below]{$b$};

		\draw[arrow] (2, 0) -- (2, 1.5)  node[midway,left]{$a$};
		\draw[arrow] (4.5,0) -- ({4.5+2*cos(360/14)},{2*sin(360/14)} ) node[pos=0.6,above]{$c_\ell$};
		\draw[arrow] (4.5,0) -- ({4.5+2*cos(5*360/28)},{2*sin(5*360/28)} ) node[pos=0.6,right]{$c_1$};
		\draw[arrow] (4.5, 0) -- (4.5, 1.5) node[midway,left]{$a$};
		\end{scope}
		
		\begin{scope}[yshift=5cm]
		\draw (2, 0) node[blob]{};
		\draw (3.25, 0) node{\ $\cdots$};
		\draw (4.5, 0) node[blob]{};

		\end{scope}
		\begin{scope}[yshift=3.5cm, xshift=4.5cm]
		\foreach \a in {2,5}{
			\draw (\a*360/28: 2) node[blob]{};
		}
			\foreach \a in {1,2,3}{
		\draw[fill] (\a*135/14+360/14: 2) circle[radius=1pt];
	}

		\end{scope}
		
		\draw (2, 2.37) node{$\vdots$};
		\draw (4.5, 2.37) node{$\vdots$};
		\draw (6, 2.37) node{$\vdots$};
		
		\begin{scope}[yshift=2.5cm]
		\draw[arrow] (2, 0) -- (2, 1) node[midway,right]{x};
		\draw[arrow] (4.5, 0) -- (4.5, 1) node[midway,right]{x};
		\end{scope}
		
		\draw (3.25, 2.37) node{\ $\iddots$};
		
		\draw[line width=0.5mm, opacity=0.2] (1.5, -0.5) rectangle (5, 4);
		\draw[line width=0.5mm, opacity=0.2] (1.5, 4.5) rectangle (4.9, 5.5);
		\draw[line width=0.5mm, opacity=0.2] (5.5, -0.5) rectangle (6.5, 4);
		
		\begin{scope}[xshift=-3.5cm]
		\draw (2, 0) node[blob]{};
		\draw (3.25, 0) node{\ $\cdots$};
		\draw (4.5, 0) node[blob]{};

		\draw (4.25,5) node{$\obj{y}{j}[3]$};
		
		\draw[arrow] (2, 0) -- (3, 0)  node[midway,below]{y};
		\draw[arrow] (3.5, 0) -- (4.5, 0) node[midway,below]{y};
		\draw[arrow] (4.5, 0) -- (5.5, 0) node[pos=0.65,below]{y};

		\draw[arrow] (2, 0) -- (2, 1)  node[midway,right]{x};
		\draw[arrow] (4.5, 0) -- (4.5, 1)  node[midway,left]{x};
		
		\begin{scope}[yshift=1cm]
		\draw (2, 0) node[blob]{};
		\draw (3.25, 0) node{\ $\cdots$};
		\draw (4.5, 0) node[blob]{};

		\draw[arrow] (2, 0) -- (3, 0) node[midway,above]{y};
		\draw[arrow] (3.5, 0) -- (4.5, 0) node[midway,below]{y};
		\draw[arrow] (4.5, 0) -- (5.5, 0) node[pos=0.65,above]{y};

		\draw[arrow] (2, 0) -- (2, 1)  node[midway,right]{x};
		\draw[arrow] (4.5, 0) -- (4.5, 1)  node[midway,right]{x};
		\end{scope}
		
		\begin{scope}[yshift=3.5cm]
		
		\draw (2, 0) node[blob]{};
		\draw (3.25, 0) node{\ $\cdots$};
		\draw (4.5, 0) node[blob]{};

		\draw[arrow] (2, 0) -- (3, 0) node[midway,above]{y};
		\draw[arrow] (3.5, 0) -- (4.5, 0) node[midway,above]{y};
		\draw[arrow] (4.5, 0) -- (5.5, 0) node[pos=0.65,above]{y};

		\draw[arrow] (4.5, 0) -- (8, -3.5) node[midway,above]{x};
		\draw[arrow] (2, 0) -- (5.5, -3.5) node[midway,above]{x};
		
		\end{scope}

		\draw (2, 2.37) node{$\vdots$};
		\draw (4.5, 2.37) node{$\vdots$};

		\begin{scope}[yshift=2.5cm]
		\draw[arrow] (2, 0) -- (2, 1) node[midway,right]{x};
		\draw[arrow] (4.5, 0) -- (4.5, 1) node[midway,right]{x};
		\end{scope}
		
		\draw (3.25, 2.37) node{\ $\iddots$};
		
		\draw[line width=0.5mm, opacity=0.2] (1.5, -0.5) rectangle (5, 4);

		\end{scope}

		\begin{scope}[xshift=-7cm]

\draw[arrow] (4.5, 0) -- (5.5, 0) node[pos=0.65,below]{y};
\draw (2.25,0) node{$\cdots$};
\draw (4.5,0) node{$\cdots$};
\draw (3.375,0) node{$\cdots$};
\draw (3.25, 2.37) node{\ $\iddots$};
		\draw (2, 2.37) node{$\vdots$};
\draw (4.5, 2.37) node{$\vdots$};
\begin{scope}[yshift=1cm]

\draw[arrow] (4.5, 0) -- (5.5, 0) node[pos=0.65,above]{y};
\draw (2.25,0) node{$\cdots$};
\draw (4.5,0) node{$\cdots$};
\draw (3.375,0) node{$\cdots$};
\end{scope}

\begin{scope}[yshift=3.5cm]
\draw (2.25,0) node{$\cdots$};
\draw (4.5,0) node{$\cdots$};
\draw (3.375,0) node{$\cdots$};
\draw[arrow] (4.5, 0) -- (5.5, 0) node[pos=0.65,above]{y};
\draw[arrow] (2, 0) -- (5.5, -3.5) node[midway,above]{x};
\draw[arrow] (4.5, 0) -- (8, -3.5) node[midway,above]{x};

\end{scope}

\end{scope}

				\begin{scope}[xshift=-10.5cm]
		\draw (2, 0) node[blob]{};
		\draw (3.25, 0) node{\ $\cdots$};
		\draw (4.5, 0) node[blob]{};

		\draw[arrow] (2, 0) -- (3, 0)  node[midway,below]{y};
		\draw[arrow] (3.5, 0) -- (4.5, 0) node[midway,below]{y};
		\draw[arrow] (4.5, 0) -- (5.5, 0) node[pos=0.65,below]{y};

		\draw[arrow] (2, 0) -- (2, 1)  node[midway,right]{x};
		\draw[arrow] (4.5, 0) -- (4.5, 1)  node[midway,left]{x};
		
		\begin{scope}[yshift=1cm]
		\draw (2, 0) node[blob]{};
		\draw (3.25, 0) node{$\cdots$};
		\draw (4.5, 0) node[blob]{};

		\draw[arrow] (2, 0) -- (3, 0) node[midway,above]{y};
		\draw[arrow] (3.5, 0) -- (4.5, 0) node[midway,below]{y};
		\draw[arrow] (4.5, 0) -- (5.5, 0) node[pos=0.65,above]{y};

		\draw[arrow] (2, 0) -- (2, 1)  node[midway,right]{x};
		\draw[arrow] (4.5, 0) -- (4.5, 1)  node[midway,right]{x};
		\end{scope}
		
		\begin{scope}[yshift=3.5cm]
		
		\draw (2, 0) node[blob]{};
		\draw (3.25, 0) node{\ $\cdots$};
		\draw (4.5, 0) node[blob]{};

		\draw[arrow] (2, 0) -- (3, 0) node[midway,above]{y};
		\draw[arrow] (3.5, 0) -- (4.5, 0) node[midway,above]{y};
		\draw[arrow] (4.5, 0) -- (5.5, 0) node[pos=0.65,above]{y};

		\draw[arrow] (2, 0) -- (5.5, -3.5) node[midway,above]{x};
		\draw[arrow] (4.5, 0) -- (8, -3.5) node[midway,above]{x};
		\draw (2,1) node{$\obj{0}{i,j}$};
		\end{scope}

		\draw (2, 2.37) node{$\vdots$};
		\draw (4.5, 2.37) node{$\vdots$};

		\begin{scope}[yshift=2.5cm]
		\draw[arrow] (2, 0) -- (2, 1) node[midway,right]{x};
		\draw[arrow] (4.5, 0) -- (4.5, 1) node[midway,right]{x};
		\end{scope}
		
		\draw (3.25, 2.37) node{\ $\iddots$};
		
		\draw[line width=0.5mm, opacity=0.2] (1.5, -0.5) rectangle (5, 4);
		
		\end{scope}

		\begin{scope}[yshift=-50,xshift=-110]
		\draw (0,0) node{\parbox{300pt}{\small \ \textbf{Relations:} \begin{enumerate}[(i)]\item $xy=yx$\item $ay=bx=0$ \item $c_r(x^{\frac{p-1}{\ell}}-e^{\frac{\pi i}{\ell}}\eta^ry^{\frac{q-1}{\ell}})=0$\end{enumerate}}};
		\end{scope}
		\end{tikzpicture}
		\caption{The quiver describing the category $\mathcal{B}$ for loop polynomials. There are $\ell$ blocks of size $\frac{q-1}{\ell}\times\frac{p-1}{\ell}$.\label{figLoopQuiver}}
	\end{figure}
\end{thm}
\begin{rmk}
Before moving on to the proof, we remark that the quiver for maximally graded loop polynomials, as studied in \cite[Theorem 2.13]{HabermannSmith}, is on-the-nose \cref{figLoopQuiver} with $\ell=1$. In particular, the relations (iii) becomes vacuous in this case. Foreshadowing \cref{DeformationsSection}, it is these relations which get deformed when looking at deformations of the algebra.
\end{rmk}
\begin{proof}[Proof of \cref{LoopEndAlgebra}]
In order to prove the statement, all that is required is to show that the morphisms of matrix factorisations above compose in the claimed way. The first and second relations follow from the maximally graded case, \cite[Theorem 2.13]{HabermannSmith}, \textit{mutatis mutandis}. The third relation is not present in the maximally graded case, although checking that it holds is similarly straightforward by composing generators. Namely, consider 
\begin{align*}
x^{\frac{p-1}{\ell}},\ y^{\frac{q-1}{\ell}}&\in\Hom^0(\obj{0}{p-1,(\ell-1)\frac{q-1}{\ell}},\obj{0}{p-1,q-1}),\\
c_r&\in\Hom^0(\obj{0}{p-1,q-1},\obj{w_r}{}[3]).
\end{align*}
Following on from \cref{MorphismsRMK}, we compute the composition of morphisms in odd degree\footnote{We mean morphisms between the odd degrees of the matrix factorisations, not that the morphism is in odd degree. Indeed, all computations are for morphisms in degree zero.} to demonstrate the relationship. In particular, the morphism $c_r$ is represented in these degrees by (a scalar multiple) of the transpose of the vector spanning the Hom-space in \cref{KijKwrHoms} after substituting $i=p-1$, $j=q-1$. It is a straightforward exercise to show that the morphism of matrix factorisations corresponding to $x^{\frac{p-1}{\ell}}$ is given by the $(\ell+1)\times \ell$ matrix
\begin{align*}
\begin{pmatrix}
1  & 0& \dots & 0 \\
0 & 1 &\dots & 0\\
\vdots  &\vdots &\ddots  &\vdots \\
0  & 0 & \dots & 1\\
0  & 0 & \dots & 0
\end{pmatrix}
\end{align*}
in odd degrees, and similarly the morphism $y^{\frac{q-1}{\ell}}$ is given in odd degrees by the matrix of the same form as above with a row of zeros in the first row and then the $\ell\times\ell$ identity matrix underneath. It is then immediate that the third relation holds. 
\end{proof}
\begin{rmk}\label{ChoicesRMK2}
Following on from \cref{ChoicesRMK}, it is here that it becomes evident that the resulting category is independent of the choice of representatives for $\obj{0}{i,j}$. In particular, making different choices would lead to rearranging the blocks in \cref{figLoopQuiver}, which is evidently the same quiver. 
For example, below is the quiver for an alternative collection of $I_{i,j}$ for $\ell=2$. Whilst the modules in the collection change, the resulting quiver-with-relations is the same, and can be seen to be a rearrangement of that of \cref{figLoopQuiver} in the case of $\ell=2$.
	\begin{figure}[H]
	\centering
	\begin{tikzpicture}[blob/.style={circle, draw=black, fill=black, inner sep=0, minimum size=\blobsize}, arrow/.style={->, shorten >=6pt, shorten <=6pt}, scale=0.9]
	\def\blobsize{1.2mm}

	\draw (2, 0) node[blob]{};
	\draw (3.25, 0) node{\ $\cdots$};
	\draw (4.5, 0) node[blob]{};
	\draw (6, 0) node[blob]{};
	
	\draw[arrow] (2, 0) -- (3, 0)  node[midway,below]{$y$};
	\draw[arrow] (3.5, 0) -- (4.5, 0) node[midway,below]{$y$};
	\draw[arrow] (4.5, 0) -- (6, 0) node[midway,below]{$b$};

	\draw[arrow] (2, 0) -- (2, 1)  node[midway,left]{$x$};
	\draw[arrow] (4.5, 0) -- (4.5, 1)  node[midway,right]{$x$};
	
	\draw (6,-1) node{$\obj{x}{i}[3]$};
	
	\draw (6.25,5.25) node{$\obj{w_r}{}[3]$};
		\draw [
	thick,
	decoration={
		brace,
		raise=0.5cm
	},
	decorate
	] (1.8,5.25) -- (4.7,5.25)
	node [pos=0.5,anchor=south,yshift=0.55cm] {$\frac{q-1}{\ell}$};
	
	\draw [
	thick,
	decoration={
		brace,
		raise=0.5cm
	},
	decorate
	] (6.25,3.75) -- (6.25,-.25)
	node [pos=0.5,anchor=west,xshift=0.55cm] {$\frac{p-1}{\ell}$};
	
	\begin{scope}[yshift=1cm]
	\draw (2, 0) node[blob]{};
	\draw (3.25, 0) node{\ $\cdots$};
	\draw (4.5, 0) node[blob]{};
	\draw (6, 0) node[blob]{};

	\draw[arrow] (2, 0) -- (3, 0) node[midway,above]{$y$};
	\draw[arrow] (3.5, 0) -- (4.5, 0) node[midway,above]{$y$};
	\draw[arrow] (4.5, 0) -- (6, 0) node[midway,below]{$b$};

	\draw[arrow] (2, 0) -- (2, 1)  node[midway,right]{$x$};
	\draw[arrow] (4.5, 0) -- (4.5, 1)  node[midway,right]{$x$};
	
	\end{scope}
	
	\begin{scope}[yshift=3.5cm]
	
	\draw (2, 0) node[blob]{};
	\draw (3.25, 0) node{\ $\cdots$};
	\draw (4.5, 0) node[blob]{};
	\draw (6, 0) node[blob]{};

	\draw[arrow] (2, 0) -- (3, 0) node[midway,above]{$y$};
	\draw[arrow] (3.5, 0) -- (4.5, 0) node[midway,above]{$y$};
	\draw[arrow] (4.5, 0) -- (6, 0) node[midway,below]{$b$};

	\draw[arrow] (2, 0) -- (2, 1.5)  node[midway,left]{$a$};
	\draw[arrow] (4.5,0) -- ({4.5+2*cos(360/14)},{2*sin(360/14)} ) node[pos=0.6,above]{$c_\ell$};
	\draw[arrow] (4.5,0) -- ({4.5+2*cos(5*360/28)},{2*sin(5*360/28)} ) node[pos=0.6,right]{$c_1$};
	\draw[arrow] (4.5, 0) -- (4.5, 1.5) node[midway,left]{$a$};
	\end{scope}
	
	\begin{scope}[yshift=5cm]
	\draw (2, 0) node[blob]{};
	\draw (3.25, 0) node{\ $\cdots$};
	\draw (4.5, 0) node[blob]{};

	\end{scope}
	\begin{scope}[yshift=3.5cm, xshift=4.5cm]
	\foreach \a in {2,5}{
		\draw (\a*360/28: 2) node[blob]{};
	}
	\foreach \a in {1,2,3}{
		\draw[fill] (\a*135/14+360/14: 2) circle[radius=1pt];
	}
	
	\end{scope}
	
	\draw (2, 2.37) node{$\vdots$};
	\draw (4.5, 2.37) node{$\vdots$};
	\draw (6, 2.37) node{$\vdots$};
	
	\begin{scope}[yshift=2.5cm]
	\draw[arrow] (2, 0) -- (2, 1) node[midway,left]{$x$};
	\draw[arrow] (4.5, 0) -- (4.5, 1) node[midway,right]{$x$};
	\end{scope}
	
	\draw (3.25, 2.37) node{\ $\iddots$};
	
	\draw[line width=0.5mm, opacity=0.2] (1.5, -0.5) rectangle (5, 4);
	\draw[line width=0.5mm, opacity=0.2] (1.5, 4.5) rectangle (4.9, 5.5);
	\draw[line width=0.5mm, opacity=0.2] (5.5, -0.5) rectangle (6.5, 4);
		\draw (.75,5) node{$\obj{y}{j}[3]$};
	\begin{scope}[yshift=-4.5cm]
	\draw (2, 0) node[blob]{};
	\draw (3.25, 0) node{\ $\cdots$};
	\draw (4.5, 0) node[blob]{};

	\draw[arrow] (2, 0) -- (3, 0)  node[midway,below]{$y$};
	\draw[arrow] (3.5, 0) -- (4.5, 0) node[midway,below]{$y$};
	\draw[arrow] (4.5, 0) -- (2, 4.5) node[midway,above]{$y$};

	\draw[arrow] (2, 0) -- (2, 1)  node[midway,right]{$x$};
	\draw[arrow] (4.5, 0) -- (4.5, 1)  node[midway,right]{$x$};
	
	\begin{scope}[yshift=1cm]
	\draw (2, 0) node[blob]{};
	\draw (3.25, 0) node{\ $\cdots$};
	\draw (4.5, 0) node[blob]{};

	\draw[arrow] (2, 0) -- (3, 0) node[midway,above]{$y$};
	\draw[arrow] (3.5, 0) -- (4.5, 0) node[midway,above]{$y$};
	\draw[arrow] (4.5, 0) -- (2, 4.5) node[pos=.35,above]{$y$};

	\draw[arrow] (2, 0) -- (2, 1)  node[midway,right]{$x$};
	\draw[arrow] (4.5, 0) -- (4.5, 1)  node[midway,right]{$x$};
	\end{scope}
	
	\begin{scope}[yshift=3.5cm]
	
	\draw (2, 0) node[blob]{};
	\draw (3.25, 0) node{\ $\cdots$};
	\draw (4.5, 0) node[blob]{};
	
	\draw[arrow] (2, 0) -- (3, 0) node[midway,above]{$y$};
	\draw[arrow] (3.5, 0) -- (4.5, 0) node[midway,above]{$y$};
	\draw[arrow] (4.5, 0) -- (2, 4.5) node[midway,above]{$y$};
	
	\draw[arrow] (4.5, 0) -- (4.5, 1) node[pos=0.35,right]{$x$};
	\draw[arrow] (2, 0) -- (2, 1) node[pos=0.35,left]{$x$};
	
	\end{scope}
	
	\draw (2, 2.37) node{$\vdots$};
	\draw (4.5, 2.37) node{$\vdots$};

	\begin{scope}[yshift=2.5cm]
	\draw[arrow] (2, 0) -- (2, 1) node[midway,right]{$x$};
	\draw[arrow] (4.5, 0) -- (4.5, 1) node[midway,right]{$x$};
	\end{scope}
	
	\draw (3.25, 2.37) node{\ $\iddots$};
	
	\draw[line width=0.5mm, opacity=0.2] (1.5, -0.5) rectangle (5, 4);
	
	\end{scope}
	
			\begin{scope}[yshift=-50,xshift=50]
	\draw (0,0) node{\parbox{300pt}{\small \ \textbf{Relations:} \begin{enumerate}[(i)]\item $xy=yx$\item $ay=bx=0$ \item $c_r(x^{\frac{p-1}{\ell}}-e^{\frac{\pi i}{\ell}}\eta^ry^{\frac{q-1}{\ell}})=0$\end{enumerate}}};
	\end{scope}
	\end{tikzpicture}
			\caption{Quiver corresponding to an alternative collection of $I_{i,j}$ for $\ell=2$. Whilst we have considered different objects, it is clear that the resulting category is the same, since the quiver is manifestly a rearrangement of \cref{LoopEndAlgebra} for $\ell=2$.  \label{AlternativeQuiver}}
\end{figure}
\end{rmk}
\begin{ex}
Consider the polynomial $x^py+y^px$ with $\ell=p-1$. Then, the corresponding quiver is given by
\begin{figure}[H]

	\begin{tikzpicture}[blob/.style={circle, draw=black, fill=black, inner sep=0, minimum size=\blobsize}, arrow/.style={->, shorten >=6pt, shorten <=6pt}]
	\def\blobsize{1.2mm}

	\draw (-3, 0) node[blob]{};
	\draw (-1.5, 0) node[blob]{};
	\draw (0, 0) node{$\cdots$};
	\draw (1.5, 0) node[blob]{};
	\draw (3, 0) node[blob]{};
	\draw ({3+1.5*cos(-60)}, {1.5*sin(-60)}) node[blob]{};
	\draw ({3+1.5*cos(60)}, {1.5*sin(60)}) node[blob]{};
	
	\begin{scope}[xshift=3cm]
			\foreach \a in {1,4}{
		\draw (-60+\a*120/5: 1.5) node[blob]{};
	}
	\foreach \a in {1,2,3}{
		\draw[fill] (-36+\a*18: 1.5) circle[radius=1pt];
	}
	\end{scope}

	\draw[arrow,yshift=0.5ex] (-3,0) -- (-1.5,0) node[midway,above]{$x$};
	\draw[arrow,yshift=-0.5ex] (-3,0) -- (-1.5,0) node[midway,below]{$y$};
	
	\draw[arrow,yshift=0.5ex, shorten >=2ex] (-1.5,0) -- (0,0) node[midway,above]{$x$};
	\draw[arrow,yshift=-0.5ex, shorten >=2ex] (-1.5,0) -- (0,0) node[midway,below]{$y$};
	
	\draw[arrow,yshift=0.5ex, shorten <=1.5ex] (0,0) -- (1.5,0) node[midway,above]{$x$};
	\draw[arrow,yshift=-0.5ex, shorten <=1.5ex] (0,0) -- (1.5,0) node[midway,below]{$y$};
	
	\draw[arrow,yshift=0.5ex] (1.5,0) -- (3,0) node[midway,above]{$x$};
	\draw[arrow,yshift=-0.5ex] (1.5,0) -- (3,0) node[midway,below]{$y$};
	
	\draw[arrow] (3,0) -- ({3+1.5*cos(60)}, {1.5*sin(60)})  node[midway,left]{$a$};
	\draw[arrow] (3,0) -- ({3+1.5*cos(-60)}, {1.5*sin(-60)})  node[midway,left]{$b  $};
	
	\draw[arrow] (3,0) -- ({3+1.5*cos(36)},{1.5*sin(36)}) node[pos=0.6,below]{$c_1$};
	\draw[arrow] (3,0) -- ({3+1.5*cos(-36)},{1.5*sin(-36)}) node[pos=0.6,above]{$c_{p-1}$};

\begin{scope}[xshift=11cm]
\draw (0,0) node{\parbox{300pt}{\small \ \textbf{Relations:} \begin{enumerate}[(i)]\item $xy=yx$\item $ay=bx=0$ \item $c_r(x-e^{\frac{\pi i}{p-1}}\eta^ry)=0$\end{enumerate}}};
\end{scope}
\end{tikzpicture}
\caption{Quiver corresponding to $x^py+y^px$.}
\end{figure}

This corresponds to a tilting module of the $\Z$-graded ring $\C[x,y]/(x^py+y^px)$ with $|x|=|y|=1$, and is a special case of \cite[Theorem 2.1]{BIY}.
\end{ex}

\subsection{Generation}
\label{BGeneration}

Now that we have characterised the full subcategory $\mathcal{B}\subseteq \mathrm{HMF}(\C^2,\Gamma,\w)$, we aim to show that it generates. To this end, we utilise the following result of Polishchuk-Vaintrob: 
\begin{lem}[{\cite[Proposition 2.3.1]{PolishVainCFT},} cf. {\cite[Corollary 5.3]{Dyckerhoff}}]
	\label{PVGeneration}
	The category $\mathrm{HMF}(\C^2, \Gamma, \w)$ is split-generated by the $L$-grading shifts of the stabilisation of the module $R/(x, y)$.\hfill$\qed$
\end{lem}
It should be emphasised that \cref{PVGeneration} is a statement about \emph{split} generation of the homotopy category $\mathrm{HMF}(\C^2, \Gamma, \w)$, rather than \emph{generation} of the dg-category $\mathrm{mf}(\C^2, \Gamma, \w)$.  Nevertheless, establishing that $\mathcal{B}$ split generates $\mathrm{HMF}(\C^2, \Gamma, \w)$ is the major step towards the required generation statement. 
\begin{prop}
	\label{BSplitGenerates}
	The category $\mathrm{HMF}(\C^2, \Gamma, \w)$ is split-generated by $\mathcal{B}$.
\end{prop}
Before giving the proof, we record the required generation result as the following corollary.

\begin{cor}
	\label{BGenerates}
	The functor
	\begin{align*}
	 \Tw \mathcal{B} \rightarrow \mathrm{mf}(\C^2, \Gamma, \w)
	\end{align*}
	is a quasi-equivalence.
\end{cor}
\begin{proof}
By \cref{BSplitGenerates}, the functor
\begin{align*}
\Tw^\pi\mathcal{B}\rightarrow \mathrm{HMF}(\C^2, \Gamma, \w)
\end{align*}
is a quasi equivalence. Since $\mathcal{B}$ is intrinsically formal, there is also a quasi-equivalence 
\begin{align*}
\Tw^\pi\mathcal{B}\rightarrow \mathrm{mf}(\C^2, \Gamma, \w).
\end{align*}
To show that $\mathcal{B}$ \emph{generates} $\mathrm{mf}(\C^2, \Gamma, \w)$, observe that, since $\mathcal{B}$ is a full exceptional collection in $\Tw\mathcal{B}$, the latter is already idempotent complete, as explained in \cite[Remark 5.14]{SeidelBook}. Therefore, taking the idempotent completion of $\Tw\mathcal{B}$ does nothing, and we get the result.
\end{proof}

\begin{proof}[Proof of \cref{BSplitGenerates}]
	We focus on the case of $\ell>1$ in order to differentiate it from \cite[Proposition 2.14]{HabermannSmith}, only explaining the modifications needed to incorporate the change of grading. Namely, let 
	\begin{align*}
	V=\{\obj{x}{i},\obj{y}{j},\obj{w_r}{},\obj{0}{i,j}\}
	\end{align*}
	be the objects of $\mathcal{B}$. We then show that $R(l)/(x,y)\in\langle V\rangle$ for any $l\in L$ and appeal to \cref{PVGeneration} to show that $V$ split-generates $\mathrm{HMF}(\C^2,\Gamma,\w)$. \\	
	
Since $[2]$ is equivalent to $(\vec{c})$, we must only show that $R(l)/(x,y)\in\langle V\rangle$ for $l\in L/\Z\vec{c}\simeq\Z/\left(\frac{pq-1}{\ell}\right)$. The argument to build the modules 
	\begin{gather*}
\{R(a\x+b\y)/(x,y)\}_{\{p-\frac{p-1}{\ell}+1\leq a\leq p,\ 2\leq b\leq q\}}\\
\{R(a\x+\y)/(x,y)\}_{\{p-\frac{p-1}{\ell}\leq a\leq p-1\}}\\
\{R((p-\frac{p-1}{\ell})\x+b\y)/(x,y)\}_{\{1\leq b\leq \frac{q-1}{\ell}+1\}}
\end{gather*}
 with the exception of $R((p-\frac{p-1}{\ell})\x+\y)/(x,y)$ in the second collection of modules and $R(\vec{c})/(x,y)$ in the third carry over from \cite[Proposition 2.14]{HabermannSmith} essentially unchanged. Note that, along the way, we built the modules 
 \begin{align}\label{PreviouslyBuiltModules}
 \begin{split}
 R(a\x)/(x), & \qquad\text{for } 1\leq a\leq p-1 \\
 R(b\y)/(y) & \qquad\text{for } 1\leq b\leq q-1,
 \end{split}
 \end{align} 
 and we will repeatedly appeal to this fact. \\

To build $R((p-\frac{p-1}{\ell})\x+\y)/(x,y)$, we begin by constructing $\obj{w_r}{}(l\x)$ for $l=0,\dots, (\ell-1)\frac{p-1}{\ell}$ and any $r=1,\dots,\ell$. Firstly, we have that 
\begin{align*}
R(\x)/(w_r)\simeq\mathrm{Cone}\Big(R((1-\frac{p-1}{\ell})\x)/(x)\xrightarrow{w_r}R(\x)/(xw_r)\Big).
\end{align*}
The domain of this morphism is $\obj{x}{p-\frac{p-1}{\ell}}$, whilst the codomain can be constructed as an extension of $R(\x)/(x)$ by $R/(w_r)$. From this, we build $R(2\x)/(xw_r)$ as the extension
\begin{align*}
0\rightarrow R(\x)/(w_r)\xrightarrow{x}R(2\x)/(xw_r))\rightarrow R(2\x)/(x)\rightarrow 0,
\end{align*}
where the module $R(2\x)/(x)$ was constructed in \eqref{PreviouslyBuiltModules}. From this, we get 
\begin{align*}
R(2\x)/(w_r)\simeq\mathrm{Cone}\big(R((2-\frac{p-1}{\ell})\x)/(x)\xrightarrow{w_r}R(2\x)/(xw_r)\big),
\end{align*}
where the domain is in \eqref{PreviouslyBuiltModules}, and the codomain was constructed immediately above. We then proceed inductively, alternatingly constructing $R(l\x)/(xw_r)$ and then $\obj{w_r}{}(l\x)$ for $l=0,\dots, (\ell-1)\frac{p-1}{\ell}$ (appealing to \eqref{PreviouslyBuiltModules} once $l>\frac{p-1}{\ell}$). Note that this argument is valid for any $r=1,\dots, \ell$.\\
With these modules constructed, we then iteratively construct the following modules. 
\begin{gather*}
0\rightarrow R/(w_1)\xrightarrow{w_2}R(\frac{p-1}{\ell}\x)/(w_1w_2)\rightarrow R(\frac{p-1}{\ell}\x)/(w_2)\rightarrow 0\\
0\rightarrow R(\frac{p-1}{\ell}\x)/(w_1w_2)\xrightarrow{w_3}R(2\frac{p-1}{\ell}\x)/(w_1w_2w_3)\rightarrow R(2\frac{p-1}{\ell}\x)/(w_3)\rightarrow 0\\
\vdots\\
\begin{aligned}
0\rightarrow R((\ell-1)\frac{p-1}{\ell}\x)/(w_1w_2\dots w_{\ell-1})\xrightarrow{w_{\ell}}R&((\ell-1)\frac{p-1}{\ell}\x)/(w_1w_2w_3\dots w_{\ell})\\
&\rightarrow R(((\ell-1)\frac{p-1}{\ell})\x)/(w_{\ell})\rightarrow 0
\end{aligned}
\end{gather*}

We then observe that $R((\ell-1)\frac{p-1}{\ell}\x)/(w_1w_2w_3\dots w_{\ell})[1]=R((p-\frac{p-1}{\ell})\x+\y)/(xy)$ (cf. \eqref{GradingShiftFunctor}), and so we can construct $R((p-\frac{p-1}{\ell})\x+\y)/(x,y)$ as the cone of the morphism
	\begin{align*}
\begin{tikzcd}[row sep=0.5ex, ampersand replacement = \&]
R((\ell-1)\frac{p-1}{\ell})\x+\y)/(y)\& \\
\bigoplus \ar[r, shorten >=3ex, shorten <=2ex,"(x\ y)"]\& R((p-\frac{p-1}{\ell})\x+\y)/(xy).\\
R((p-\frac{p-1}{\ell})\x)/(x)\& 
\end{tikzcd}
\end{align*} 
Note that $R((\ell-1)\frac{p-1}{\ell})\x+\y)/(y)=R(q-\frac{q-1}{\ell}y)/(y)$, and so both modules in the domain are in \eqref{PreviouslyBuiltModules}. \\

Finally, all that is left to do is construct the module $R(\vec{c})/(x,y)$, which we again do iteratively. 
\begin{gather*}
0\rightarrow R(\vec{c}-(\frac{q-1}{\ell}+1)\y)/(x)\xrightarrow{y} R(\vec{c}-\frac{q-1}{\ell}\y)/(x)\rightarrow R(\vec{c}-\frac{q-1}{\ell}\y)/(x,y)\rightarrow 0\\
0\rightarrow R(\vec{c}-\frac{q-1}{\ell}\y)/(x)\xrightarrow{y} R(\vec{c}-(\frac{q-1}{\ell}-1)\y)/(x)\rightarrow R(\vec{c}-(\frac{q-1}{\ell}-1)\y)/(x,y)\rightarrow 0\\
\vdots\\
0\rightarrow R(\vec{c}-2\y)/(x)\xrightarrow{y} R(\vec{c}-\y)/(x)\rightarrow R(\vec{c}-\y)/(x,y)\rightarrow 0.
\end{gather*}
The module $R(\vec{c}-(\frac{q-1}{\ell}+1)\y)/(x)=R((\ell-1)\frac{p-1}{\ell}\x)/(x)$ is from \eqref{PreviouslyBuiltModules}, and the modules on the right of the above short exact sequences are graded stabilisations of the origin which have been previously constructed. Putting this all together, we have
\begin{align*}
R(\vec{c})/(x,y)\simeq\Cone(R(\vec{c}-\y)/(x)\xrightarrow{y}& R(\vec{c})/(x))
\end{align*}
where the term on the domain was constructed above and the codomain is $\obj{x}{p-1}[2]$.
\end{proof}
We deduce the following corollary, whose proof follows from \cref{BGenerates} and \cref{LoopEndAlgebra} in the same way as in the proof of \cite[Theorem 2.19]{HabermannSmith}:
\begin{cor}[\cref{TiltingCor}, undeformed loop polynomial case]\label{LoopTiltingCor} The object 
\begin{align*}
\mathcal{E}:=\left(\bigoplus_{\substack{i=p-\frac{p-1}{\ell},\dots, p-1\\
		j=1,\dots, q-1}}\obj{0}{i,j}\right)\oplus&\left(\bigoplus_{\substack{i=p-\frac{p-1}{\ell},\dots, p-1}}\obj{x}{i}[3]\right)\\
	\oplus	&\left(\bigoplus_{\substack{j=q-\frac{q-1}{\ell},\dots, q-1}}\obj{y}{j}[3]\right)\oplus\left(\bigoplus_{\substack{r=1,\dots, \ell}}\obj{w_r}{}[3]\right)
\end{align*}
is a tilting object for $\mathrm{mf}(\C^2,\Gamma,\w)$. 
\end{cor}

\section{Loop A-model}
\label{LoopAModel} 
\subsection{The crepant resolution of the $A_{\ell-1}$ singularity}\label{CrepResolutionSection}
We begin this section with a recounting of the classically understood crepant\footnote{Recall that a resolution $\pi:\widetilde{X}\rightarrow X$ is \emph{crepant} if $\pi^*K_X= K_{\widetilde{X}}$.} resolution of du Val singularities of type $A$ -- see, for example \cite{Reid}. This subsection will be applicable to all invertible polynomials in two variables. \\

Recall that the $A_{\ell-1}$ singularity is defined as $\C^2/\mu_{\ell}=X=\Spec\C[x,y]^{\mu_\ell}\simeq\Spec\C[u,v,w]/(uv-w^\ell)$, where $\mu_\ell$ acts by $\xi\cdot(\xt,\yt)=(\xi \xt,\xi^{-1}\yt)$. The crepant resolution is
\begin{align*}
\pi:\widetilde{X}\rightarrow X,
\end{align*}
where 
\begin{align*}
\widetilde{X}=\bigcup_{i=1}^{\ell}\widetilde{X}_i,
\end{align*}
and each $\widetilde{X}_i\simeq \C^2_{(\lambda_i,\mu_i)}$. Here, the coordinates $\lambda_i,\mu_i$ are related to the coordinates on $X$ by $\lambda_i=u/w^{i-1}$, $\mu_i=w^i/u$ where this makes sense. Correspondingly, the transition functions are given by
\begin{align*}
\widetilde{X}_i\setminus(\mu_i=0)&\xrightarrow{\sim} \widetilde{X}_{i+1}\setminus(\lambda_{i+1}=0)\\
(\lambda_i,\mu_i)&\mapsto (\mu_i^{-1},\mu_i^2\lambda_i).
\end{align*}

On each chart, the resolution is given as 
\begin{align*}
\pi|_{\widetilde{X}_i}:\C^2_{(\lambda_i,\mu_i)}&\rightarrow X\\
(\lambda_i,\mu_i)&\mapsto (\lambda_i^i\mu_i^{i-1},\lambda_i^{\ell-i}\mu_i^{\ell+1-i},\lambda_i\mu_i).
\end{align*}
This resolution is constructed by iterated blowups of $X$. An important fact is that the blowup process yields an embedding of $\widetilde{X}$ into $\C^3\times\P^2\times\dots\times\P^2$, where there are ${\lfloor\frac{\ell}{2}\rfloor}$ factors of $\P^2$. We define the symplectic form $\omega_{\widetilde{X}}$ to be the pullback of the product of the standard symplectic form on $\C^3$ and the Fubini--Study form on the $\P^2$ factors. We also drop $\widetilde{X}$ from the notation and refer to the symplectic form simply as $\omega$, since we believe that no confusion can arise. This form is simple to write down in a given patch, although has quite a few terms. As an example, in the patch $\widetilde{X}_1$, it is given by 
\begin{align*}
\phi^*_1\omega= \frac{i}{2}\Big((1+|\mu_1|^2+|\mu_1|^4)\diff\lambda_1&\wedge\diff\bar{\lambda}_1+(\lambda_1\bar{\mu}_1+2|\mu_1|^2\lambda_1\bar{\mu}_1)\diff\mu_1\wedge\diff\bar{\lambda}_1\\&+(\mu_1\bar{\lambda}_1+2|\mu_1|^2\mu_1\bar{\lambda}_1)\diff\lambda_1\wedge\diff\bar{\mu}_1
+(4|\lambda_1|^2|\mu_1|^2+|\lambda_1|^2)\diff\mu_1\wedge\diff\bar{\mu}_1\\
&\qquad\qquad+\frac{(1+4|\mu_1|^2+|\mu_1|^4)\diff\mu_1\wedge\diff\bar{\mu}_1}{(1+|\mu_1|^2+|\mu_1|^4)^2}\Big),
\end{align*}
where $\phi_1:\widetilde{X}_1\rightarrow \widetilde{X}$ is the inclusion of the chart $\widetilde{X}_1\simeq \C^2_{(\lambda_1,\mu_1)}$. It is extremely important that, since the spheres in the exceptional locus are holomorphic, $\omega$ is \emph{not exact}.\\
In terms of writing down $[\omega]\in H^2(\widetilde{X})$, observe that in the case $\ell=2$, then $\widetilde{X}=\mathrm{tot}\ \mathcal{O}_{\P^1}(-2)=T^*\P^1$, and $[\omega]=2[\pi^*\omega_{\mathrm{FS}}]=-\pi[\mathrm{PD}(C)]$, where $C$ is the exceptional sphere, and $\omega_{\mathrm{FS}}$ is the Fubini--Study form on $\P^1$ (see, for example, \cite{Ritter}). Symplectically, the crepant resolution of the $A_{\ell-1}$ is the completion of $\ell-1$ copies of $\mathbb{D}T^*\P^1$ plumbed together according to the $A_{\ell-1}$ Dynkin diagram, where the symplectic form on each factor is as above. Therefore, each exceptional sphere has symplectic area $2\pi$. Putting this all together, the symplectic form we work with on the total space has cohomology class
\begin{align}\label{EvenSympForm}
[\omega]=-\pi\frac{\ell^2}{4}\mathrm{PD}(C_{\frac{\ell}{2}})-\sum_{r=1}^{\frac{\ell}{2}-1}r\pi(\ell-r)(\mathrm{PD}(C_r)+\mathrm{PD}(C_{\ell-r}))
\end{align}
for $\ell$ even, and 
\begin{align}\label{OddSympForm}
[\omega]=-\sum_{r=1}^{\frac{\ell-1}{2}}r\pi(\ell-r)(\mathrm{PD}(C_r)+\mathrm{PD}(C_{\ell-r}))
\end{align}
for $\ell$ odd, where $C_r$ are the exceptional spheres. 
It goes back to work of Brieskorn, \cite{Brieskorn}, that the Milnor fibre and minimal resolution of an ADE singularity are diffeomorphic (in fact, work of Ohta--Ono, \cite{OhtaOno}, show that there is a unique diffeomorphism type for symplectic fillings of the link of a simple singularity). On the other hand, they are very different as symplectic manifolds -- the standard symplectic form on the Milnor fibre is exact. In this work, we insist on equipping the total space $\widetilde{X}$ with the symplectic structure as above. 
\subsection{Equivariant Morsifications}\label{EqMorsificationSection}
In the introduction, we stated that we are associating a Fukaya--Seidel category to a specific equivariant Morsification of the superpotential. In this subsection, we justify this definition. \\
The Fukaya--Seidel category for a hypersurface singularity, as defined in \cite{SeidelBook}, is independent of all choices, up to derived equivalence. In particular, it is independent of the choice of Morsification. We would like to be able to make the same statement about equivariant Morsifications; however, it quickly becomes clear that, whilst Morsifications abound, equivariant Morsifications are quite special. In particular, it is easy to cook up examples of superpotentials which have no equivariant Morsifications. The following seems to be the simplest such example. 
\begin{ex}\label{NoEqMorseEx}
	Let $\wt=x^4+y^4:\C^2\rightarrow \C$ and $G=\langle \alpha,\beta|\ \alpha^2=\beta^2=(\alpha\beta)^2=-1\rangle\subseteq\mathrm{SL}(2,\C)$. Concretely $G$, acts on $\C^2$ via
	\begin{align*}
	\alpha=\begin{pmatrix}
	\varepsilon & 0 \\
	0 & \varepsilon^{-1}
	\end{pmatrix},\ \beta=\begin{pmatrix}
	0 & 1 \\
	-1 & 0
	\end{pmatrix},
	\end{align*}
	where $\varepsilon$ is a fixed $4$-th root of unity. The invariant polynomials are $u=x^5y-xy^5$, $v=x^4+y^4$ and $w=x^2y^2$, and adding any combination of $u,v$ or $w$ does not yield a polynomial with only non-degenerate critical points. Therefore, $\wt$ cannot be equivariantly Morsified. 
\end{ex}
On the other hand, if equivariant Morsifications exist, it seems reasonable to expect in general that the resulting category is independent of this choice by a deformation-invariance argument, as in the non-equivariant case. We avoid questions of this nature in general, and instead demonstrate this only in the case at hand. \\
Restricting to the case of $\ell>1$, so that we are not in the maximally graded case, one must add monomials in the Jacobian which are divided by $x^\ell$, $y^\ell$ or $xy$ in order to equivariantly Morsify $\wt$. For $\ell>2$, it is immediate that adding any linear combination of $x^\ell$ and $y^\ell$ or powers thereof does not Morsify the polynomial -- the critical point at the origin remains degenerate. On the other hand, the origin is a non-degenerate critical point for any deformation of $\wt$ where we add any linear combination of $x^\ell$, $y^\ell$, $xy$ or higher powers, \emph{provided the coefficient of $xy$ is non-trivial}. In particular, any equivariant Morsification of $\wt$ is a further deformation of the resonant Morsification $\wt_{\varepsilon}$. For $\ell=2$, one can now also equivariantly Morsify by adding $\varepsilon_1x^2+\varepsilon_2y^2$ for $\varepsilon_i\neq0$, although this can similarly be deformed through polynomials with non-degenerate critical points to the resonant Morsification. Therefore, the space of equivariant Morsifications is path-connected, and the resulting categories will be derived equivalent by a deformation invariance argument, as in the non-equivariant case. We therefore feel justified in defining $\mathcal{FS}(\wttilde_\varepsilon)$ to be \emph{the} Fukaya--Seidel category associated $\wttilde$. \\
The second subtle point of defining the orbifold Fukaya--Seidel category of an equivariant Morsification is the order of pulling back and deforming the polynomial. In this paper, we will insist on equivariantly Morsifying and then pulling back. The reason is that, if one were to pull back $\widecheck{\overline{\w}}$ (without Morsifying), then the exceptional locus would be critical. It therefore doesn't seem to make sense to `Morsify', since this would imply that this critical locus splits into a Milnor number's worth of critical points. Even if one were to ignore this and deform the superpotential, only requiring that it becomes non-degenerate on each chart and that this patches together to something globally defined, this is equivalent to pulling back a superpotential on $X$ which comes from an equivariant Morsification. 
\subsection{Definition of the Fukaya--Seidel category for a non-exact total space}\label{NonExactFSSection}
Whilst the A--model strategy follows that of \cite{HabermannSmith}, there are technical subtleties in this setting due to the non-exactness. In particular, the exactness of the total space is key in the construction of the Fukaya--Seidel category in \cite{SeidelBook} by ruling out sphere and disc bubbling. Since our total space is non-exact, we follow the strategy of \cite[Section 3]{AOK}, who also allow for non-exact total spaces and Lagrangians by demanding the following two conditions:
\begin{enumerate}[(i)]
	\item The smooth fibre $\Sigma$ is exact,
	\item The homotopy group $\pi_2(\Sigma)$ and the relative homotopy group $\pi_2(\Sigma,V_i)$, for any (potentially non-exact) vanishing cycle $V_i$, vanish.
\end{enumerate}
These conditions are both fulfilled in our case. The first requirement follows from the fact that the fibre is a punctured surface, and so every two-form is exact. The vanishing of $\pi_2(\Sigma)$ follows from the uniformisation theorem for Riemann surfaces, whilst nontrivial $\pi_2(\Sigma,V_i)$ implies that $V_i$ is trivial in homology. In our construction, we only consider Lagrangians which are nontrivial in homology.\\

Another subtlety about working with a non-exact total space is that we must work over the Novikov field, defined as 
\begin{align*}
\Lambda_\C:=\{\sum_{i\geq 1}a_iT^{\lambda_i}|\ a_i\in\C,\ \lambda_i\in\R,\ \lim_{i\rightarrow \infty}\lambda_i\rightarrow\infty \}.
\end{align*}
In general, this is crucial for two reasons. The first is that there can be infinitely many discs contributing to a given $A_\infty$ operation, and the sum is not guaranteed to converge. As a consequence of Gromov compactness, the sum converges in the Novikov field. Fortunately, we will see that this is not an issue for us, since the only non-trivial operation is the product, and there are only finitely many discs contributing to this.\\
The second crucial consequence of working over the Novikov field is Hamiltonian invariance. Even when sums converge, it is possible to lose Hamiltonian invariance when specialising the Novikov parameter. The essential point is that $T^A=T^B$ if and only if $A=B$, but the multivalued-ness of the complex logarithm means that $z^A=z^B$ does not imply $A=B$ for any $z\in\C$ unless one were to restrict the values of $A$ and $B$ so as not to cross the branch locus. This is what we do, only checking that Hamiltonian invariance is preserved between Lagrangians which we are allowing as objects in our category, rather than any possible Lagrangian. \\

Another technical point which can't be handled as in the maximally graded case is that, in order for the construction of the Fukaya--Seidel category to be well-defined, we need there to be a complete K\"ahler metric on the total space, with respect to which the gradient flow of $\wttilde$ is bounded from below outside of a compact set. In the maximally graded setting, this follows from the polynomials being tame in the sense of \cite[Proposition 3.11]{Broughton}, however, this no longer applies directly, since the total space  in the case at hand is not $\C^n$. Nevertheless, we can appeal to the tameness of invertible polynomials to demonstrate that $|\nabla\wttilde|$ is bounded from below outside of a compact set. \\
Observe that, by construction, we have an isomorphism $(\C^2\setminus{\mathbb{D}^2})/\widecheck{\overline{\Gamma}}\simeq \widetilde{X}\setminus{U}$, where $\mathbb{D}^2$ is a small disc centred at the origin in $\C^2$, and $U$ is a tubular neighbourhood of the exceptional locus in $\widetilde{X}$. Since $\wt:\C^2\rightarrow \C$ is tame, it remains so when it descends to $(\C^2\setminus{\mathbb{D}^2})/\widecheck{\overline{\Gamma}}$ since the action of $\widecheck{\overline{\Gamma}}$ is free away from the origin. By the above isomorphism, this shows that $|\nabla\wttilde|$ is bounded from below on $\widetilde{X}\setminus U$ with respect to the metric induced from the blow-up construction (which agrees with the metric on $(\C^2\setminus{\mathbb{D}^2})/\widecheck{\overline{\Gamma}}$ away from the exceptional locus). \\

To summarise, we have that $\wttilde_\varepsilon:\widetilde{X}\rightarrow \C$ defines a Lefschetz fibration. We take $\mathcal{A}_{\widecheck{\overline{\Gamma}}}$ to be a collection of Lagrangian branes in the smooth fibre (which is well-defined by the preceding paragraph), and check that Hamiltonian invariance is preserved between these branes for a given choice of Novikov parameter specialisation. We then \emph{define}, as in \cite{AOK}, the Fukaya--Seidel category (over $\C$) as $ \mathcal{FS}(\wttilde):=\Tw \mathcal{A}_{\widecheck{\overline{\Gamma}}}$. 
\subsection{A resonant Morsification}
In the maximally graded case studied in \cite[Section 3]{HabermannSmith}, a Morsification  $\wt_\varepsilon=\xt^p\yt+\yt^q\xt-\varepsilon \xt\yt$ with particularly nice properties was used. Such a Morsification was called \emph{resonant}, and allowed the symmetry present to be maintained and ultimately exploited. Whilst it was not considered in the maximally graded setting, the resonant Morsifications are in fact also \emph{equivariant} for any $\widecheck{\overline{\Gamma}}\simeq\mu_\ell$. Therefore, such a Morsification descends to a map $\widecheck{\overline{\w}}_\varepsilon:X\rightarrow \C$ given by 
\begin{align*}
\widecheck{\overline{\w}}_\varepsilon(u,v,w)=w(u^{\frac{p-1}{\ell}}+v^{\frac{q-1}{\ell}}-\varepsilon).
\end{align*}
Pulling this back to the chart $\widetilde{X}_i\simeq\C^2_{(\lambda_i,\mu_i)}$ of the crepant resolution $\widetilde{X}$ yields
\begin{align*}
\wttilde_{i,\varepsilon}(\lambda_i,\mu_i)=\lambda_i\mu_i(\lambda_i^{\frac{i(p-1)}{\ell}}\mu_i^{\frac{(i-1)(p-1)}{\ell}}+\lambda_i^{\frac{(\ell-i)(q-1)}{\ell}}\mu_i^{\frac{(\ell+1-i)(q-1)}{\ell}}-\varepsilon)
\end{align*}
and this patches together to give a globally defined map $\wttilde_{\varepsilon}:\widetilde{X}\rightarrow \C$. From now on, we will refer to the restrictions of $\wttilde_\varepsilon$ to the charts as $\wttilde_{i}$, with the reference to $\varepsilon$ left implicit. \\

Analogously to the maximally graded case, the critical points of $\wttilde_\varepsilon$ can be grouped into four types: 
\begin{enumerate}[(i)]
	\item \label{crit1} $\mu_1=0$, $\lambda_1^{\frac{p-1}{\ell}}=\varepsilon$
	\item \label{crit2} $\mu_\ell^{\frac{q-1}{\ell}}=\varepsilon$, $\lambda_\ell=0$
	\item \label{crit3} $\mu_i=\lambda_i=0$ for $i=1,\dots, \ell$
	\item \label{crit4} $\lambda_1^{\frac{p-1}{\ell}}=\frac{q-1}{pq-1}\varepsilon$, $\lambda_1^{\frac{(\ell-1)(q-1)}{\ell}}\mu_1^{q-1}=\frac{p-1}{pq-1}\varepsilon$.
\end{enumerate}
Note that the critical points of type (\ref{crit4}) can equivalently be described in any of the charts of $\widetilde{X}$; we have given them in the first for convenience. The critical points in the first three groups have critical value $0$, whilst those in the last have critical value 
\begin{align*}
\frac{-\varepsilon\mu_1\lambda_1}{pq-1}.
\end{align*}
As in our previous work, there is a clear symmetry of these critical points. Namely, let $(\lambda_{1,\text{crit}}^+,\mu_{1,\text{crit}}^+)$ be the unique positive real critical point of type (\ref{crit4}) in the chart $\widetilde{X}_1$, with corresponding critical value $c_{\text{crit}}$. Letting $\zeta$ and $\eta$ denote the roots of unity
\begin{align}
\label{RootsOfUnity}
\zeta = e^{2\pi i/(p-1)} \quad \text{and} \quad \eta = e^{2\pi i/(q-1)},
\end{align}
we see that there is a $\mu_{p-1}\times\mu_{q-1}$ action on the critical points of type (\ref{crit4}) given by
\begin{align*}
\{(\zeta^{m\ell}\lambda_{1,\text{crit}}^+, \zeta^{m(1-\ell)}\eta^n\mu_{1,\text{crit}}^+) : 0 \leq m \leq p-2\text{, } 0 \leq n \leq q-2\}.
\end{align*}
The critical value corresponding to $(\zeta^{m\ell}\lambda_{1,\text{crit}}^+, \zeta^{m(1-\ell)}\eta^n\mu_{1,\text{crit}}^+))$ is $\zeta^m\eta^nc_\mathrm{crit}$, so there are $\frac{\gcd(p-1, q-1)}{\ell}$ critical points in each of these critical fibres. In order to describe each critical point of type (iv) uniquely with respect to the symmetry and the reference critical point, we restrict to the subset
\begin{align}\label{CritPtSym}
\{(\zeta^{m\ell}\lambda_{1,\text{crit}}^+, \zeta^{m(1-\ell)}\eta^n\mu_{1,\text{crit}}^+) : 0 \leq m \leq \frac{p-1}{\ell}-1\text{, } 0 \leq n \leq q-2\}.
\end{align}
We do not require this to be a subgroup or have any other additional structure -- all we require is that this is a collection of elements in $\mu_{p-1}\times\mu_{q-1}$ such that each point in this collection corresponds to a critical point of type (iv). 
\begin{rmk}
Whilst it appears that we are making a choice in the above, choosing different elements of $\mu_{p-1}\times\mu_{q-1}$ leads to on-the-nose the same result, analogously to \cref{ChoicesRMK2} on the B--side. 
\end{rmk}

Our strategy for understanding the Fukaya--Seidel category of $\wttilde_\varepsilon$ is essentially modelled on the strategy of the maximally graded case. Namely, we fix our regular fibre $\Sigma$ to be $\wttilde^{-1}_1(-\delta)$ where $\delta$ is a positive real number much less than $\eps$ (in other words, we take $*=-\delta$). This reference fibre is equivalently described in any of the charts, and is equivalent to the quotient of the Milnor fibre of $\wt$ by $\widecheck{\overline{\Gamma}}$, as studied in \cite[Section 6.1]{HabStackyCurves}. For the critical points of types (\ref{crit1})--(\ref{crit3}) we follow the strategy employed in the maximally graded case and define the vanishing path given by the straight line segment from $-\delta$ to $0$.  For the critical point $(\zeta^{m\ell}\lambda_{1,\text{crit}}^+, \zeta^{m(1-\ell)}\eta^n\mu_{1,\text{crit}}^+)$, meanwhile, we define the \emph{preliminary vanishing path} $\gamma_{m,n}$ by following the circular arc $-\delta e^{i\theta}$ as $\theta$ increases from $0$ to
\begin{align*}
\theta_{m,n} \coloneqq 2\pi \left(\frac{m}{p-1}+\frac{n}{q-1}\right)
\end{align*}
and then following the radial straight line segment from $-\zeta^m\eta^n\delta$ to $\zeta^m\eta^nc_\mathrm{crit}$. The preliminary vanishing paths are then altered to the temporary, and then final, vanishing paths in the same way as the maximally graded case. 
\subsection{The zero-fibre and its smoothing}
The fibre of $\wttilde_\varepsilon$ over zero has $\ell+2$ components: in the charts $\widetilde{X}_1$ and $\widetilde{X}_\ell$, there are the line $\{\mu_1=0\}$ and $\{\lambda_\ell=0\}$, respectively. The complex lines $(0,\mu_i)$ and $(\lambda_{i+1},0)$ patch together to give $\ell-1$ projective lines in an $A_{\ell-1}$ configuration. Finally, there is the curve given in local charts by $\fact_i=\wttilde_i\lambda_i^{-1}\mu_i^{-1}$. This fibre is sketched in \cref{fig:loopsingmf}.
\begin{figure}[H]
	\centering
	\scalebox{.5}{
	\includegraphics[width=0.9\linewidth]{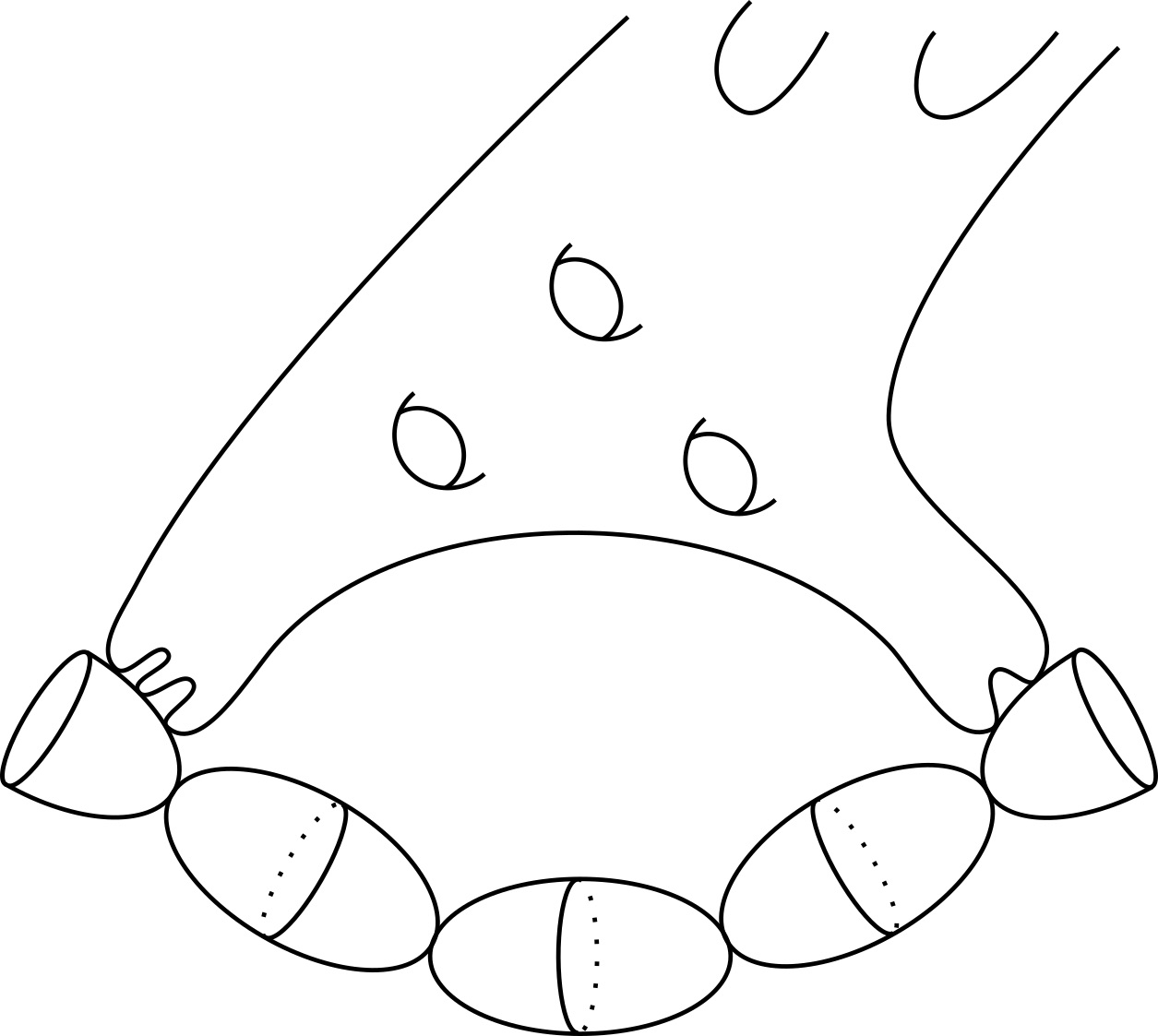}}
	\caption{Sketch of the fibre of $\wttilde_\varepsilon$ above the origin. This does not represent any specific polynomial, and is only meant to convey the general shape. See \cref{fig:loopsingularfibrespecificexample} for a specific example.}
	\label{fig:loopsingmf}
\end{figure}

Upon smoothing to $\wttilde_\varepsilon^{-1}(-\delta)$, each of the nodes is smoothed to a thin neck whose waist curve is the corresponding vanishing cycle in $\Sigma$. We denote these vanishing cycles by $\vc{\mu_1\fact_1}{m}$, $\vc{\lambda_\ell\fact_\ell}{n}$ and $\vc{\mu_i\lambda_i}{}$ for $m=0, \dots, \frac{p-1}{\ell}-1$, $n=0, \dots, \frac{q-1}{\ell}-1$ and $i=1,\dots,\ell$, corresponding to critical points\footnote{ $\sqrt[\frac{p-1}{\ell}]{\varepsilon}$, $\sqrt[\frac{q-1}{\ell}]{\varepsilon}$ are the real roots.} $(\zeta^{m\ell}\sqrt[\frac{p-1}{\ell}]{\varepsilon}, 0)$, $(0, \eta^{n\ell}\sqrt[\frac{q-1}{\ell}]{\varepsilon})$ and $(\lambda_i,\mu_i)=(0, 0)$ respectively\footnote{The symmetry of the critical points given in \eqref{CritPtSym} is in the chart $\widetilde{X}_1$. In order to to see the symmetry in another chart, one must compose it with the transition function. In particular, the symmetry for the critical points of type (\ref{crit2}) must be computed in the chart $\widetilde{X}_{\ell}$, where the action is as given here, rather than in \eqref{CritPtSym}.}.

\begin{rmk}
	The topology of the smooth fibre was computed in \cite[Section 6.1]{HabStackyCurves}. Namely, it is a curve of genus 
	\begin{align*}
	g(\Sigma)=\frac{1}{2\ell}(pq-1-\gcd(\ell(p-1),p+q-2))
	\end{align*}
	with 
	\begin{align*}
	2+\gcd(p-1,\frac{p+q-2}{\ell})
	\end{align*}
	boundary punctures. Note that this could also have be deduced in the same way as in \cite[Remark 3.1]{HabermannSmith}, although, since the $\vc{\mu_i\lambda_i}{}$ are homologous, the rank of $\HH_1(\Sigma;\Z)$ is $\frac{pq-1}{\ell}+1$, rather than $\frac{pq-1}{\ell}+\ell$.
\end{rmk}

As in our previous work, we argue that the monodromy of parallel transport around an arc of small enough radius $\delta$ is supported in the neck regions which emerge upon smoothing, and is the product of (partial) Dehn twists in these regions. After deleting these regions of $\Sigma$ to obtain $\Sigma'$, we may trivialise the fibration $\wttilde_\varepsilon$ over the disc of radius $\delta$ such that the smooth fibre is given by $\Sigma'$. See \cref{NeckRegionsRemoved} for a sketch of the complement of the neck regions. Concretely, $\Sigma'$ comprises: the $\lambda_1$-axis (so $\mu_1=0$) with small discs removed around the $\frac{p-1}{\ell}$-th roots of $\varepsilon$, as well at the origin, the $\mu_1$-hyperplane with the a disc around the origin removed, the hyperplanes $\mu_i=0$ and $\lambda_i=0$ in each chart $\widetilde{X}_i\simeq\C^2_{(\lambda_i,\mu_i)}$ for $i=2,\dots,\ell-1$, with a small disc around the origin and infinity in each removed\footnote{We are double counting here, since, for example, the $\lambda_2=0$ hyperplane with a disc around the origin removed is biholomorphic to the $\mu_3=0$ hyperplane in $\widetilde{X}_3$ with a disc around the origin removed. We have done this to be explicit about what is visible in each chart, rather than providing a minimal description of $\Sigma'$.}, the $\mu_\ell$-axis ($\lambda_\ell=0$) with small discs around the $\frac{q-1}{\ell}$-th roots of $\varepsilon$ and the origin in this chart removed, the $\mu_\ell=0$ complex line with the origin removed, and, lastly a $\frac{(p-1)(q-1)}{\ell}$-fold cover of the line $\{u+v=\varepsilon\}$ with small discs around $(\varepsilon,0)$ and $(0,\varepsilon)$ removed. Here, the covering map is given in the chart $\widetilde{X}_i$ by $(\lambda_i,\mu_i)\mapsto (\lambda_i^{\frac{i(p-1)}{\ell}}\mu_i^{\frac{(i-1)(p-1)}{\ell}}, \lambda_i^{\frac{(\ell-i)(q-1)}{\ell}}\mu_i^{\frac{(\ell+1-i)(q-1)}{\ell}})$, although it is most convenient to consider it in the first or last charts, where the fact that it is a $\frac{(p-1)(q-1)}{\ell}$-fold cover becomes obvious. See \cref{NeckRegionsRemoved} for a sketch of what the fibre appearing in \cref{fig:loopsingmf} looks like with the neck regions removed (ignoring the Lagrangians in this figure for the moment). 
\begin{figure}[H]
	\centering
	\scalebox{.6}{
	\includegraphics[width=0.9\linewidth]{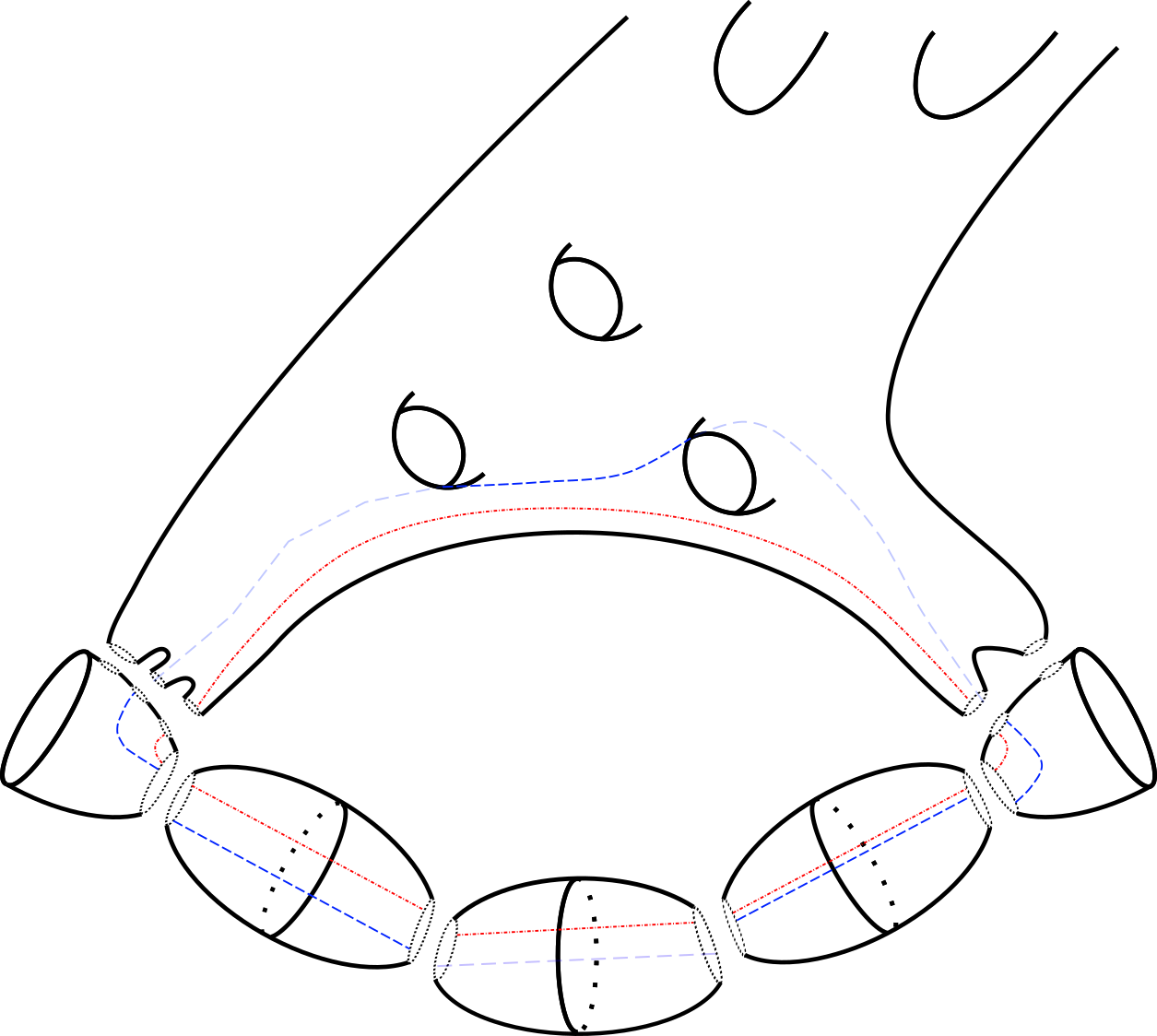}}
	\caption{Red Lagrangian (alternating dashes and dots) corresponds to the real vanishing cycle. Blue (dashed) Lagrangian is another vanishing cycle -- away from the neck regions in the exceptional locus, it is just the shift of the real vanishing cycle by a fixed argument. Light dashes indicate the Lagrangian is on the back side of the surface. }
	\label{NeckRegionsRemoved}
\end{figure}

\subsection{The preliminary vanishing cycles}\label{PrelimVCSubsection}

Let $\vcpr{0}{m,n}$ denote the preliminary vanishing cycle in $\Sigma$ corresponding to the critical point $(\zeta^{m\ell}\lambda_{1,\text{crit}}^+, \zeta^{m(1-\ell)}\eta^n\mu_{1,\text{crit}}^+)$ and the preliminary vanishing path $\gamma_{m,n}$.  In this subsection we explain the necessary alterations to the arguments in the maximally graded case required in order to describe these cycles. 
\begin{rmk}\label{IndependenceOfChartsRMK}
Before moving on to our argument, we would like to remind the reader that the critical points of type (\ref{crit4}) are visible in any chart. It would therefore be equivalent to work in any chart when studying the vanishing cycles corresponding to these critical points, and we choose to work in the chart $\widetilde{X}_1$.
\end{rmk}

Since $\wttilde_\varepsilon$ has real coefficients in a given chart, we can temporarily view it as a map $\R^2\rightarrow \R$. In the chart $\widetilde{X}_1$, this function has a local minimum at $(\lambda^+_{1,\mathrm{crit}}, \mu^+_{1,\mathrm{crit}})$, where it attains the value $c_\mathrm{crit} < 0$. There are no critical values in the interval $(c_\mathrm{crit}, 0)$, so the level sets $\wttilde_1^{-1}(c)$ for $c$ in this range have a component which is a smooth loop encircling $(\lambda^+_{1,\mathrm{crit}}, \mu^+_{1,\mathrm{crit}})$, and which shrinks down to this point as $c \searrow c_\mathrm{crit}$.  As $c \nearrow 0$ this loop, which we denote by $\Lambda_c$, converges to a piecewise smooth curve, $\Lambda_0$, whose segments are given in the chart $\widetilde{X}_i$ by the $\mu_i=0$, $\lambda_i=0$ and $\fact_i=\varepsilon$ (recall that $\fact_i=\wttilde_i\lambda_i^{-1}\mu_i^{-1}$). It is worth reiterating that the curve $\fact_i$ only intersects the coordinate hyperplanes of a given chart in the fibre over the origin only when $i=1,\ell$, where it intersects $\mu_1=0$ and $\lambda_\ell=0$, respectively. \\

Returning to the full complex picture, we exploit the symmetry of the situation, as well as the fact that the monodromy of parallel transport around a small arc centred at the origin is supported in the discs which are removed to give $\Sigma'$. In particular, our main task is to understand the monodromy in these regions. Fortunately, by taking $\varepsilon$ sufficiently small, the symplectic form can be made arbitrarily close to 
\begin{align*}
\frac{i}{2}(\diff\lambda_i\wedge\diff\bar{\lambda}_i+\diff\mu_i\wedge\diff\bar{\mu}_i)
\end{align*}
in the neck regions, meaning that only a minor adaptation of the parallel transport arguments of the maximally graded case is required.  Namely, over a path $c(t)$ contained in a neck region where the symplectic form is as above, symplectic parallel transport between the fibres of $\wttilde_\varepsilon$ is described by the ODE
\begin{equation}
\label{ParTransODE}
\begin{pmatrix}\dot{\lambda}_i \\ \dot{\mu}_i\end{pmatrix} = \frac{\dot{c}}{|\diff \wttilde_i|^2} \begin{pmatrix}\overline{\partial_{\lambda_i}\wttilde_i} \\ \overline{\partial_{\mu_i}\wttilde_i}\end{pmatrix}.
\end{equation}
This equation clearly preserves the real part of the equation along the path which follows the negative real axis. Therefore, the loops $\Lambda_c$ are taken to each other by parallel transport, and the loop $\Lambda_{-\delta}$ is the preliminary vanishing cycle $\vc{0}{0,0}^{\text{pr}}$. \\

Since parallel transport only affects the part of the curve which passes through the neck region, where the monodromy is contained, we are able to use the local description of the symplectic form in each neck region and patch the result together. In particular, how the curve is identified between patches is unaffected by parallel transport, since this identification happens away from the neck regions. Moreover, away from the neck regions, the curve $\vc{0}{0,0}^{\text{pr}}$ is particularly simple to describe: in the $\lambda_1$-axis of $\widetilde{X}_1$, it comprises the real line segment joining the deleted disc at the origin to the deleted disc about the real $\frac{p-1}{\ell}-{\text{th}}$ root of $\eps$, and in the $\mu_1$-projection, it is the straight line emanating from the deleted disc at the origin and going to infinity along the positive real axis. In both the $\lambda_i$ and $\mu_i$ projections in the charts $\widetilde{X}_i$ for $i=2,\dots,\ell-1$, as well as the $\lambda_{\ell}$ projection in $\widetilde{X}_{\ell}$, it is similarly the straight line emanating from the deleted disc at the origin and going to infinity along the positive real axis. In the $\mu_{\ell}$-projection in the chart $\widetilde{X}_\ell$ it is the straight line segment joining the deleted disc at the origin to the real $\frac{q-1}{\ell}-\text{th}$ root of $\eps$ in the $\mu_\ell$-line. Finally, the rest of the curve is the positive real lift of the line segment joining the deleted discs about $(\eps,0)$ and $(0,\eps)$ in $\{u+v=\eps\}$ under the covering map described above. A sketch of this real Lagrangian vanishing cycle away from the neck regions is given in \cref{NeckRegionsRemoved}. It is straightforward to check that, in each of the neck regions, the curve $\vc{0}{0,0}^{\text{pr}}$ is given by a hyperbola. For example, near the origin in the chart $\widetilde{X}_1$, it is given by 
\begin{align}\label{HyperbolicApproximation}
(\lambda_1,\mu_1)= \sqrt{\delta/\eps}(e^s,e^{-s})
\end{align}
where $s$ is a small real variable. \\

Now that we understand the real vanishing cycle $\vc{0}{0,0}^{\text{pr}}$, we utilise the symmetry present in the situation to characterise $\vc{0}{m,n}^{\text{pr}}$. As before, we decompose $\gamma_{m,n}$ into its radial segment from $-\zeta^m\eta^n\delta$ to $\zeta^m\eta^nc_\mathrm{crit}$ and circular segment from $-\delta$ to $-\zeta^m\eta^n\delta$. By construction, the map
\begin{align*}
f_{m,n}^{(i)}:\widetilde{X}_i&\rightarrow \widetilde{X}_i\\
(\lambda_i,\mu_i)&\mapsto (\zeta^{m(\ell+1-i)}\eta^{n(1-i)}\lambda_i,\zeta^{m(i-\ell)}\eta^{in}\mu_i)
\end{align*}
is a symplectomorphism on each patch with intertwines the transition maps and fits together to yield a globally defined symplectomorphism $f_{m,n}:\widetilde{X}\rightarrow \widetilde{X}$ which intertwines the map $\wttilde_\varepsilon :\widetilde{X}\rightarrow \C$. It is therefore clear that $f_{m,n}(\vc{0}{0,0}^{\text{pr}})$ is the vanishing cycle corresponding to the straight line segment from $-\zeta^m\eta^n\delta$ to the critical point\footnote{We are free the describe the critical points of type (\ref{crit4}) in any chart, so we choose the first chart for simplicity.} $(\zeta^{m\ell}\lambda_{1,\text{crit}}^+, \zeta^{m(1-\ell)}\eta^n\mu_{1,\text{crit}}^+)$, and a full description of $\vc{0}{m,n}^{\text{pr}}$ results from parallel transporting this around the arc from $-\zeta^m\eta^n\delta$ to $-\delta$. Since the monodromy of parallel transport along such an arc is contained in the neck regions, it is immediate that $\vc{0}{m,n}^{\text{pr}}=f_{m,n}(\vc{0}{0,0}^{\text{pr}})$ in $\Sigma'$, where it comprises: the straight line segment in the $\lambda_1-$line from the deleted disc about the origin to the deleted disc about the critical point $\zeta^{m\ell}\sqrt[\frac{p-1}{\ell}]{\varepsilon}$ (where $\sqrt[\frac{p-1}{\ell}]{\varepsilon}$ is the real root) and the straight line segment emanating from the deleted disc at the origin with argument $2\pi \big(\frac{m(1-\ell)}{p-1} +\frac{n}{q-1}\big)$. Similarly, in the $\lambda_i$ and $\mu_i$ projections in the chart $\widetilde{X}_i$, the Lagrangian is given by the straight line segment  emanating from the deleted disc at the origin with arguments $2\pi \big(\frac{m(\ell+1-i)}{p-1} +\frac{n(1-i)}{q-1}\big)$ and $2\pi \big(\frac{m(i-\ell)}{p-1} +\frac{in}{q-1}\big)$, respectively. In the $\lambda_{\ell}$ projection in the chart $\widetilde{X}_{\ell}$, the Lagrangian is emanating from the deleted disc at the origin with argument $2\pi \big(\frac{m}{p-1} +\frac{n(1-\ell)}{q-1}\big)$, and in the $\mu_{\ell}$ projection it is given by the straight line segment joining the deleted disc about the origin to the deleted disc about $\eta^{n\ell}\sqrt[\frac{q-1}{\ell}]{\varepsilon}$. Finally, the rest of the curve in $\Sigma'$ is given by the lift of the line segment joining the deleted discs about $(\eps,0)$ and $(0,\eps)$ in $\{u+v=\eps\}$ to the $\zeta^{m\ell}\R_+\times\zeta^{m(1-\ell)}\eta^n\R_+$-locus in $\C^2\simeq \widetilde{X}_1$ under the projection described above. In the complement of the neck regions, a sketch of such a vanishing cycle is given by the blue (dashed) Lagrangian in \cref{NeckRegionsRemoved}. In the neck regions of $\Sigma$, the curve $f_{m,n}(\vc{0}{0,0}^{\text{pr}})$ is given by a hyperbola of the form \eqref{HyperbolicApproximation}, and to complete the description of $\vc{0}{m,n}^{\text{pr}}$, we must analyse how this changes under parallel transport along the arc from $-\zeta^m\eta^n\delta$ to $-\delta$ through the angle $\theta_{m,n}$. \\

In the neck region about the origin in a given chart, the polynomial $\wttilde_i$ can be approximated by $-\eps\lambda_i\mu_i$. Here, the parallel transport equation \eqref{ParTransODE} becomes
\begin{equation}
\label{LocalParTrans}
\begin{pmatrix}\dot{\lambda}_i \\ \dot{\mu}_i\end{pmatrix} = \frac{-\dot{c}}{\eps(|\lambda_i|^2+|\mu_i|^2)} \begin{pmatrix}\overline{\mu}_i \\ \overline{\lambda}_i\end{pmatrix}.
\end{equation}
Moreover, the $\zeta^{m(\ell+1-i)}\eta^{n(1-i)}\R_+\times\zeta^{m(i-\ell)}\eta^{in}\R_+$-locus in the $\vc{\lambda_i\mu_i}{}$-neck region is approximated by 
\begin{align*}
(\lambda_i,\mu_i)=\sqrt{\delta/\eps}(\zeta^{m(\ell+1-i)}\eta^{n(1-i)}e^s,\zeta^{m(i-\ell)}\eta^{in}e^{-s}).
\end{align*}
Studying the solution to parallel transport over the line $c(t)=-\delta e^{it}$ as $t$ ranges from $\theta_{m,n}$ to $0$, we postulate a solution of the form $(\lambda_i,\mu_i)=\sqrt{\delta/\eps}(e^{s+i\varphi},e^{-s+i(t-\varphi)})$, where $\varphi$ is a real function of $s$ and $t$. Plugging this into \eqref{LocalParTrans} yields
\begin{align*}
\begin{pmatrix}\dot{\varphi}\lambda_i \\ (1-\dot{\varphi})\mu_i\end{pmatrix} = \frac{\lambda_i\mu_i}{|\lambda_i|^2+|\mu_i|^2} \begin{pmatrix}\overline{\mu}_i \\ \overline{\lambda}_i\end{pmatrix}.
\end{align*}
The initial condition is given by 
\begin{align*}
\varphi(s,\theta_{m,n})=2\pi \big(\frac{m(\ell+1-i)}{p-1} +\frac{n(1-i)}{q-1}\big),
\end{align*}
and so the general solution is calculated to be 
\begin{align*}
\varphi(s,t)=2\pi \big(\frac{m(\ell+1-i)}{p-1} +\frac{n(1-i)}{q-1}\big)+\frac{e^{-2s}(t-\theta_{m,n})}{e^{2s}+e^{-2s}}.
\end{align*}
in particular, the value of the function at the end of the parallel transport is given by 
\begin{align*}
\varphi(s,0)&=2\pi \big(\frac{m(\ell+1-i)}{p-1} +\frac{n(1-i)}{q-1}\big)+\frac{-e^{-2s}\theta_{m,n}}{e^{2s}+e^{-2s}}\\
&=\frac{2\pi}{e^{2s}+e^{-2s}}\Big(\frac{m(\ell+1-i)e^{2s}+m(\ell-i)e^{-2s}}{p-1}+\frac{n(1-i)e^{2s}-ine^{-2s}}{q-1}\Big).
\end{align*}
This describes the argument of the $\lambda_i$-component of $\vc{0}{m,n}^{\text{pr}}$ (or the negative of the $\mu_i$-component), and is in agreement with our expectation: as $s\rightarrow\infty$ this neck region joins the $\lambda_i$-axis, where we know that the $\lambda_i$-component of $\vc{0}{m,n}^{\text{pr}}$ has argument $2\pi \big(\frac{m(\ell+1-i)}{p-1} +\frac{n(1-i)}{q-1}\big)$. Similarly, as $s\rightarrow -\infty$, this neck region joins the $\mu_i$-axis, where we know that the negative of the $\mu_i$-component of $\vc{0}{m,n}^{\text{pr}}$ has argument $2\pi \big(\frac{m(\ell-i)}{p-1} -\frac{in}{q-1}\big)$. Note that $\arg \lambda_1=\arg\lambda_2=\dots=\arg\lambda_\ell$, and similarly for the arguments of $\mu_i$. We therefore see that, if we view the smooth fibre as being glued from cylinders as in \cite[Section 3.2]{HabMilnorFibreHMS} and coordinatise the middle cylinder as upwards\footnote{Note that we have reversed the orientation of the diagram in comparison to \cite[Figure 2]{HabMilnorFibreHMS} so that the orientation of the surface agrees with the orientation of the page.} being in the positive $\lambda_i$ direction, that the Lagrangian $\vc{0}{m,n}$ enters the middle cylinder with $\lambda_\ell$ argument $-\frac{2\pi\ell n}{q-1}$ from the left (which we are thinking of as being the $\ell^\text{th}$ neck region), and winds upwards -- increasing the argument of $\lambda_i$ -- by $\frac{2\pi\ell m}{p-1}+\frac{2\pi\ell n}{q-1}$ degrees to exit the right hand side of the cylinder (the first neck region) at $\frac{2\pi\ell m}{p-1}$ degrees. Whilst more modification is needed to reach the final configuration for our collection of vanishing cycles, this winding is visible in the central cylinder of \cref{fig:milnorfiberexample}.
\begin{rmk}
It should be emphasised that \cref{fig:milnorfiberexample} is just the ribbon graph associated to the basis of homology given by the exact Lagrangians, together with one of the non-exact Lagrangians. It doesn't matter which non-exact Lagrangian one picks to form this graph, since they all represent the same homology class and will result in the same diagram. This is convenient since one can view the neck regions where there is monodromy as neighbourhoods of Lagrangians in order to depict the winding in these regions. 
\end{rmk}
In order to fully describe the curve $\vc{0}{m,n}^{\text{pr}}$ in the smooth fibre, we run analogous arguments about the other neck regions in which the Lagrangian passes -- namely, the discs deleted about the point $\zeta^{m\ell}\sqrt[\frac{p-1}{\ell}]{\varepsilon}$ on the $\lambda_1$-axes and the point $\eta^{n\ell}\sqrt[\frac{p-1}{\ell}]{\varepsilon}$ on the $\mu_\ell$-axis. In the neck region corresponding to $\vc{\mu_1\fact_1}{m}$ we run the same argument as above, but this time with local coordinate $\lambda_1'$, where $\lambda_1=\zeta^{m\ell}\lambda_{1,\text{crit}}^+-\lambda_1'$, and in this case it is the coordinate $\lambda_1'$ which interpolates from $\frac{2\pi\ell m}{p-1}$ to $2\pi \big(\frac{m(\ell-1)}{p-1} -\frac{in}{q-1}\big)$. The analogous statement is true about the neck corresponding to $\vc{\lambda_\ell\fact_\ell}{n}$. Moreover, the $\{\fact_i=\varepsilon\}$ part of the curve is essentially uninteresting, since the corresponding components of the different $\vc{0}{m,n}^{\text{pr}}$ in this segment are different lifts of the same segment in $\{u+v=\varepsilon\}$. 
\begin{rmk}
	It should be noted that the Lagrangians in \cref{fig:milnorfiberexample} wind a fixed amount in each neck region, not just those in the middle cylinder; however, the Lagrangians do not intersect in any other region. For example, working right-to-left in this figure, $\vc{0}{1,1}$ enters the left cylinder lower than $\vc{0}{0,1}$ and also winds more in the downwards direction. Indeed, the Hamiltonian isotopies were precisely chosen so that all intersections between $\vc{0}{i,j}$ happen in the central cylinder. We have therefore suppressed this winding in the left and right annuli so as to not obfuscate the more important point that there will be a relation for the Floer product. 
\end{rmk}
\begin{rmk}
From the above description, it is clear that the analogue of \cite[Remark 3.2]{HabermannSmith} holds. Namely, for two Lagrangians $\vc{0}{m,n}$ and $\vc{0}{M,N}$ with $0<m-M<\frac{p-1}{\ell}$ and $\kappa\frac{q-1}{\ell}<n-N<(\kappa+1)\frac{q-1}{\ell}$, the difference in $\lambda_i$ arguments varies monotonically from $-2\pi(n-N)\ell/(q-1)$ to $2\pi(m-M)\ell/(p-1)$. The second term is strictly positive, and the first satisfies $2\pi(\kappa+1)<-2\pi(n-N)\ell/(q-1) <-2\pi\kappa$, so the arguments of these two curves are equal modulo $2\pi$ precisely $\kappa+1$ times. In addition, they intersect transversally at each of these points. In analogy with \cref{StabOriginMorphisms}, we also do not only have morphisms when $0<m-M<\frac{p-1}{\ell}$: when $0<M-m<\frac{p-1}{\ell}$ and $\kappa\frac{q-1}{\ell}<n-N<(\kappa+1)\frac{q-1}{\ell}$ for $\kappa\geq 1$, there are precisely $\kappa$ transverse intersection. 
\end{rmk}
\subsection{Modifying the vanishing paths}
As in \cite[Section 3.4]{HabermannSmith}, the vanishing paths constructed above do not form a distinguished basis; indeed, the paths not only intersect, but do so along segments. Fortunately, the argument of \textit{loc. cit.} used to circumvent this issue can be adapted to this setting. Namely, we perturb the vanishing paths $\gamma_{m,n}$ slightly to $\gamma_{m,n}'$ so that they no longer intersect along their radial segments. In addition, for those vanishing cycles whose corresponding vanishing path has radial segment greater than than $2\pi$ we alter the vanishing paths to go \emph{outside} of the critical points to form $\gamma_{m,n}''$, although it can be shown in the same way as \cite[Lemma 3.3]{HabermannSmith} that this alteration has no effect on the corresponding vanishing cycle. Finally, we perturb the fibration slightly to separate the vanishing paths going to the origin, as well as those whose critical points correspond to the same critical value. This is summarised in the following proposition, whose proof follows \textit{mutatis mutandis} from that of the maximally graded case given in \cite[Proposition 3.4]{HabermannSmith}. 
\begin{prop}
	\label{propAfromPrelim}
	There exists a perturbation of $\wttilde_{\varepsilon}$ and a distinguished basis of vanishing paths such that the corresponding vanishing cycles are arbitrarily small perturbations of the $\vcpr{0}{m,n}$, $\vc{\mu_1\fact_1}{m}$, $\vc{\lambda_\ell\fact_\ell}{n}$ and $\vc{\lambda_i\mu_i}{}$ for $0\leq m,M\leq \frac{p-1}{\ell}-1$, $0\leq n,N\leq q-2$ as constructed above.  The $\vcpr{0}{m,n}$ are ordered by decreasing value of $\theta_{m,n}$, and by choosing the starting direction for our clockwise ordering to be $e^{i\theta}$, for $\theta$ a small positive angle, they occur before all of the other vanishing cycles.\hfill$\qed$
\end{prop}
\subsection{Isotoping the vanishing cycles and computing the morphisms}
Now that we have constructed a distinguish basis of vanishing paths, we must compute the relevant Floer cohomology groups. In order to do this, we must Hamiltonian isotope the Lagrangians in the smooth fibre so that they intersect transversally. Fortunately, two Lagrangians only intersect non-transversally in the case of $\vc{0}{m,n}$ and $\vc{0}{M,N}$ where $m=M$ or $n\equiv N\bmod \frac{q-1}{\ell}$, where they intersect along segments. This is clear from the description of $\vc{0}{m,n}$ and $\vc{0}{M,N}$ in $\Sigma'$ as $f_{m,n}(\vc{0}{0,0})$ and $f_{M,N}(\vc{0}{0,0})$, respectively. After describing the relevant perturbations, we will have constructed the \emph{final} vanishing cycles, denoted by $\mathcal{A}_{\widecheck{\overline{\Gamma}}}$. \\

We begin by describing the isotopies of the Lagrangians on the $\lambda_1$-axis. Namely, to achieve transversality, we isotope each $\vc{0}{m,n}$ anticlockwise in the $\lambda_1$ direction between the two necks by an amount proportional to $n$, correspondingly altering the curve at the boundaries of the neck region to keep the curve continuous. To make this isotopy Hamiltonian, we push the curve the corresponding amount in the clockwise direction around the $\vc{\lambda_1\mu_1}{}$ neck region. The picture in this case is essentially \cite[Figure 8]{HabermannSmith}.\\
To achieve transversality in the $\mu_\ell$ line, we perform essentially the same procedure as above, and of that in the maximally graded case; however, here we isotope the curves in the anticlockwise $\mu_\ell$ direction by an amount proportional to $m+\frac{p-1}{\ell}\lfloor\frac{n\ell}{q-1}\rfloor$ 
in the region between the two necks, and then the curve is pushed clockwise in the $\vc{\lambda_\ell\mu_\ell}{}$ neck region to compensate and make the isotopy Hamiltonian. See Figure \ref{fig:milnorfiberexample} for an example\footnote{The construction of this smooth fibre from gluing cylinders and strips is not the one given in \cite[Section 6.1]{HabStackyCurves}, although is equivalent to this, since it results from a different choice representative of nodes in the construction of the quotient ribbon graph in \textit{loc. cit.}}, where we present the surface as being glued from cylinders. 
 \begin{figure}[H]
 	\centering
 	\includegraphics[width=0.8\linewidth]{LoopMilnorFibreEdited} 	
 	\begin{tikzpicture}
 	\draw (-15,-1) node[]{$\vc{0}{0,0}=$};
 	\draw[thick] (-14.25,-1) -- (-13.5,-1);
 	\draw (-15,-1.5) node[]{$\vc{0}{1,0}=$};
 	\draw[thick, dash pattern={on 3pt off 1pt},color1] (-14.25,-1.5) -- (-13.5,-1.5);
 	
 	\draw (-12,-1) node[]{$\vc{0}{0,1}=$};
 	\draw[thick, dash dot, red ] (-11.25,-1) -- (-10.5,-1);
 	\draw (-12,-1.5) node[]{$\vc{0}{1,1}=$};
 	\draw[very thick, dash pattern={on 1pt off 1pt},blue] (-11.25,-1.5) -- (-10.5,-1.5);
 	\end{tikzpicture}
 	\caption{Smooth fibre for $\wt=x^5y+y^3x$ with $\ell=2$ as being glued from cylinders -- top and bottom of each cylinder are identified. This is a genus three surface with three boundary punctures. Vanishing cycles and spin structures are shown. The light green (big) and light blue (small) triangles contribute to the relation corresponding to (iii) in \cref{figLoopQuiver}. The analogue of \cref{fig:loopsingmf} in this situation is \cref{fig:loopsingularfibrespecificexample}.}
 	\label{fig:milnorfiberexample}
 \end{figure}

\begin{figure}[H]
	\centering
	\includegraphics[width=0.3\linewidth]{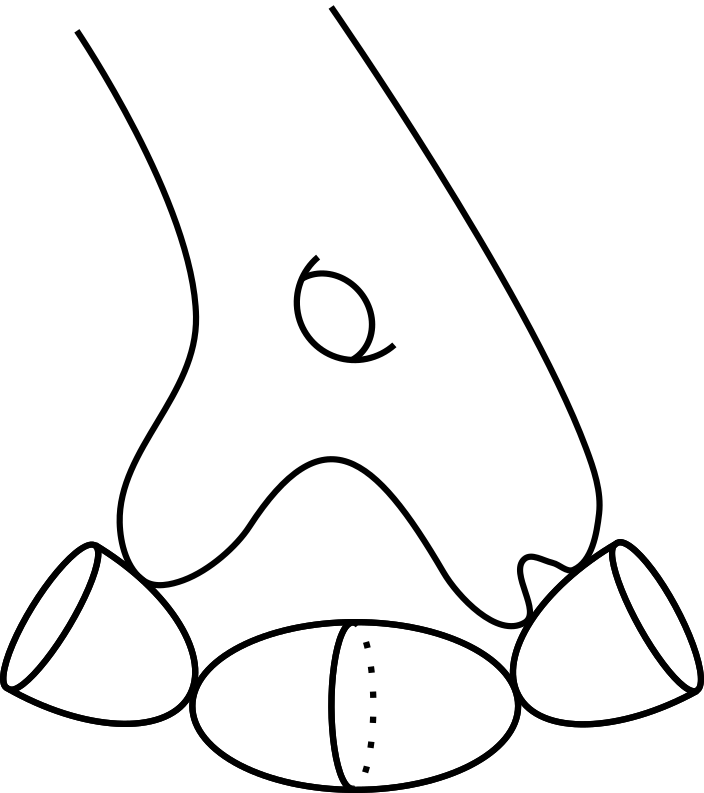}
	\caption{The analogue of \cref{fig:loopsingmf} for $\wt=x^5y+y^3x$  with $\ell=2$. In particular, the vertical Lagrangians in \cref{fig:milnorfiberexample} are contracted to the intersection points of the irreducible components in this figure as one goes  from the reference point to the origin along the negative real axis.}
	\label{fig:loopsingularfibrespecificexample}
\end{figure}

The resulting set of vanishing cycles, $\vc{0}{m,n}$, are then pairwise disjoint, except on the $\vc{\lambda_i\mu_i}{}$ neck regions, where they are either disjoint or intersect transversally. Moreover, we will see momentarily that all of these intersection points are graded in degree $0$, meaning that the differential on the Floer complex vanishes, and all intersection points survive to cohomology. Once this has been shown, it is then straightforward to see that, additively, the objects in $\mathcal{A}_{\widecheck{\overline{\Gamma}}}$ match that of $\mathcal{B}$ via
	\begin{equation}
\label{ObjectMatch}
\begin{aligned}
\vc{0}{m,n} &\leftrightarrow \obj{0}{i,j}
\\ \vc{\mu_\ell\fact}{m} &\leftrightarrow \obj{x}{i}[3]
\\ \vc{\lambda_1\fact}{n} &\leftrightarrow \obj{y}{j}[3]
\\ \vc{\lambda_k\mu_k}{} &\leftrightarrow \obj{w_r}{}[3]
\end{aligned}
\quad \text{with} \quad
\begin{aligned}
i+m&=p-1
\\ j+n&=q-1\\
r+k&=\ell.
\end{aligned}
\end{equation}
Note that, in the above, the last identification is taken modulo $\ell$, so that $\vc{\lambda_\ell\mu_\ell}{}$ is identified with $\obj{w_\ell}{}[3]$. All that is left to do to complete the theorem in the undeformed case is show that the objects are graded, so that the corresponding Floer complexes are graded, and then check that the morphisms compose in the claimed way.
\subsection{Brane structures}
In this subsection, we specify the brane structure on the Lagrangians we consider by determining a grading and spin structure. This discussion is also valid for \cref{ChainAmodel} and \cref{BPAmodel}. \\
By construction, the morphisms in the B--model are  $\Z$-graded, and here we grade the Lagrangians so the corresponding Floer cohomology groups will similarly be $\Z$-graded. Recall (\cite{SeidelGraded}) that a symplectic manifold $(M,\omega)$ is \emph{gradable} if $K_M^{\otimes 2}\simeq \mathcal{O}_M$, and a grading is a choice of such trivialisation. This is possible if and only if $2c_1(M)=0$, and so, in particular, all Calabi--Yau manifolds are gradable. Given a trivialising section $\Theta\in \Gamma(M,K_M^{\otimes2})$, there is a map
\begin{align*}
\alpha_M:\mathrm{LGr}(TM)&\rightarrow S^1\\
L_m&\mapsto \mathrm{arg}(\Theta|_{L_m}).  
\end{align*}
For any Lagrangian $L\hookrightarrow M$, there is a natural map $L\rightarrow \mathrm{LGr}(TM)$ given by taking the class of the tangent space of $L$ at each point, and we say that $L$ is \emph{gradable} with respect to the grading on $M$ if the map $L\rightarrow \mathrm{LGr}(TM)\xrightarrow{\alpha_M} S^1\simeq\R/\Z$ lifts to a map $\alpha_L^{\#}:L\rightarrow \R$. This is possible if and only if the Maslov class of $L$ vanishes, i.e. the map $L\rightarrow S^1$ is homotopic to the constant map. \\
For $\Sigma$ a real 2-dimensional manifold, i.e. a surface, gradings correspond naturally to line fields \cite[Section 13(c)]{SeidelBook}, meaning a section of $\P_{\R}(T\Sigma)\simeq \mathrm{LGr}(T\Sigma)$. Given a line field $\eta$ grading $\Sigma$, a Lagrangian $\gamma:S^1\rightarrow \Sigma$ is gradable with respect to this line field if and only if $\gamma^*\eta$ and $\gamma^*TL$ are homotopic in $\gamma^*\P_\R(T\Sigma)$. Given a graded surface $\Sigma$ with two graded Lagrangians $\alpha_{L_i}^{\#}:L_i\rightarrow \R$ for $i=0,1$, \cite[Example 11.20]{SeidelBook} shows that a transverse intersection point $x\in L_0\pitchfork L_1$ has degree 
\begin{align*}
\lfloor \alpha_{L_1}^{\#}(x)-\alpha_{L_0}^{\#}(x)\rfloor +1
\end{align*}
in $CF^*(L_0,L_1)$. \\

In our situation, the required choice of grading is the one which restricts to $\Sigma$ from the grading of $\widetilde{X}$. By the crepancy of the resolution $\pi$, this has vanishing first Chern class, and, since $\widetilde{X}$ is diffeomorphic to the Milnor fibre of $x^2+y^2+z^\ell$, its homotopy type is a bouquet of $2$-spheres. This implies that $\widetilde{X}$ has vanishing first cohomology, and so, in particular, there is a unique choice of trivialisation of $K_{\widetilde{X}}^{\otimes 2}$ up to homotopy\footnote{If $K_{\widetilde{X}}^{\otimes 2}\simeq \mathcal{O}_{\widetilde{X}}$ the homotopy classes of trivialisations form a torsor for $H^1(\widetilde{X})$.}. We are therefore justified in the use of the definite article when talking about gradings of $\widetilde{X}$. Restricting the trivialisation of $K_{\widetilde{X}}^{\otimes 2}$ to $\Sigma$ uniquely determines a line field on $\Sigma$ with respect to which our Lagrangians are gradable. Conversely, any line field with respect to which each vanishing cycle is gradable is homotopic to the one which cones from restricting the trivialisation of $K_{\widetilde{X}}^{\otimes 2}$. This uses the more general fact that, if one has a collection of Lagrangians on a surface which spans the first homology, then there is a unique homotopy class of line field with respect to which this collection is gradable -- cf. \cite[Section 1]{LekiliPolishchuk}.\\
That these Lagrangians are gradable means that the argument between the line field and the projectivised tangent direction of a given Lagrangian is within a sufficiently small range. For example, in the presentation of the smooth fibre as being glued from annuli as in \cref{fig:milnorfiberexample}, the line field can be taken to be approximately horizontal on each annulus and approximately parallel to the boundary along each attaching strip. This is proven in \cite[Section 6.1]{HabStackyCurves}; however, in this case, it also follows from the fact the Lagrangians are clearly gradable with respect to this line field, together with the above fact about uniqueness of line fields. \\

As in the maximally graded case, we can grade the Lagrangians such that $\alpha_{L_i}^{\#}$ is valued between $0$ and $1/2$. Moreover, if $i>j$ in the ordering on the vanishing cycles given in \cref{propAfromPrelim}, then $1/2> \alpha_{L_i}^{\#}>\alpha_{L_j}^{\#}>0$, and so all intersection points are graded in degree zero. 	\\

Now that we have fixed a grading, in which all intersection points are in degree zero, we discuss spin structures and the consequences for signs. We refer to \cite[Chapter II, Section 11]{SeidelBook} and \cite[Section 4.6]{AOK} for more details, where the latter reference also deals with non-exact Lagrangians. In the case at hand, spin structures on $S^1$ are given by double covers, of which there are two -- the trivial and non-trivial. Since the data of the Lagrangian brane in the Milnor fibre must be the restriction of brane data from the thimble, we insist on equipping our Lagrangians with the spin structure which extends to the thimble, i.e. the non-trivial double cover of $S^1$. In such cases, it is convenient to record this spin structure by $\star\in L$, marking the point where the double cover is ramified. Namely, the spin structure on $L$ restricted to $L\setminus\{\star\}$ is equipped with a fixed trivialisation. It does not matter where one chooses the star recording the spin structure to be, as long as it is not at an intersection point with another Lagrangian. For the Lagrangians corresponding to critical points whose critical value is the origin, we take the star to be anywhere away from an intersection point with another Lagrangian. For all other Lagrangians, we take the stars recording the spin structure to be away from any intersections and discs contributing to the product (i.e. on the sections of the Lagrangians which are lifts of the line $\{u+v=\varepsilon\}$ by the $\frac{(p-1)(q-1)}{\ell}$-fold cover).\\
In our case, where all morphisms are graded in degree zero, the sign contribution for a holomorphic disc $u$ with convex corners $L_{i_j}\pitchfork L_{i_{j+1}}$ (with the indices counted cyclically) contributing to $\mu_k$ is
\begin{align*}
(-1)^{\nu(u)},
\end{align*}
where $\nu(u)$ is the number of stars on the boundary. In general, one must also take into account the degrees of the intersection points, as well as the orientation of the boundary of $u$ with respect to the orientation of the $L_{i_j}$. It should therefore be reiterated that that this sign count is specific to our situation (or, more generally, where all morphisms are graded in even degree). Moreover, the (non-)exactness of the symplectic form or Lagrangian vanishing cycles does not play a role in determining these signs. 
\subsection{Composition and completion of proof}\label{CompositionSection}

Suppose $L_0$, $L_1$ and $L_2$ are three (final) vanishing cycles such that $L_0 < L_1 < L_2$ with respect to the ordering on the category $\mathcal{A}_{\widecheck{\overline{\Gamma}}}$ (we are calling them $L$ rather than $V$ to avoid conflict with our earlier notation for specific cycles).  We need to compute the composition
\begin{equation}
\label{eqFloerProduct}
HF^*(L_1, L_2) \otimes HF^*(L_0, L_1) \rightarrow HF^*(L_0, L_2),
\end{equation}
which is defined by counting pseudo-holomorphic triangles, and Seidel \cite[Section (13b)]{SeidelBook} shows that, in the exact setting, this can be done combinatorially by counting triangular regions bounded by the $L_i$. In our non-exact setting, the assumptions of \cref{CrepResolutionSection} ensure that the situation at hand is similarly combinatorial. \\

By the ordering on the vanishing cycles, the only possibility for \eqref{eqFloerProduct} to be non zero is if $L_0=\vc{0}{m,n}$ and $L_1=\vc{0}{M,N}$ where either $m\geq M$ and $n\geq N$ or $M<m$ and $n\geq N+\frac{q-1}{\ell}$. There are then four possibilities for $L_2$ which would potentially give a non-zero composition:
\begin{enumerate}[(i)]\label{HolTriangleCases}
	\item $L_2=\vc{\lambda_r\mu_r}{}$ for some $i=1,\dots, \ell$,
	\item $L_2= \vc{\lambda_1\fact}{t}$, for $t\equiv n\bmod \frac{q-1}{\ell}$
	\item $L_2=\vc{\mu_\ell\fact}{m}$, 
	\item $L_2=\vc{0}{r,s}$ for some $(r,s)\neq (M,N)$ and either
	\begin{enumerate}[(a)]
		\item $M\geq r$ and $N\geq s$, or 
		\item $M<r$ and $N\geq s+\frac{q-1}{\ell}$.
	\end{enumerate} 
\end{enumerate}
In the cases (ii) and (iii), there is a single obvious holomorphic disc contributing to the product, regardless of the ranks of the cohomology groups, and this is shown to be the only such disc by the same methods as in the maximally graded case. In case (iv), all Lagrangians involved are exact, so the product only depends on the intersection points of the Lagrangians, with signs determined by the gradings on the Lagrangians as well as the spin structures. Since we have chosen these to be away from any holomorphic discs contributing to the product, all signs are $+1$. Moreover, all of the Lagrangians in these cases are exact, and so one can work over $\C$, rather than the Novikov field, simply by rescaling. The remaining case, in contrast, is more complicated.\\

In case (i), there are $\mathrm{rank}_{\C}HF^*(\vc{0}{m,n},\vc{0}{M,N})$ holomorphic triangles contributing to \eqref{eqFloerProduct}, where the output of each is a scalar multiple of the intersection point $\vc{0}{M,N}\pitchfork\vc{\lambda_r\mu_r}{}$. The fact that these are all such maps follows from the same argument as in the maximally graded case, although where $L_{\cup}$ is now three circles, one of which intersects the other two circles once, and two circles intersect each other $\mathrm{rank}_{\C}HF^*(\vc{0}{m,n},\vc{0}{M,N})$ times. The composition argument follows as in the maximally graded case when $\mathrm{rank}_{\C}HF^*(\vc{0}{m,n},\vc{0}{M,N})=1$. On the other hand, consider the case where $\mathrm{rank}_{\C}HF^*(\vc{0}{m,n},\vc{0}{M,N})=2$. Then, $m\geq M+\frac{p-1}{\ell}$, $n \geq N$, and, without loss of generality, consider $L_0=\vc{0}{\frac{p-1}{\ell},0}$, $L_1=\vc{0}{0,0}$ and $L_2=\vc{\lambda_r\mu_r}{}$. Then, 
\begin{align*}
HF^*(L_0,L_1) =\mathrm{span}_{\Lambda_\C}\{x^{\frac{p-1}{\ell}},y^{\frac{q-1}{\ell}}\},
\end{align*}
but $HF^*(L_0,L_2)=\mathrm{span}_{\Lambda_\C}\{q_k\}$, and so there must be a relation. Note that we must work over the Novikov field here to account for $L_2$ being potentially non-exact\footnote{We will see below $\vc{\lambda_{\frac{l+1}{2}}\mu_{\frac{l+1}{2}}}{}$ is exact for $\ell$ odd, but it does no harm to include this case in the discussion of non-exact Lagrangians.} \\

Since the symplectic form $\omega_{\Sigma}=\omega|_{\Sigma}$ is exact on fibres, let $\lambda$ be a fixed primitive of this form. Then, the contributions to \eqref{ObjectMatch} are:
\begin{align}
\label{FloerRelationI}
c_r\otimes x^{\frac{p-1}{\ell}}&\mapsto \pm T^{\omega([u_1])}q_r\\
\label{FloerRelationII}
c_r\otimes y^{\frac{q-1}{\ell}}&\mapsto \mp T^{\omega([u_2])}q_r,
\end{align}
where $u_i$ is a contributing holomorphic triangle and $HF^*(L_1,L_2)=\mathrm{span}_{\Lambda_\C}\{c_r\}$. Since all intersection points can be taken to be in degree zero, the sign is determined purely by how many stars recording spin structures the boundary of the relevant disc passes through. As already mentioned, we arrange for all the stars recording the spin structures of the exact Lagrangians to be away from any holomorphic discs, and so, once we have graded the Lagrangians in question, only the spin structure of $L_2$ can affect the sign. Regardless of where this star is chosen, exactly one boundary component of $u_1$ or $u_2$ can pass through it, and so the signs of \eqref{FloerRelationI} and \eqref{FloerRelationII} are opposite -- cf. \cref{fig:milnorfiberexample}. \\

Now, since $L_2$ is the only non-exact Lagrangian in the triangle, observe that $\omega([u_i])=\lambda([\alpha_i])$, where $\alpha_i$ is the segment of $L_2$ on the boundary of $u_i$ and that $\alpha_1-\alpha_2=L_2$ in homology. Putting this all together, we see that there is a relation 
\begin{align*}
c_r\otimes x^{\frac{p-1}{\ell}}+ T^{\lambda([L_2])}c_r\otimes y^{\frac{q-1}{\ell}} =0.
\end{align*}
To complete the argument, we must evaluate $\lambda([L_2])$ and then specialise the Novikov parameter. \\

Towards this end, observe that there is an exact representative of the homology class $L_2$. In the case of $\ell$ even, it is the waist curve half way between\footnote{This corresponds to the equator on the exceptional curve $C_{\frac{\ell}{2}}$ above the origin under parallel transport.} $\vc{\lambda_{\frac{\ell}{2}}\mu_{\frac{\ell}{2}}}{}$ and $\vc{\lambda_{\frac{\ell}{2}+1}\mu_{\frac{\ell}{2}+1}}{}$, and in the case of $\ell$ odd, it is $\vc{\lambda_{\frac{\ell+1}{2}}\mu_{\frac{\ell+1}{2}}}{}$. Moreover, let $C_r$ be the $r$-th curve in the exceptional divisor of $\widetilde{X}$, and observe that, by parallel transport (and the definition of $\omega$ on $\widetilde{X}$), $\omega([C_r])=\lambda([\vc{\lambda_{r}\mu_{r}}{}])-\lambda([\vc{\lambda_{r+1}\mu_{r+1}}{}])=2\pi$. By comparing with the reference exact Lagrangian, we have that $\lambda([\vc{\lambda_r\mu_r}{}])=\pi(\ell+1-2r)$ regardless of the parity of $\ell$. To see this, note that, for $\ell$ odd, $\lambda([V_{\frac{\ell+1}{2}}])=0$ since it is exact. Then, $\lambda([V_{\frac{\ell+1}{2}\pm r}])=\pm2\pi r$ since the symplectic area of each holomorphic cylinder between two consecutive vanishing cycles has area $2\pi$, as demonstrated above. An analogous computation yields the claimed area in the case of $\ell$ even. \\
Finally, we must specialise the Novikov parameter. To do this, we set $T=e^{\frac{\sqrt{-1}}{\ell}}$ to obtain the relation (iii) in \cref{figLoopQuiver} under the identification \eqref{ObjectMatch}. That this specialisation preserves Hamiltonian invariance follows from the fact that there is no holomorphic cylinder between two Lagrangians in our collection which has symplectic area $2\pi\ell$. The cases of rank higher than two then follow by composition relations which have already been established. \\

Now that we have set it up, the proof of the main theorem in the loop case follows the same argument as in \cite[Theorem 1]{HabermannSmith}.
		
\begin{thm}[\cref{mainTheorem}, undeformed loop polynomial case]
	Under \eqref{ObjectMatch}, the $\Z$-graded $A_\infty$-category $\mathcal{A}_{\widecheck{\overline{\Gamma}}}$ is described by the quiver with relations in Figure \ref{figLoopQuiver} and is formal. In particular, by Theorem \ref{LoopEndAlgebra} it is quasi-equivalent to $\mathcal{B}$, and hence there is an induced quasi-equivalence
	\[
	\mathrm{mf}(\C^2, \Gamma, \w) \simeq \mathcal{FS}(\wttilde).
	\]
\end{thm}
\section{Chain B-model}
\label{ChainBmodel}
In this section, we study the dg-category of $L$-graded matrix factorisations of $\w=x^py+y^q$, where $L$ this time is freely generated by $\x$ and $\y$ modulo the relation
\begin{align*}
\frac{p}{\ell}\x=\frac{q-1}{\ell}\y,
\end{align*}
where $\ell\leq d=\gcd(p,q-1)$ is again the index of $\Gamma$ in $\Gamma_{\w}$. Analogously to the loop case, we consider $S=\C[x,y]$ as an $L$-graded ring with $|x|=\x$ and $|y|=\y$, so that $\w$ is quasihomogeneous of degree $\vec{c}$, and write $R=S/(\w)$. Again, $R$ is a graded Gorenstein ring of Krull dimension one and Gorenstein parameter $\alpha=\x+\y-\vec{c}$.\\

Analogously to the loop case, we write $\w=yw_1\dots w_\ell$, where 
\begin{align*}
w_r=x^{\frac{p}{\ell}}-e^{\frac{\pi\sqrt{-1}}{\ell}}\eta^ry^{\frac{q-1}{\ell}}
\end{align*}
for $\eta$ a fixed primitive $\ell^{\text{th}}$ root of unity. With this, there are $\ell+1$ matrix factorisations coming from ($\Gamma$-equivariantly) factoring $\w$. These correspond to 
\begin{align*}
\obj{y}{}^\bullet = ( \cdots \rightarrow S(-\vec{c}) \xrightarrow{w} S(-\vec{y}) \xrightarrow{y} S \rightarrow \cdots ),
\end{align*}
as well as the matrix factorisations
\begin{align*}
\obj{{w_r}}{}^\bullet = ( \cdots \rightarrow S(-\vec{c}) \xrightarrow{\w/w_r} S(-\frac{p}{\ell}\x) \xrightarrow{w_r} S\rightarrow \cdots ).
\end{align*}
In addition, we also consider the objects
\begin{align*}
\obj{y}{j}&=\obj{y}{}((j+1-q)\y)
\end{align*} 
for $j=q-\frac{q-1}{\ell},\dots, q-1$. Similarly to the loop case, for $1\leq i\leq p-1$, $q-\frac{q-1}{\ell}\leq j\leq q-1$ and $k=\lfloor\frac{(i-1)\ell}{p}\rfloor$ we write 
\begin{align*}
I_{i,j}&= (x^i)+\sum_{t=1}^k(x^{i-t\frac{p}{\ell}}y^{j-(\ell-t)\frac{q-1}{\ell}})+ (y^{j-(\ell-k-1)\frac{q-1}{\ell}})\\
&= (x^i, x^{i-\frac{p}{\ell}}y^{j-(\ell-1)\frac{q-1}{\ell}}, \dots, x^{i-k\frac{p}{\ell}}y^{j-(\ell-k)\frac{q-1}{\ell}},y^{j-(\ell-k-1)\frac{q-1}{\ell}})
\end{align*} 
and the $L$-graded $R$-modules
\begin{align*}
R(i\x+(j+1)\y)/I_{i,j}.
\end{align*}
The corresponding rank $(k+2)$ matrix factorisation, which we denote by $\obj{0}{i,j}$, is given by stabilising this module beginning with 
\begin{align*}
R((j+1)\y)\oplus\bigoplus_{t=0}^{k-1}R(\vec{c})\oplus R(\vec{c}+i\x-(k+1)\frac{q-1}{\ell}\y)
\xrightarrow{\begin{pmatrix}
	x^i& \dots& y^{j-(\ell-k-1)\frac{q-1}{\ell}}
	\end{pmatrix}}
R(i\x+(j+1)\y).
\end{align*}

From this, it is straightforward to check that the maps defining the matrix factorisation are given in even degree by 
\begin{align}\label{ChainEvenDifferential}
\diff_0=\begin{pmatrix}
y^{j-(\ell-1)\frac{q-1}{\ell}} & 0  & \dots & 0 &x^{p-i}y\\
-x^{\frac{p}{\ell}}& y^{\frac{q-1}{\ell}} & \dots& 0 & 0\\
0 & -x^{\frac{p}{\ell}} & \dots &0 & 0\\
0 & 0 &   \dots & 0& 0\\
\vdots &\vdots   & \ddots& \vdots  & \vdots \\
0 & 0 & \hdots & y^{\frac{q-1}{\ell}} & 0\\
0 & 0 &  \hdots& -x^{i-k\frac{p}{\ell}} & y^{q-j+(\ell-k-1)\frac{q-1}{\ell}}\\
\end{pmatrix},
\end{align}
and in odd degree by $\diff_1=\mathrm{Adj}(\diff_0)$. Explicitly, we have that $\obj{0}{i,j}$ corresponds to the matrix factorisation 
\begin{center}
	\begin{tikzcd}[row sep=4ex, column sep=5ex]
	S(\vec{c}-\frac{q-1}{\ell}\y) \ar[d, phantom, description, "\bigoplus"]  \ar[dd, phantom, description, "\cdots\hskip7ex\phantom{\hskip20ex\cdots}"] &S((j+1)\y)  \ar[d, phantom, description, "\bigoplus"] & S(2\vec{c}-\frac{q-1}{\ell}\y)  \ar[d, phantom, description, "\bigoplus"] \ar[dd, phantom, description, "\cdots\hskip7ex\phantom{\hskip-40ex\cdots}"]
	\\ 
	\bigoplus_{t=0}^{k-1} S(\vec{c}-\frac{q-1}{\ell}\y) \ar[d, phantom, description, "\bigoplus"] \ar[r, shorten >=2ex, shorten <=2ex,"\diff_0"] &\bigoplus_{t=0}^{k-1}S(\vec{c}) \ar[d, phantom, description, "\bigoplus"] \ar[r, shorten >=2ex, shorten <=2ex,"\diff_1"] & \bigoplus_{t=0}^{k-1}S(2\vec{c}-\frac{q-1}{\ell}\y) \ar[d, phantom, description, "\bigoplus"]\\
	S((j+1)\y+i\x-\vec{c})&S(\vec{c}+i\x-(k+1)\frac{q-1}{\ell}\y)&S((j+1)\y+i\x)
	\end{tikzcd}
\end{center}
where, as before, the rightmost term is in cohomological degree 0, and the differentials go between the whole columns, not just the middle modules. \\

As in the maximally graded and loop cases, we are interested in a full subcategory $\mathcal{B}$ of $\mathrm{mf}(\C^2,\Gamma,\w)$ consisting of the objects described above. Namely, let $\mathcal{B}$ be the category consisting of the $\frac{p(q-1)}{\ell}+\ell$ objects 
\[
V=\{\obj{0}{i,j},\obj{y}{j}, \obj{w_1}{}[3],\dots,\obj{w_\ell}{}[3]\}_{i=1, \dots, p-1;\ j=q-\frac{q-1}{\ell}, \dots, q-1}.
\]
In the following sections we compute the morphisms between the objects in this category, culminating in the description of $\mathcal{B}$ as a quiver algebra in \cref{ChainEndAlgebra}.
\subsection{Morphisms between the $\obj{x}{}$'s, $\obj{y}{}$'s and $\obj{w_r}{}$'s}
For calculations regarding the modules $\obj{y}{}$ and $\obj{w_r}{}$, the arguments carry over from the maximally graded and loop cases with minimal alteration. 

\begin{lem}
	In $\mathrm{HMF}(\C^2, \Gamma, \w)$, we have the following:
	\begin{enumerate}[(i)]
		\item \label{ChainKIorthonormal} For any $j\in \Z$, the objects $\obj{y}{j},\dots, \obj{y}{j+\frac{q-1}{\ell}-1}$ are exceptional and pairwise orthogonal.
		\item \label{ChainKIIorthonormal} The objects $K_{w_{1}}, \dots, K_{w_{\ell}}$ are exceptional and pairwise orthogonal.
		\item \label{ChainKIKIIorthonormal}For each $j=q-\frac{q-1}{\ell},\dots, q-1$ and $r=1,\dots, \ell$, the objects $\obj{y}{j}$ and $\obj{w_r}{}$ are mutually orthogonal. \hfill\qed
	\end{enumerate}
\end{lem}
\subsection{Morphisms between the $\obj{w}{}$'s and $\obj{0}{}$'s}
\label{ChainBMorphisms4}
The fact that $\Hom^\bullet(\obj{w_r}{}, \obj{0}{i,j})=0$ is routine. In the other direction, we argue as in the loop case. Namely, we observe that 
\begin{align}\label{ChainARDuality}
\dim_\C \Hom^\bullet(\obj{w_r}{},\obj{0}{i,j}(\vec{c}-\x-\y))=
\begin{cases}
1\qquad\text{if }\bullet=-3\\
0\qquad\text{otherwise},
\end{cases}
\end{align}
and then find a  non-trivial element of $\Hom^3(\obj{0}{i,j}, \obj{w_r}{})$. This results in the following lemma, whose proof is easily adapted from that of \cref{KijKwrHoms}.
\begin{lem}\label{ChainKijKwrHoms}
	For each $r=1,\dots,\ell$, there is a single morphism between $\obj{0}{i,j}$ and $\obj{w_r}{}$ given by 
	\begin{align*}
	\Hom^{3}(\obj{0}{i,j},\obj{w_r}{})=\C\cdot\begin{pmatrix}
	y^{q-1-j}\\
	e^{-\frac{\pi i}{\ell}}\eta^{-r}\\
	(e^{-\frac{\pi i}{\ell}}\eta^{-r})^2\\
	\vdots\\
	(e^{-\frac{\pi i}{\ell}}\eta^{-r})^{k}\\
	(e^{-\frac{\pi i}{\ell}}\eta^{-r})^{k+1}x^{(k+1)\frac{p}{\ell}-i}
	\end{pmatrix}.
	\end{align*}\hfill\qed
\end{lem}
Computing morphisms $\obj{0}{i,j}\rightarrow\obj{0}{I,J}$ is analogous to the maximally graded case, and follows analogous arguments to that of \cref{BMorphisms5}. Namely, morphisms $\obj{0}{i,j}\rightarrow\obj{0}{I,J}$ are spanned by the module
\begin{align*}
\big(R/I_{I,J}\big)_{(I-i)\x+(J-j)\y}\ 
\end{align*} 
in degree zero. From this, it is immediate that there are no morphisms unless $I\geq i$. Analogously to the loop case, it is now possible to have $j>J$ since $\frac{p}{\ell}\x=\frac{q-1}{\ell}\y$.  Putting this together, we conclude:
\begin{lem}
	For all $i\in\{1, \dots, p-1\}$ and $j\in\{q-\frac{q-1}{\ell},\dots,q-1\}$, we have 
	
	\begin{flalign*}
	&&\Hom^{\bullet}(\obj{0}{i,j}, \obj{0}{I,J})\simeq &\begin{cases}
	\text{span}_\C\{x^{I-i}y^{J-j}, \dots, x^{(I-i)\bmod\frac{p}{\ell}}y^{J-j+k\frac{q-1}{\ell}}\}& \text{if } I\geq i,\ J\geq j\\
	& \text{and } \bullet=0\\
	\text{span}_\C\{x^{I-i-\frac{p}{\ell}}y^{J-j+\frac{q-1}{\ell}},\dots, x^{(I-i)\bmod\frac{p}{\ell}}y^{J-j+k\frac{q-1}{\ell}}\}& \text{if }\ J<j,\ I\geq i+\frac{p}{\ell}\\
	& \text{and }\bullet=0\\
	0&\text{otherwise.}\hfill\qed
	\end{cases} &
	\end{flalign*}
	
\end{lem}
\subsection{The total endomorphism algebra of the basic objects.}
As in the loop case, we are able to describe the endomorphism algebra of the category $\mathcal{B}$ explicitly as a quiver-with-relations. The proof follows from that of the maximally graded case with the analogous alterations required in \cref{LoopEndAlgebra}.
\begin{thm}
	\label{ChainEndAlgebra}
	The cohomology-level total endomorphism algebra of the objects of $\mathcal{B}$ is the algebra of the quiver-with-relations described in Figure \ref{figChainQuiver}, with all arrows living in degree zero.  In particular, $\mathcal{B}$ is a $\Z$-graded $A_\infty$-category concentrated in degree $0$, so is intrinsically formal.

	\begin{figure}[H]
	\centering
	\begin{tikzpicture}[blob/.style={circle, draw=black, fill=black, inner sep=0, minimum size=\blobsize}, arrow/.style={->, shorten >=6pt, shorten <=6pt}, scale = 0.9]
	\def\blobsize{1.2mm}

	\draw (3.5,0) node[blob]{};
	\draw (2.5,0) node[]{$\cdots$};
	\draw (1.5,0) node[blob]{};
	
	\draw (3.5,1) node[blob]{};
	\draw (2.5,1) node[]{$\cdots$};
	\draw (1.5,1) node[blob]{};
	
	\draw (3.5,3.5) node[blob]{};
	\draw (2.5,3.5) node[]{$\cdots$};
	\draw (1.5,3.5) node[blob]{};

	\draw[arrow] (1.5,0) -- (1.5,1);	
	\draw[arrow] (3.5,0) -- (3.5,1);	
	
	\draw[arrow] (1.5,1) -- (1.5,2);	
	\draw[arrow] (3.5,1) -- (3.5,2);

	\draw[arrow] (1.5,2.5) -- (1.5,3.5);	
	\draw[arrow] (3.5,2.5) -- (3.5,3.5);	
	
	\draw[arrow] (1.5,0) -- (2.4,0);
	\draw[arrow] (2.6,0) -- (3.5,0);
	
	\draw[arrow] (1.5,1) -- (2.4,1);
	\draw[arrow] (2.6,1) -- (3.5,1);
	
	\draw[arrow] (1.5,3.5) -- (2.4,3.5);
	\draw[arrow] (2.6,3.5) -- (3.5,3.5);
	
	\draw (1.5,2.4) node[]{$\vdots$};
	\draw (3.5,2.4) node[]{$\vdots$};
	\draw (2.5,2.4) node[]{$\iddots$};	
	
		\draw (5,0) node[blob]{};
	\draw (5,1) node[blob]{};
	\draw (5,3.5) node[blob]{};
	
		\draw[arrow] (3.5, 0) -- (5, 0) node[midway,below]{$a$};
	\draw[arrow] (3.5, 1) -- (5, 1) node[midway,below]{$a$};
	\draw[arrow] (3.5, 3.5) -- (5, 3.5) node[midway,below]{$a$};	
		\draw (5,-1) node{$\obj{y}{j}[3]$};
	\draw (6,5) node{$\obj{w_r}{}[3]$};
\draw (5, 2.37) node{$\vdots$};
	
	\draw [
	thick,
	decoration={
		brace,
		mirror,
		raise=0.5cm
	},
	decorate
	] (1.4,-.25) -- (3.6,-.25)
	node [pos=0.5,anchor=north,yshift=-0.55cm] {$\frac{p}{\ell}-1$};
	
			\draw [
thick,
decoration={
	brace,
	raise=0.5cm
},
decorate
] (5.25,3.6) -- (5.25,-.1)
node [pos=0.5,anchor=west,xshift=0.55cm] {$\frac{q-1}{\ell}$};
	
	\draw[line width=0.5mm, opacity=0.2] (1, -0.5) rectangle (4, 4);

		\begin{scope}[yshift=3.5cm]
	
	\draw[arrow] (3.5,0) -- ({3.5+2*cos(360/14)},{2*sin(360/14)} ) node[pos=0.6,above]{$c_\ell$};
	\draw[arrow] (3.5,0) -- ({3.5+2*cos(5*360/28)},{2*sin(5*360/28)} ) node[pos=0.6,right]{$c_1$};
	
	\end{scope}

	\begin{scope}[yshift=3.5cm, xshift=3.5cm]
	\foreach \a in {2,5}{
		\draw (\a*360/28: 2) node[blob]{};
	}
	\foreach \a in {1,2,3}{
		\draw[fill] (\a*135/14+360/14: 2) circle[radius=1pt];
	}
	
	\end{scope}

	\draw[line width=0.5mm, opacity=0.2] (4.5, -0.5) rectangle (5.5, 4);

	\begin{scope}[xshift=-3cm]
	\draw (3.5,0) node[blob]{};
	\draw (1.5,0) node[]{$\cdots$};
	\draw (2.5,0) node[blob]{};
	\draw (0.5,0) node[blob]{};
	
		\draw (3.5,1) node[blob]{};
	\draw (1.5,1) node[]{$\cdots$};
	\draw (2.5,1) node[blob]{};
	\draw (0.5,1) node[blob]{};
	
		\draw (3.5,3.5) node[blob]{};
	\draw (1.5,3.5) node[]{$\cdots$};
	\draw (2.5,3.5) node[blob]{};
	\draw (0.5,3.5) node[blob]{};

	\draw[arrow] (-0.4,0) -- (0.5,0);
	\draw[arrow] (-0.4,1) -- (0.5,1);
	\draw[arrow] (-0.4,3.5) -- (0.5,3.5);
	
	\draw[arrow] (0.5,0) -- (0.5,1);
	\draw[arrow] (2.5,0) -- (2.5,1);	
	\draw[arrow] (3.5,0) -- (3.5,1);

	\draw[arrow] (0.5,1) -- (0.5,2);
	\draw[arrow] (2.5,1) -- (2.5,2);	
	\draw[arrow] (3.5,1) -- (3.5,2);

	\draw[arrow] (0.5,2.5) -- (0.5,3.5);
	\draw[arrow] (2.5,2.5) -- (2.5,3.5);	
	\draw[arrow] (3.5,2.5) -- (3.5,3.5);	
	
	\draw[arrow] (0.5,0) -- (1.4,0);
	\draw[arrow] (1.6,0) -- (2.5,0);
	\draw[arrow] (2.5,0) -- (3.5,0);
	
	\draw[arrow] (0.5,1) -- (1.4,1);
	\draw[arrow] (1.6,1) -- (2.5,1);
	\draw[arrow] (2.5,1) -- (3.5,1);

	\draw[arrow] (0.5,3.5) -- (1.4,3.5);
	\draw[arrow] (1.6,3.5) -- (2.5,3.5);
	\draw[arrow] (2.5,3.5) -- (3.5,3.5);
	
	\draw[arrow] (3.5,0) -- (4.5,0);
\draw[arrow] (3.5,1) -- (4.5,1);
\draw[arrow] (3.5, 3.5) -- (4.5,3.5);
	
	\draw (0.5,2.4) node[]{$\vdots$};
	\draw (3.5,2.4) node[]{$\vdots$};
	\draw (2.5,2.4) node[]{$\vdots$};
	\draw (1.5,2.4) node[]{$\iddots$};	
	
	\draw[arrow] (0.5,3.5) -- (4.5,0);
	\draw[arrow] (1.5,3.5) -- (5.5,0);
	\draw[arrow] (2.5,3.5) -- (6.5,0);

	\draw[line width=0.5mm, opacity=0.2] (0, -0.5) rectangle (4, 4);

\draw [
thick,
decoration={
	brace,
	mirror,
	raise=0.5cm
},
decorate
] (0.4,-.25) -- (3.6,-.25)
node [pos=0.5,anchor=north,yshift=-0.55cm] {$\frac{p}{\ell}$};

		\begin{scope}[yshift=-2.5cm,xshift=-2cm]
	\draw (0,-.15) node{\parbox{300pt}{\small \ \textbf{Morphisms:} \begin{enumerate}[(i)]\item Horizontal arrows are\\ labelled by $x$, \item Vertical and diagonal arrows\\ are labelled by $y$. \end{enumerate}}};
	\end{scope}
	\end{scope}
	\begin{scope}[xshift=-7cm]
	\draw (0.5,0) node[]{$\cdots$};
	\draw (2,0) node[]{$\cdots$};
	\draw (3.5,0) node[]{$\cdots$};
	
	\draw (0.5,1) node[]{$\cdots$};
	\draw (2,1) node[]{$\cdots$};
	\draw (3.5,1) node[]{$\cdots$};
	
	\draw (0.5,2.4) node[]{$\vdots$};
	\draw (2,2.4) node[]{$\iddots$};
	\draw (3.5,2.4) node[]{$\vdots$};
	
	\draw (0.5,3.5) node[]{$\cdots$};
	\draw (2,3.5) node[]{$\cdots$};
	\draw (3.5,3.5) node[]{$\cdots$};
	
	\draw[arrow] (0.5,3.5) -- (4.5,0);
	\draw[arrow] (1.5,3.5) -- (5.5,0);
	\draw[arrow] (2.5,3.5) -- (6.5,0);
	\draw[arrow] (3.5,3.5) -- (7.5,0);
	
		\draw[line width=0.5mm, opacity=0.2] (0, -0.5) rectangle (4, 4);
	
	\end{scope}
	
	\begin{scope}[xshift=-11cm]
	\draw (3.5,0) node[blob]{};
	\draw (1.5,0) node[]{$\cdots$};
	\draw (2.5,0) node[blob]{};
	\draw (0.5,0) node[blob]{};
	
	\draw (3.5,1) node[blob]{};
	\draw (1.5,1) node[]{$\cdots$};
	\draw (2.5,1) node[blob]{};
	\draw (0.5,1) node[blob]{};
	
	\draw (3.5,3.5) node[blob]{};
	\draw (1.5,3.5) node[]{$\cdots$};
	\draw (2.5,3.5) node[blob]{};
	\draw (0.5,3.5) node[blob]{};

	\draw[arrow] (3.5,0) -- (4.4,0);
	\draw[arrow] (3.5,1) -- (4.4,1);
	\draw[arrow] (3.5,3.5) -- (4.4,3.5);
	
	\draw[arrow] (0.5,0) -- (0.5,1);
	\draw[arrow] (2.5,0) -- (2.5,1);	
	\draw[arrow] (3.5,0) -- (3.5,1);

	\draw[arrow] (0.5,1) -- (0.5,2);
	\draw[arrow] (2.5,1) -- (2.5,2);	
	\draw[arrow] (3.5,1) -- (3.5,2);

	\draw[arrow] (0.5,2.5) -- (0.5,3.5);
	\draw[arrow] (2.5,2.5) -- (2.5,3.5);	
	\draw[arrow] (3.5,2.5) -- (3.5,3.5);	
	
	\draw[arrow] (0.5,0) -- (1.4,0);
	\draw[arrow] (1.6,0) -- (2.5,0);
	\draw[arrow] (2.5,0) -- (3.5,0);
	
	\draw[arrow] (0.5,1) -- (1.4,1);
	\draw[arrow] (1.6,1) -- (2.5,1);
	\draw[arrow] (2.5,1) -- (3.5,1);
	
	\draw[arrow] (0.5,3.5) -- (1.4,3.5);
	\draw[arrow] (1.6,3.5) -- (2.5,3.5);
	\draw[arrow] (2.5,3.5) -- (3.5,3.5);
	
	\draw (0.5,2.4) node[]{$\vdots$};
	\draw (2.5,2.4) node[]{$\vdots$};
	\draw (3.5,2.4) node[]{$\vdots$};
	\draw (1.5,2.4) node[]{$\iddots$};	
	
	\draw[arrow] (0.5,3.5) -- (4.5,0);
	\draw[arrow] (1.5,3.5) -- (5.5,0);
	\draw[arrow] (2.5,3.5) -- (6.5,0);
	\draw[arrow] (3.5,3.5) -- (7.5,0);

	\draw[line width=0.5mm, opacity=0.2] (0, -0.5) rectangle (4, 4);
		\draw (2,4.5) node{$\obj{0}{i,j}$};
	
	\end{scope}

	\begin{scope}[yshift=-2.5cm,xshift=4cm]
	\draw (0,0) node{\parbox{300pt}{\small \ \textbf{Relations:} \begin{enumerate}[(i)]\item $xy=yx,$\item $ay=0$, \item $c_r(x^{\frac{p}{\ell}}-e^{\frac{\pi i}{\ell}}\eta^ry^{\frac{q-1}{\ell}})=0$.\end{enumerate}}};
	\end{scope}
	\end{tikzpicture}
	\caption{The quiver describing the category $\mathcal{B}$ for chain polynomials. There are $\ell-1$ blocks of size $\frac{p(q-1)}{\ell^2}$ and one of size $(\frac{p}{\ell}-1)\frac{q-1}{\ell}$. \label{figChainQuiver}}
\end{figure}	

\end{thm}
\subsection{Generation}
The last step in showing that the collection of objects in $\mathcal{B}$ generates $\mathrm{mf}(\C^2,\Gamma,\w)$. 
\begin{prop}
	\label{ChainBGenerates}
	The functor
	\begin{align*}
	\Tw \mathcal{B} \rightarrow \mathrm{mf}(\C^2, \Gamma, \w)
	\end{align*}
	is a quasi-equivalence.
\end{prop}
\begin{proof}
The strategy is the same as \cref{BGenerates}, so we will be brief, only explaining the points which differ from the loop and maximally graded chain cases. As before, we must prove that $R(l)/(x,y)\in\langle V\rangle $ for each $l\in L/\vec{c}\Z\simeq\Z/\left(\frac{pq}{\ell}\right)$, and we begin by  observing that $R(a\x+b\y)/(x,y)\in\langle V\rangle $ for $a=1,\dots, p-1$, $b=1,\dots,q+1-\frac{q-1}{\ell}$ by the same argument as in the loop and maximally graded cases. Furthermore, we see that, up to grading shifts by $[-2]$, this includes $R(l\y)/(x,y)$ for $l=1,\dots, (\ell-1)\frac{q-1}{\ell}$, as well as $R((a-\frac{p}{\ell})\x+\y)/(x,y)$ for $a=1,\dots,p-1$. From the $R(l\y)/(x,y)$, we build $R(l\y)/(y)$ for $l=1,\dots, (\ell-1)\frac{q-1}{\ell}$ analogously to the maximally graded case. Namely, the cone of $R((j+1)\y-\vec{c})/(y)\xrightarrow{x^{\frac{p}{\ell}}} R(j\vec{y})/(y)$ is $R(j\y)/(x^{\frac{p}{\ell}},y)$, and can be built by iterative cones from the previously constructed stabilisations of the origin. Together with the grading shifts of $\obj{y}{j}$, show that $R(l\y)/(y)\in\langle V\rangle$ for all $l\in\Z$ except for $l\equiv -\frac{q-1}{\ell} \bmod \vec{c}$. We build this remaining module analogously to $R((p-\frac{p-1}{\ell})\x+\y)/(xy)$ in the loop case. Namely, we can iteratively build $R((k-1)\frac{q-1}{\ell}\y)/(w_1\dots w_{k})$ for $k=2,\dots,\ell$, and observe that $R((\ell-1)\frac{q-1}{\ell}\y)/(w_1\dots w_{\ell})[1]=R(\vec{c}-\frac{q-1}{\ell}\y)/(y)$. Now that we have constructed modules $R(l\y)/(y)$ for all $l\in \Z$, we can produce $R(l\y)/(x,y)$ for any $l=1,\dots,q-1$ as in the maximally graded case. Then, all that is left to do is construct $R/(x,y)$ and $R(a\vec{x}+\vec{y})$ for $a=p-\frac{p}{\ell}+1,\dots, p-1$. The latter modules are constructed similarly to the maximally graded case by observing that the cone of the morphism 
\begin{align*}
\bigoplus_{t=0}^{\ell-1} R/(w)\xrightarrow{{\begin{pmatrix}
		x^i& \dots& x^{i-(\ell-1)\frac{p}{\ell}}y^{(\ell-1)\frac{q-1}{\ell}}
		\end{pmatrix}} } R(i\x)/(w)
\end{align*}
is $\obj{0}{i,q-1}$, and so $R(i\x)/(w)[1]=R(i\x+\y)/(y)\in\langle V\rangle$. From here, the proof proceeds as in the maximally graded case.

\end{proof}

We deduce the following corollary, whose proof follows from \cref{ChainBGenerates} and \cref{ChainEndAlgebra} in the same way as in the proof of \cite[Theorem 4.9]{HabermannSmith}:
\begin{cor}[\cref{TiltingCor}, undeformed chain polynomial case] The object 
	\begin{align*}
	\mathcal{E}:=\left(\bigoplus_{\substack{i=1,\dots, p-1\\
			j=q-\frac{q-1}{\ell},\dots, q-1}}\obj{0}{i,j}\right)\oplus\left(\bigoplus_{\substack{j=q-\frac{q-1}{\ell},\dots, q-1}}\obj{y}{j}[3]\right)\oplus\left(\bigoplus_{\substack{r=1,\dots, \ell}}\obj{w_r}{}[3]\right)
	\end{align*}
	is a tilting object for $\mathrm{mf}(\C^2,\Gamma,\w)$. 
\end{cor}

\section{Chain A-model}
\label{ChainAmodel}
In this section, we characterise the Fukaya--Seidel category of $\wt=\xt^p+\xt\yt^q$ with symmetry group $\widecheck{\overline{\Gamma}}$ not necessarily trivial. The content of \cref{CrepResolutionSection,EqMorsificationSection,NonExactFSSection} applies to this section unaltered, and the subsequent subsections essentially follow from applying the same alterations of the maximally graded case as was needed to study the loop A-model. 
\subsection{A resonant Morsification}
As in the maximally graded and loop cases, we introduce the Morsification $\wt_\varepsilon=\xt^p+\xt\yt^q-\varepsilon\xt\yt$, and observe that it descends to $\widecheck{\overline{\w}}_\varepsilon:X\rightarrow \C$ as 
\begin{align*}
\widecheck{\overline{\w}}_\varepsilon(u,v,w)=u^{\frac{p}{\ell}}+wv^{\frac{q-1}{\ell}}-\varepsilon w. 
\end{align*}
 Pulling this back to the chart $\widetilde{X}_i$, we get
\begin{align*}
\wttilde_{i}(\lambda_i,\mu_i)=\lambda_i^\frac{pi}{\ell}\mu_i^{\frac{(i-1)p}{\ell}}+\lambda_i^{\frac{(\ell-i)(q-1)}{\ell}+1}\mu_i^{\frac{(\ell+1-i)(q-1)}{\ell}+1}-\varepsilon\lambda_i\mu_i,
\end{align*}
where we continue to suppress the $\varepsilon$ from the notation when considering $\wttilde_\varepsilon:\widetilde{X}\rightarrow \C$ on charts. \\

The first thing to notice is that, when $p=\ell=2$ and $q=3$, $\wttilde_\varepsilon$ is only singular on one of the charts, where it is given by the loop polynomial with $p=q=2$. This observation actually constitutes part of the initial evidence of the conjecture of Futaki and Ueda, where they reasoned that the equivariant A--model should be derived equivalent to the maximally graded A--model of $\check{x}^2\check{y}+\check{x}\check{y}^2$ in this case. Whilst this turns out to be true, there is a subtlety in its proof. \\
In the maximally graded case of $\check{x}^2\check{y}+\check{x}\check{y}^2$, the smooth fibre is a thrice punctured torus. In the case of $\check{x}^2+\check{x}\check{y}^3$ with $\ell=2$, it is a \emph{twice} punctured torus. The reason being that, even though $\wttilde_\varepsilon:\widetilde{X}\rightarrow \C$ only has critical points on the chart $\widetilde{X}_2$, the smooth fibre contains the point $(-\delta,0)\in \widetilde{X}_1$ which is not visible in $\widetilde{X}_2$. Therefore, the difference in smooth fibres gives rise to the possibility that the Fukaya--Seidel category doesn't match that of the corresponding maximally graded loop polynomial. As it turns out, in this case, there are no compositions to compute -- the resulting quiver is just the $D_4$ quiver with no relations.\\
More generally, any chain polynomial $\check{x}^p+\check{x}\check{y}^{np+1}$ with $\ell=p$ will have all critical points visible in the charts $\widetilde{X}_2,\dots, \widetilde{X}_\ell$. Moreover, on these charts, the superpotential agrees with the superpotential of $\check{x}^p\check{y}+\check{x}\check{y}^{n(p-1)+1}$ with $\ell=p-1$. Whilst both Milnor fibres have the same genus, the former has two punctures and the latter has three. In these more general cases, there \emph{are} discs which contribute to products, with some of these discs in the chain case passing through the additional point in comparison to the loop case. For example, capping off the boundary component between the left and middle cylinders in \cref{fig:milnorfiberexample} results in the Milnor fibre for $\check{x}^3+\check{x}\check{y}^7$ -- see \cref{fig:compactifyingcomparisonexample} for a sketch of how the fibres above the origin compare. Similarly, comparing \cref{figLoopQuiver} and \cref{figChainQuiver} in these cases, we see that the quivers have the same shape, although the relations are different. \\
Fortunately, the strategy of proof for the chain cases where $\ell=p$ go through with essentially only superficial alteration.

\begin{figure}
	\centering
	\includegraphics[width=0.3\linewidth]{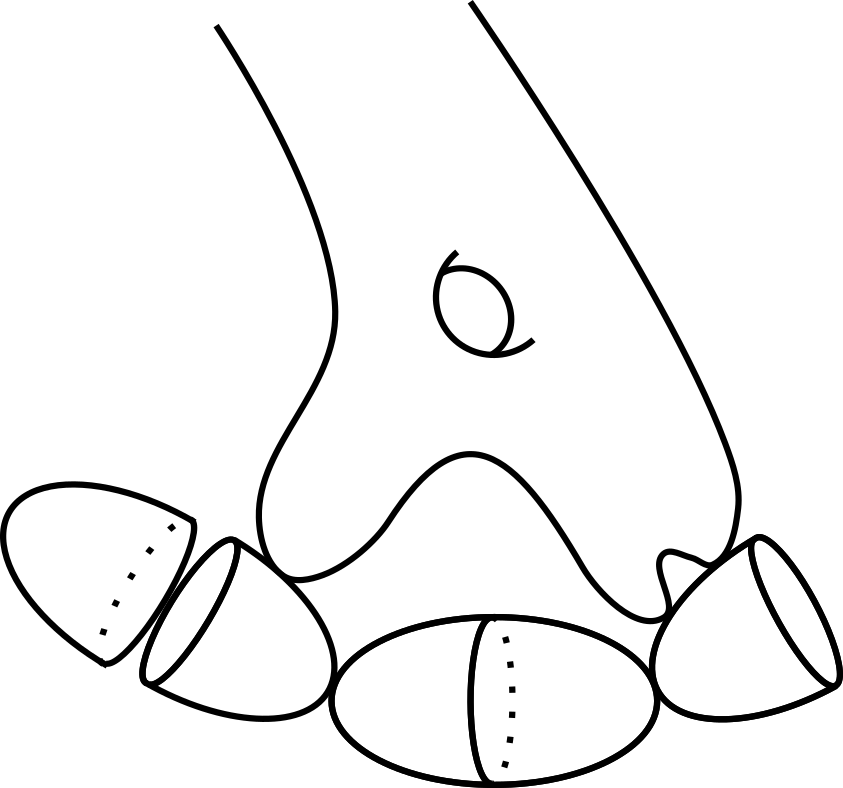}
	\caption{Sketch of fibre above the origin for $\check{x}^3+\check{x}\check{y}^7$.}
	\label{fig:compactifyingcomparisonexample}
\end{figure}

For chain polynomials with $p>\ell$, the critical points are grouped into the following three types:
\begin{enumerate}[(i)]
	\item \label{crit1chain} $\mu_\ell^{\frac{q-1}{\ell}}=\varepsilon$, $\lambda_\ell=0$
	\item \label{crit2chain} $\mu_i=\lambda_i=0$ for $i=1,\dots, \ell$
	\item \label{crit3chain} $\mu_\ell^{\frac{q-1}{\ell}}=\frac{\varepsilon}{q}$, $\lambda_\ell^{p-1}=\frac{\varepsilon(q-1)}{pq}\mu_\ell^{1-\frac{(\ell-1)p}{\ell}}$.
\end{enumerate}
For the case of $p=\ell$, there are only $\ell-1$ critical points of type (\ref{crit2chain}), although there is a critical point at $\lambda_1=0,\ \mu_1=\frac{1}{\varepsilon}$, which we notate as being of type (\ref{crit2chain})'. For the purposes of computation, this distinction is unimportant. In particular, in the parallel transport computations, the amount of winding in each neck region is not determined by where two irreducible components intersect, but rather the total change in argument of the vanishing path. We will therefore treat the cases of $p=\ell$ and $p\neq\ell$ together, notating the vanishing cycle corresponding to the critical point of type (ii)' as $\vc{\lambda_1\mu_1}{}$. \\
The critical values of type (i) and (ii) are both $0$, and the critical value corresponding to a critical point of type (iii) is 
\begin{align*}
\frac{-\varepsilon\mu_\ell\lambda_\ell(p-1)(q-1)}{pq}.
\end{align*}
Similarly to the loop case, there is a clear symmetry of these critical points. Namely, let $(\lambda_{\ell,\text{crit}}^+,\mu_{\ell,\text{crit}}^+)$ be the unique positive real critical point of type (\ref{crit3chain}) in the chart $\widetilde{X}_\ell$, with corresponding critical value $c_{\text{crit}}$. Letting $\zeta$ and $\eta$ denote the roots of unity
\begin{align}
\label{ChainRootsOfUnity}
\zeta = e^{2\pi i/(p-1)} \quad \text{and} \quad \eta = e^{2\pi i/(q-1)},
\end{align}
as well as $\alpha=e^{2\pi i/(p-1)(q-1)}$, we see that there is a $\mu_{p-1}\times\mu_{q-1}$ action on the critical points of type (\ref{crit3chain}) given by
\begin{align*}
\{(\zeta^{m}\eta^{n(1-\ell)}\alpha^n\lambda_{\ell,\text{crit}}^+, \eta^{n\ell}\mu_{\ell,\text{crit}}^+) : 0 \leq m \leq p-2\text{, } 0 \leq n \leq q-2\}.
\end{align*}
The critical value corresponding to $(\zeta^{m}\eta^{n(1-\ell)}\alpha^n\lambda_{\ell,\text{crit}}^+, \eta^{n\ell}\mu_{\ell,\text{crit}}^+)$ is $\alpha^{m(q-1)+np}c_\mathrm{crit}$, so there are $\frac{\gcd(p, q-1)}{\ell}$ critical points in each of these critical fibres. We therefore restrict to the subset 
\begin{align*}
\{(\zeta^{m}\eta^{n(1-\ell)}\alpha^n\lambda_{\ell,\text{crit}}^+, \eta^{n\ell}\mu_{\ell,\text{crit}}^+) : 0 \leq m \leq p-2\text{, } 0 \leq n \leq \frac{q-1}{\ell}-1\}
\end{align*}
in order to describe all of the critical points of type (\ref{crit3chain}) via symmetry.\\

The technical input and strategy for the chain case is identical to that of the loop case, with the only differences being superficial. We therefore keep this section brief, and only describe the alterations necessary in the chain case as they differ from the loop and maximally graded chain cases. \\

The vanishing paths corresponding to critical points (\ref{crit1chain}) and (\ref{crit2chain}) are still straight lines form $-\delta$ to the origin, and we denote the corresponding vanishing cycles by $\vc{\lambda_1\fact_1}{n}$ and $\vc{\lambda_i\mu_i}{}$. The preliminary vanishing paths $\gamma_{m,n}$ are given by the circular arc from $-\delta e^{-\theta}$ as $\theta$ increases from $0$ to 
\begin{align*}
\theta_{m,n}=2\pi\left(\frac{m}{p-1}+\frac{pn}{(p-1)(q-1)}\right),
\end{align*}
and then the straight line path from $-\alpha^{m(q-1)+np}\delta$ to $\alpha^{m(q-1)+np}c_{\mathrm{crit}}$. Again, we write $\vcpr{0}{m,n}$ for the corresponding vanishing cycles.
\subsection{The zero fibre and its smoothing}
 The fibre of $\wttilde_{\varepsilon}$ above the origin has $\ell+1$ components. Namely, it comprises: the line $\{\lambda_\ell=0\}$ in the chart $\widetilde{X}_{\ell}$, the complex lines $(0,\mu_i)$ and $(\lambda_{i+1},0)$ in the charts $\widetilde{X}_i$ and $\widetilde{X}_{i+1}$, respectively, which patch together to give $\ell-1$ projective lines in an $A_{\ell-1}$ configuration, and the curve given by  $\fact_i=\wttilde_{i}\lambda_i^{-1}\mu_i^{-1}$ in the charts $\widetilde{X}_{i}$ for $i=2,\dots, \ell$, and by $\fact_1=\wttilde_1\lambda_1^{-1}$ in $\widetilde{X}_1$. 
 
 \begin{figure}[H]
 	\centering
 	\scalebox{.5}{
 	\includegraphics[width=0.9\linewidth]{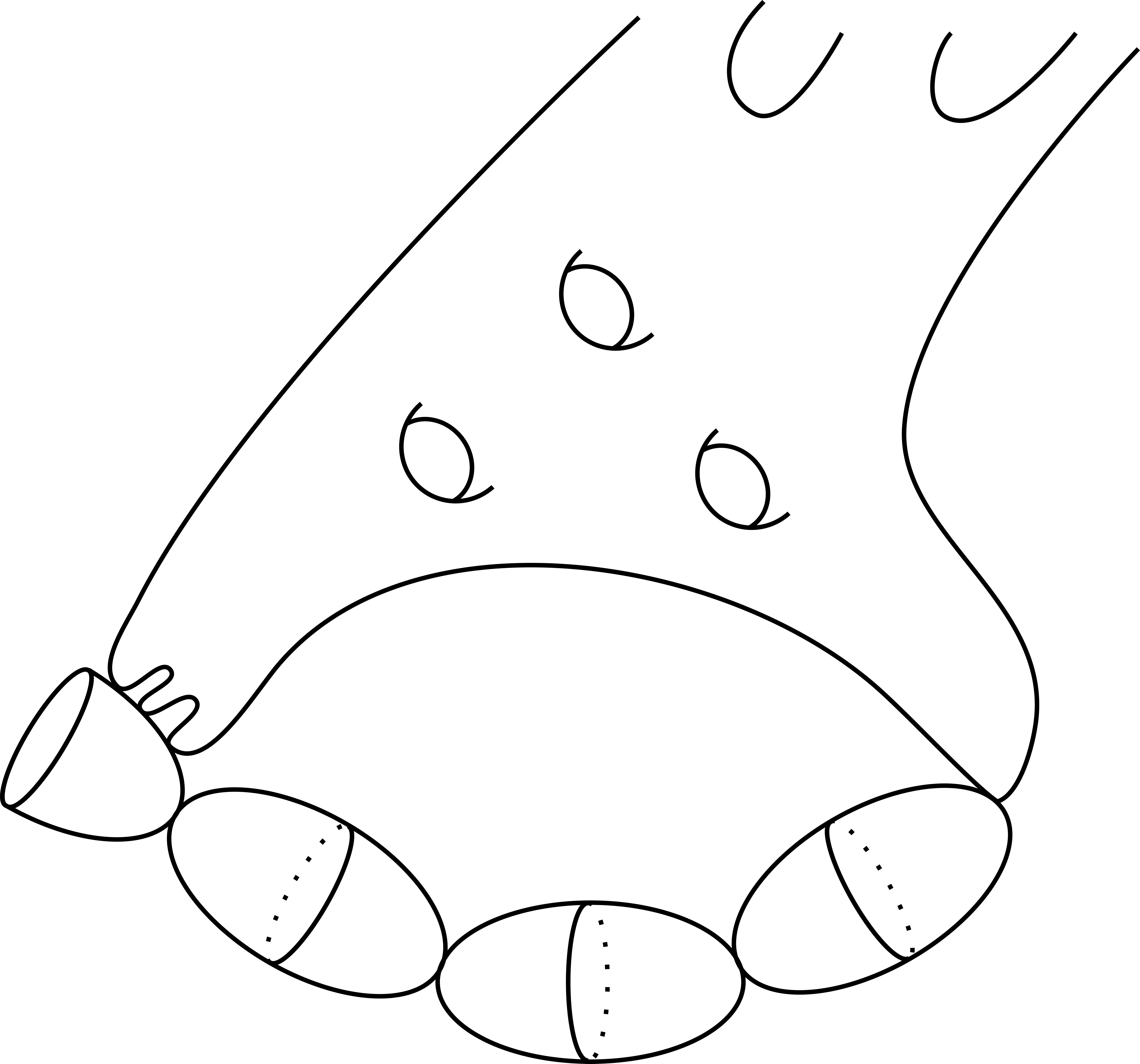}} 
 	\caption{A sketch of the fibre of $\wttilde_\varepsilon$ above the origin.}
 	\label{fig:chainsingmf}
 \end{figure}

\begin{rmk}
The topology of the smooth fibre was computed in \cite[Section 6.2]{HabStackyCurves}. Namely, it is a curve of genus 
\begin{align*}
g(\Sigma)=\frac{1}{2\ell}(pq-p+\ell-\gcd(\ell q,p+q-1))
\end{align*}
with $1+\gcd(q,\frac{p+q-1}{\ell})$ boundary punctures. 
\end{rmk}
\subsection{The vanishing cycles}
Just as in the loop case, we fist construct the real vanishing cycle $\vcpr{0}{0,0}$ and then construct the rest by a combination of symmetry considerations and parallel transport. In this case, the symplectomorphism 
$f_{m,n}^{(i)}$ is given by 
\begin{align*}
f_{m,n}^{(i)}:\widetilde{X}_i&\rightarrow \widetilde{X}_i\\
(\lambda_i,\mu_i)&\mapsto (\xi^{m(\ell+1-i)}\eta^{n(1-i)}\alpha^{n(\ell+1-i)}\lambda_i,\xi^{n(i-\ell)}\eta^{ni}\alpha^{n(i-\ell)}\mu_i). 
\end{align*}
Applying the parallel transport arguments of the loop case, we see that the $\mu_i$ coordinate of the Lagrangian $\vcpr{0}{m,n}$ interpolates between 
\begin{align*}
-2\pi\left(\frac{m(q-1)(\ell+1-i)+np(1-i)+n\ell}{(p-1)(q-1)}\right)\quad\text{and}\quad 2\pi\left(\frac{m(i-\ell)(q-1)+nip-n\ell}{(p-1)(q-1)}\right)
\end{align*}
in the neck region containing $\vc{\lambda_i\mu_i}{}$ as $|\mu_\ell|$ increases, whilst the argument $\mu_\ell-\eta^{n\ell}\mu_{\ell,\text{crit}}^+$ interpolates the other way in its argument about the neck region containing $\vc{\lambda_1\fact_1}{n}$. \\

As in the loop case, we modify the vanishing paths so that they are disjoint away from the distinguished point when the difference in the argument of the straight-line segments is less than $2\pi$. We then push the remaining paths to go outside of the critical points, and see that this does not affect the vanishing cycles. Correspondingly, we deduce the analogue of \cref{propAfromPrelim} in the chain case. We then Hamiltonian isotope the vanishing cycles in the distinguished fibre in the same way as in the loop case (although we only need to do this on the neck regions in $\widetilde{X}_\ell$ now), producing a collection of ordered and transversely intersecting Lagrangians with which we can compute $\mathcal{A}_{\widecheck{\overline{\Gamma}}}$. 
\subsection{Composition and gradings}
As in the loop case, the Floer complexes can be taken to all be graded in degree zero, and this is proven in the same way. With this grading, we have the proof of the main theorem in the chain case. 
\begin{thm}[\cref{mainTheorem}, undeformed chain polynomial case]
Under the correspondence 
	\begin{equation}
\label{ChainObjectMatch}
\begin{aligned}
\vc{0}{m,n} &\leftrightarrow \obj{0}{i,j}
\\ \vc{\lambda_1\fact}{n} &\leftrightarrow \obj{y}{j}[3]
\\ \vc{\lambda_k\mu_k}{} &\leftrightarrow \obj{w_r}{}[3]
\end{aligned}
\quad \text{with} \quad
\begin{aligned}
i+m&=p-1
\\ j+n&=q-1\\
k+r&=\ell
\end{aligned}
\end{equation}
the $\Z$-graded $A_\infty$-category $\mathcal{A}_{\widecheck{\overline{\Gamma}}}$ is identified with the quiver algebra of \cref{ChainEndAlgebra} and is formal, so there is a quasi-equivalence 
\begin{align*}
\mf(\C^2,\Gamma,\w)\simeq \mathcal{FS}(\wttilde)
\end{align*}
\end{thm}
\begin{proof}
The only thing to check is that the morphisms in $\mathcal{A}_{\widecheck{\overline{\Gamma}}}$ compose in the claimed way, but this follows by the same reasoning as in the loop case.
\end{proof}
\section{Brieskorn--Pham polynomials}
\label{BPPolynomials}
\subsection{The B-model} 
We now deal with the last, and most simple, of invertible polynomials in two variables. Namely, we let $\w=x^p+y^q$, and consider the $L$-graded rings $S=\C[x,y]$ and $R=S/(\w)$, where $L$ is generated by $|x|=\x$ and $|y|=\y$ modulo the relations
\begin{align*}
\frac{p}{\ell}\x=\frac{q}{\ell}\y,
\end{align*}
where $\ell\leq d=\gcd(p,q)$ is the index of $\Gamma$ in $\Gamma_{\w}$ so that $L\simeq\Z\oplus\Z/\left(\frac{d}{\ell}\right)$ and $\w$ is homogeneous of degree $\vec{c}=p\x=q\y$. Similarly to the loop case, we assume without loss of generality that $p\geq q\geq2$. Moreover, we have $L/\vec{c}\Z\simeq\Z/\left(\frac{pq}{d}\right)\times \Z/\left(\frac{d}{\ell}\right)$. Similarly to the previous cases, we write $\w=w_1\dots w_\ell$, where
\begin{align*}
w_r=x^{\frac{p}{\ell}}-e^{\frac{\pi\sqrt{-1}}{\ell}}\eta^ry^{\frac{q}{\ell}}
\end{align*}
for $\eta$ a primitive $\ell^{\text{th}}$ root of unity. With this, there are $\ell$ matrix factorisations coming from equivariantly factoring $\w$ 
\begin{align*}
\obj{{w_r}}{}^\bullet = ( \cdots \rightarrow S(-\vec{c}) \xrightarrow{\w/w_r} S(-\frac{p}{\ell}\x) \xrightarrow{w_r} S\rightarrow \cdots )
\end{align*}
In the maximally graded case, this is the zero object in the category of matrix factorisations, but in the non-maximally graded case these give non-trivial objects. We will therefore assume from now on that $\ell>1$, since the two cases are treated differently and the maximally graded case is already established as a special case of \cite{FutakiUedaBP}. \\

The modules supported at the origin are the analogues of those in the loop and chain cases. Namely, we write $k=\lfloor\frac{(j-1)\ell}{q}\rfloor$, and for $i=(\ell-1)\frac{p}{\ell}+1,\dots, p-1$ and $j=1,\dots, q-1$, consider the ideals
\begin{align*}
I_{i,j}&=(x^{i-(\ell-k-1)\frac{p}{\ell}})+\sum_{t=1}^k(x^{i-(\ell-k-1+t)\frac{p}{\ell}}y^{j-(k+1-t)\frac{q}{\ell}})+(y^j)\\
&=(x^{i-(\ell-k-1)\frac{p}{\ell}},x^{i-(\ell-k)\frac{p}{\ell}}y^{j-k\frac{q}{\ell}},\dots,  x^{i-(\ell-1)\frac{p}{\ell}}y^{j-\frac{q}{\ell}},y^j).
\end{align*}
In addition, we need to consider the ideals
\begin{align*}
I_{(\ell-1)\frac{p}{\ell},j}&=(x^{k\frac{p}{\ell}})+\sum_{t=1}^{k-1}(x^{(k-t)\frac{p}{\ell}}y^{j-(k+1-t)\frac{q}{\ell}})+(y^{j-\frac{q}{\ell}})\\
&=(x^{k\frac{p}{\ell}},x^{(k-1)\frac{p}{\ell}}y^{j-k\frac{q}{\ell}},\dots ,x^{\frac{p}{\ell}}y^{j-2\frac{q}{\ell}},y^{j-\frac{q}{\ell}})
\end{align*}
for $i=(\ell-1)\frac{p}{\ell}$ and $j=\frac{q}{\ell}+1,\dots, q-1$. In both cases, we let $\obj{0}{i,j}$ be the matrix factorisation corresponding to the $L$-graded $R$-module
\begin{align*}
R(i\x+j\y)/(I_{i,j}).
\end{align*}
For the modules with $i>(\ell-1)\frac{p}{\ell}$, this is given by the rank $(k+2)$ matrix factorisation
 \begin{center}
	\begin{tikzcd}[row sep=4ex, column sep=5ex]
	S(\vec{c}-\frac{p}{\ell}\x) \ar[d, phantom, description, "\bigoplus"]  \ar[dd, phantom, description, "\cdots\hskip7ex\phantom{\hskip20ex\cdots}"] &S(\vec{c}-(k+1)\frac{p}{\ell}\x+j\y)  \ar[d, phantom, description, "\bigoplus"] & S(2\vec{c}-\frac{p}{\ell}\x)  \ar[d, phantom, description, "\bigoplus"] \ar[dd, phantom, description, "\cdots\hskip7ex\phantom{\hskip-40ex\cdots}"]
	\\ 
	\bigoplus_{t=0}^{k-1} S(\vec{c}-\frac{p}{\ell}\x) \ar[d, phantom, description, "\bigoplus"] \ar[r, shorten >=2ex, shorten <=2ex,"\diff_0"] &\bigoplus_{t=0}^{k-1}S(\vec{c}) \ar[d, phantom, description, "\bigoplus"] \ar[r, shorten >=2ex, shorten <=2ex,"\diff_1"] & \bigoplus_{t=0}^{k-1}S(2\vec{c}-\frac{p}{\ell}\x) \ar[d, phantom, description, "\bigoplus"]\\
	S(i\x+j\y-\vec{c})&S(i\x)&S(i\x+j\y)
	\end{tikzcd}
\end{center}
where 

\begin{align}\label{BPEvenDifferential}
\diff_0=\begin{pmatrix}
y^{j-k\frac{q}{\ell}} & 0 & 0  & \dots & 0 &x^{p-i+\frac{p}{\ell}(\ell-1-k)}\\
-x^{\frac{p}{\ell}}& y^{\frac{q}{\ell}} & 0 &\dots& 0 & 0\\
0 & -x^{\frac{p}{\ell}} & y^{\frac{q}{\ell}} & \dots &0 & 0\\
0 & 0 & -x^{\frac{p}{\ell}}&  \dots & 0& 0\\
\vdots &\vdots & \vdots   & \ddots& \vdots  & \vdots \\
0 & 0 & 0  & \hdots & y^{\frac{q}{\ell}} & 0\\
0 & 0 & 0  & \hdots& -x^{i-(\ell-1)\frac{p}{\ell}} & y^{q-j}\\
\end{pmatrix},
\end{align}
$\diff_1=\mathrm{Adj}(\diff_0)$, and the differentials go between the whole rows, not just the middle modules. The matrix factorisations for $\obj{0}{(\ell-1)\frac{p}{\ell},j}$ are similar, although are of rank $k+1$. Analogously to the loop and chain cases, we consider the category $\mathcal{B}$ whose objects are $R/(w_r)[3]$ for $r=1,\dots, \ell$ and $\obj{0}{i,j}$ above. Arguing as before, we arrive at: 

\begin{thm}
	\label{BPEndAlgebra}
	The cohomology-level total endomorphism algebra of the objects of $\mathcal{B}$ is the algebra of the quiver-with-relations described in Figure \ref{figBPQuiver}, with all arrows living in degree zero.  In particular, $\mathcal{B}$ is a $\Z$-graded $A_\infty$-category concentrated in degree $0$, so is intrinsically formal.
	\begin{figure}[ht]
	\centering
	\begin{tikzpicture}[blob/.style={circle, draw=black, fill=black, inner sep=0, minimum size=\blobsize}, arrow/.style={->, shorten >=6pt, shorten <=6pt}, scale = 0.9]
	\def\blobsize{1.2mm}
	
	\draw (3.5,0) node[blob]{};
	\draw (2.5,0) node[]{$\cdots$};
	\draw (1.5,0) node[blob]{};

	\draw (3.5,1) node[blob]{};
	\draw (2.5,1) node[]{$\cdots$};
	\draw (1.5,1) node[blob]{};

	\draw (3.5,3.5) node[blob]{};
	\draw (2.5,3.5) node[]{$\cdots$};
	\draw (1.5,3.5) node[blob]{};

	\draw[arrow] (-0.4,0) -- (0.5,0);
	\draw[arrow] (-0.4,1) -- (0.5,1);
	\draw[arrow] (-0.4,3.5) -- (0.5,3.5);

	\draw[arrow] (1.5,0) -- (1.5,1);	
	\draw[arrow] (3.5,0) -- (3.5,1);

	\draw[arrow] (1.5,1) -- (1.5,2);	
	\draw[arrow] (3.5,1) -- (3.5,2);

	\draw[arrow] (1.5,2.5) -- (1.5,3.5);	
	\draw[arrow] (3.5,2.5) -- (3.5,3.5);

	\draw[arrow] (1.5,0) -- (2.4,0);
	\draw[arrow] (2.6,0) -- (3.5,0);

	\draw[arrow] (1.5,1) -- (2.4,1);
	\draw[arrow] (2.6,1) -- (3.5,1);

	\draw[arrow] (1.5,3.5) -- (2.4,3.5);
	\draw[arrow] (2.6,3.5) -- (3.5,3.5);
	
	\draw (1.5,2.4) node[]{$\vdots$};
	\draw (3.5,2.4) node[]{$\vdots$};
	\draw (2.5,2.4) node[]{$\iddots$};

	\draw (5,5.5) node{$\obj{w_r}{}[3]$};
	
	\draw [
	thick,
	decoration={
		brace,
		mirror,
		raise=0.5cm
	},
	decorate
	] (1.4,-.25) -- (3.6,-.25)
	node [pos=0.5,anchor=north,yshift=-0.55cm] {$\frac{q}{\ell}-1$};

	\begin{scope}[yshift=3.5cm]

\draw[arrow] (3.5,0) -- ({3.5+2*cos(360/14)},{2*sin(360/14)} ) node[pos=0.6,above]{$c_\ell$};
\draw[arrow] (3.5,0) -- ({3.5+2*cos(5*360/28)},{2*sin(5*360/28)} ) node[pos=0.6,right]{$c_1$};

\end{scope}

\begin{scope}[yshift=3.5cm, xshift=3.5cm]
\foreach \a in {2,5}{
	\draw (\a*360/28: 2) node[blob]{};
}
\foreach \a in {1,2,3}{
	\draw[fill] (\a*135/14+360/14: 2) circle[radius=1pt];
}

\end{scope}

	\draw[line width=0.5mm, opacity=0.2] (1, -0.5) rectangle (4, 4);

	\begin{scope}[xshift=-7cm]
	\draw (0.5,0) node[]{$\cdots$};
	\draw (2,0) node[]{$\cdots$};
	\draw (3.5,0) node[]{$\cdots$};
	
	\draw (0.5,1) node[]{$\cdots$};
	\draw (2,1) node[]{$\cdots$};
	\draw (3.5,1) node[]{$\cdots$};
	
	\draw (0.5,2.4) node[]{$\vdots$};
	\draw (2,2.4) node[]{$\iddots$};
	\draw (3.5,2.4) node[]{$\vdots$};
	
	\draw (0.5,3.5) node[]{$\cdots$};
	\draw (2,3.5) node[]{$\cdots$};
	\draw (3.5,3.5) node[]{$\cdots$};
	
	\draw[arrow] (0.5,3.5) -- (4.5,0);
	\draw[arrow] (1.5,3.5) -- (5.5,0);
	\draw[arrow] (2.5,3.5) -- (6.5,0);
	\draw[arrow] (3.5,3.5) -- (7.5,0);
	
		\draw[arrow] (3.6,0) -- (4.5,0);
	\draw[arrow] (3.6,1) -- (4.5,1);
	\draw[arrow] (3.6, 3.5) -- (4.5,3.5);
	
	\draw[line width=0.5mm, opacity=0.2] (0, -0.5) rectangle (4, 4);

	\end{scope}
	
	\begin{scope}[xshift=-3cm]
	\draw (3.5,0) node[blob]{};
	\draw (2.5,0) node[]{$\cdots$};
	\draw (1.5,0) node[blob]{};
	\draw (0.5,0) node[blob]{};
	
	\draw (3.5,1) node[blob]{};
	\draw (2.5,1) node[]{$\cdots$};
	\draw (1.5,1) node[blob]{};
	\draw (0.5,1) node[blob]{};
	
	\draw (3.5,3.5) node[blob]{};
	\draw (2.5,3.5) node[]{$\cdots$};
	\draw (1.5,3.5) node[blob]{};
	\draw (0.5,3.5) node[blob]{};
	
	\draw[arrow] (3.5,0) -- (4.5,0);
	\draw[arrow] (3.5,1) -- (4.5,1);
	\draw[arrow] (3.5, 3.5) -- (4.5,3.5);

	\draw[arrow] (0.5,0) -- (0.5,1);
	\draw[arrow] (1.5,0) -- (1.5,1);	
	\draw[arrow] (3.5,0) -- (3.5,1);	
	
	\draw[arrow] (0.5,1) -- (0.5,2);
	\draw[arrow] (1.5,1) -- (1.5,2);	
	\draw[arrow] (3.5,1) -- (3.5,2);

	\draw[arrow] (0.5,2.5) -- (0.5,3.5);
	\draw[arrow] (1.5,2.5) -- (1.5,3.5);	
	\draw[arrow] (3.5,2.5) -- (3.5,3.5);	
	
	\draw[arrow] (0.5,0) -- (1.5,0);
	\draw[arrow] (1.5,0) -- (2.4,0);
	\draw[arrow] (2.6,0) -- (3.5,0);
	
	\draw[arrow] (0.5,1) -- (1.5,1);
	\draw[arrow] (1.5,1) -- (2.4,1);
	\draw[arrow] (2.6,1) -- (3.5,1);
	
	\draw[arrow] (0.5,3.5) -- (1.5,3.5);
	\draw[arrow] (1.5,3.5) -- (2.4,3.5);
	\draw[arrow] (2.6,3.5) -- (3.5,3.5);
	
	\draw (0.5,2.4) node[]{$\vdots$};
	\draw (1.5,2.4) node[]{$\vdots$};
	\draw (3.5,2.4) node[]{$\vdots$};
	\draw (2.5,2.4) node[]{$\iddots$};	
	
	\draw[arrow] (0.5,3.5) -- (4.5,0);
	\draw[arrow] (1.5,3.5) -- (5.5,0);
	\draw[arrow] (2.5,3.5) -- (6.5,0);

	\draw[line width=0.5mm, opacity=0.2] (0, -0.5) rectangle (4, 4);
	\draw [
	thick,
	decoration={
		brace,
		mirror,
		raise=0.5cm
	},
	decorate
	] (0.4,-.25) -- (3.6,-.25)
	node [pos=0.5,anchor=north,yshift=-0.55cm] {$\frac{q}{\ell}$};
	
	\end{scope}

		\begin{scope}[xshift=-11cm]
	
	\draw (3.5,1) node[blob]{};
	\draw (2.5,1) node[]{$\cdots$};
	\draw (1.5,1) node[blob]{};
	\draw (0.5,1) node[blob]{};
	
	\draw (3.5,3.5) node[blob]{};
	\draw (2.5,3.5) node[]{$\cdots$};
	\draw (1.5,3.5) node[blob]{};
	\draw (0.5,3.5) node[blob]{};
	
	\draw[arrow] (3.5,1) -- (4.4,1);
	\draw[arrow] (3.5, 3.5) -- (4.4,3.5);

	\draw[arrow] (0.5,1) -- (0.5,2);
	\draw[arrow] (1.5,1) -- (1.5,2);	
	\draw[arrow] (3.5,1) -- (3.5,2);

	\draw[arrow] (0.5,2.5) -- (0.5,3.5);
	\draw[arrow] (1.5,2.5) -- (1.5,3.5);	
	\draw[arrow] (3.5,2.5) -- (3.5,3.5);

	\draw[arrow] (0.5,1) -- (1.5,1);
	\draw[arrow] (1.5,1) -- (2.4,1);
	\draw[arrow] (2.6,1) -- (3.5,1);
	
	\draw[arrow] (0.5,3.5) -- (1.5,3.5);
	\draw[arrow] (1.5,3.5) -- (2.4,3.5);
	\draw[arrow] (2.6,3.5) -- (3.5,3.5);
	
	\draw (0.5,2.4) node[]{$\vdots$};
	\draw (1.5,2.4) node[]{$\vdots$};
	\draw (3.5,2.4) node[]{$\vdots$};
	\draw (2.5,2.4) node[]{$\iddots$};	
	
	\draw[arrow] (0.5,3.5) -- (4.5,0);
	\draw[arrow] (1.5,3.5) -- (5.5,0);
	\draw[arrow] (2.5,3.5) -- (6.5,0);
	\draw[arrow] (3.5,3.5) -- (7.5,0);
	
	\draw[line width=0.5mm, opacity=0.2] (0, 0.5) rectangle (4, 4);
	\draw [
	thick,
	decoration={
		brace,
		raise=0.5cm
	},
	decorate
	] (0.25,0.9) -- (0.25,3.6)
	node [pos=0.5,anchor=east,xshift=-0.55cm] {$\frac{p}{\ell}-1$};
		\draw (2,4.5) node{$\obj{0}{i,j}$};
	\end{scope}

	\begin{scope}[yshift=-2.5cm,xshift=-2cm]
	\draw (0,-.15) node{\parbox{300pt}{\small \ \textbf{Morphisms:} \begin{enumerate}[(i)]\item Horizontal arrows are\\ labelled by $y$, \item Vertical and diagonal arrows\\ are labelled by $x$. \end{enumerate}}};
	\end{scope}
	
	\begin{scope}[yshift=-2.5cm,xshift=4cm]
	\draw (0,0) node{\parbox{300pt}{\small \ \textbf{Relations:} \begin{enumerate}[(i)]\item $xy=yx,$\item $c_r(x^{\frac{p}{\ell}}-e^{\frac{\pi i}{\ell}}\eta^ry^{\frac{q}{\ell}})=0$.\end{enumerate}}};
	\end{scope}
	\end{tikzpicture}
	\caption{The quiver describing the category $\mathcal{B}$ for Brieskorn--Pham polynomials. There are $\ell-2$ blocks of size $\frac{p}{\ell}\times\frac{q}{\ell}$, one of size $(\frac{p}{\ell}-1)\times\frac{q}{\ell}$, and one of size $\frac{p}{\ell}\times(\frac{q}{\ell}-1)$ \label{figBPQuiver}}
\end{figure}

\end{thm}
As before, we claim that the collection of objects in $\mathcal{B}$ generates $\mathrm{mf}(\C^2,\Gamma,\w)$; the proof follows the same strategy of Propositions \ref{BGenerates} and \ref{ChainBGenerates}. 
\begin{prop}
	\label{BPBGenerates}
	The functor
	\begin{align*}
	\Tw \mathcal{B} \rightarrow \mathrm{mf}(\C^2, \Gamma, \w)
	\end{align*}
	is a quasi-equivalence.\hfill\qed
\end{prop}

From this, we deduce the \cref{TiltingCor} in the undeformed Brieskorn--Pham case in the same way as in the loop and chain cases. We state the corollary in the case of $\ell>1$; the case of $\ell=1$ is well known, and goes back to at least \cite[Theorem 1.2]{FutakiUedaBP}, \cite[Theorem 6]{FutakiUedaProceedings}.
\begin{cor}[\cref{TiltingCor}, Undeformed Brieskorn--Pham polynomial case] For $\ell>1$, the object 
	\begin{align*}
	\mathcal{E}:=\left(\bigoplus_{\substack{i=(\ell-1)\frac{p}{\ell},\dots, p-1\\
			j=2,\dots, q-1}}\obj{0}{i,j}\right)\oplus\left(\bigoplus_{i=(\ell-1)\frac{p}{\ell}+\frac{q}{\ell},\dots, p-1}\obj{0}{i,j}\right)\oplus\left(\bigoplus_{\substack{r=1,\dots, \ell}}\obj{w_r}{}[3]\right)
	\end{align*}
	is a tilting object for $\mathrm{mf}(\C^2,\Gamma,\w)$.
\end{cor}

\subsection{The A-model}\label{BPAmodel}
The computation of the A-model for Brieskorn--Pham polynomials is, by this point, routine. The argument follows the previous two cases, and so we only summarise the results here. For a Brieskorn--Pham polynomial $\wt=\xt^p+\yt^q$, we have $d=\gcd(p,q)$, and we again Morsify as $\wt_\varepsilon=\xt^p+\yt^q-\varepsilon\xt\yt$. This descends to $\C^2/\mu_{\ell}$ as $\widecheck{\overline{\w}}=u^\frac{p}{\ell}+v^\frac{q}{\ell}-\varepsilon w$, and pulls back to the chart $\widetilde{X}_i$ as 
\begin{align*}
{\wttilde}_i(\lambda_i,\mu_i)=\lambda_i^{\frac{pi}{\ell}}\mu_i^{\frac{(i-1)p}{\ell}}+\lambda_i^{\frac{(\ell-i)q}{\ell}}\mu_i^{\frac{(\ell+1-i)q}{\ell}}-\varepsilon\lambda_i\mu_i. 
\end{align*}
The critical values fall into the groups
\begin{enumerate}[(i)]
	\item $\mu_i=\lambda_i=0$ for $\begin{cases}
		i=1,\dots, \ell\ \text{if } q>\ell,\\
		i=1,\dots, \ell-1\ \text{if } p>q=\ell,\\
		i=2,\dots, \ell-1\ \text{if } p=q=\ell,\\
	\end{cases}$
	\item[(i)'] $\mu_\ell=0,\ \lambda_\ell=\frac{1}{\varepsilon}$ if $p>q=\ell$
	\item[(i)''] $\mu_\ell=0,\ \lambda_\ell=\frac{1}{\varepsilon}$ and $\lambda_1=0,\ \mu_1=\frac{1}{\varepsilon}$ if $p=q=\ell$
	\item $\mu_1^{q-1}=\frac{\varepsilon}{q}\lambda_1^{1-\frac{(\ell-1)q}{\ell}}$, $\lambda_1^{\frac{p}{\ell}-1}=\frac{\varepsilon\mu_1}{p}$.
\end{enumerate}
\begin{rmk}
Analogously to the chain case, the Brieskorn--Pham cases with $\ell=q$ have critical points at the origin in fewer than $\ell$ charts. Indeed, one can check that the equivariant superpotential of $\check{x}^{n\ell}+\check{y}^\ell$ matches that of $\check{x}^{n(\ell-1)}+\check{x}\check{y}^{\ell}$ on charts where both superpotentials have critical points at the origin; however, as in the chain case, the resulting Fukaya--Seidel categories are not equivalent unless $\ell=2$ or $\ell=3,n=1$ due to the difference in topology of the smooth fibre. In the case of $\ell=q=3$ and $n=1$, this then further reduces to the maximally graded polynomial $\check{x}^2\check{y}+\check{x}\check{y}^2$. 
\end{rmk}
Analogously to chain polynomials with $\ell=p$, the analysis of the cases where critical points of type (i)' or (i)'' exist is only superficially different to the cases where $q>\ell$. We therefore label the vanishing cycles corresponding to the critical point of type (i)' as $\vc{\lambda_\ell\mu_\ell}{}$, or the critical points of type (ii)' as $\vc{\lambda_\ell\mu_\ell}{}$ and $\vc{\lambda_1\mu_1}{}$, respectively.\\
The critical value corresponding to critical points in the first three groups is zero, whist the $\frac{(p-1)(q-1)-1}{\ell}$ critical points in the second group have critical value
\begin{align*}
\frac{-\varepsilon\mu_1\lambda_1((p-1)(q-1)-1)}{pq}.
\end{align*}

As  before, we denote $(\lambda_{1,\text{crit}}^+, \mu_{1,\text{crit}}^+)$ the unique positive real critical point whose critical value is $c_\text{crit}$. The other critical points are then attained by symmetry considerations. Namely, for 
\begin{align}\label{BPSymmetrySet}
(m,n)\in(\{0,\dots, p-2\}\times\{0,\dots, q-2\})\setminus\{(p-2,q-2)\}, 
\end{align}
we have that the other critical points of type (iv) are given in the chart $\widetilde{X}_1$ by
\begin{align}\label{BPAction}
(\alpha^{n+m(q-1)}\lambda_{1,\text{crit}}^+, \alpha^{n(\frac{p}{\ell}-1)+m-\frac{m(\ell-1)q}{\ell}}\mu_{1,\text{crit}}^+),
\end{align}
where $\alpha=e^{\frac{2\pi i\ell}{pq-p-q}}$. Of course, as in the other cases, there is redundancy in considering the full set in \eqref{BPSymmetrySet}, and so we restrict to $m=0,\dots,\frac{p}{\ell}-2$ and $n=0,\dots,q-2$, as well as $m=\frac{p}{\ell}-1$, $n=0,\dots, q-2-\frac{q}{\ell}$.\\

\begin{figure}[H]
	\centering
	\scalebox{.5}{
		\includegraphics[width=0.9\linewidth]{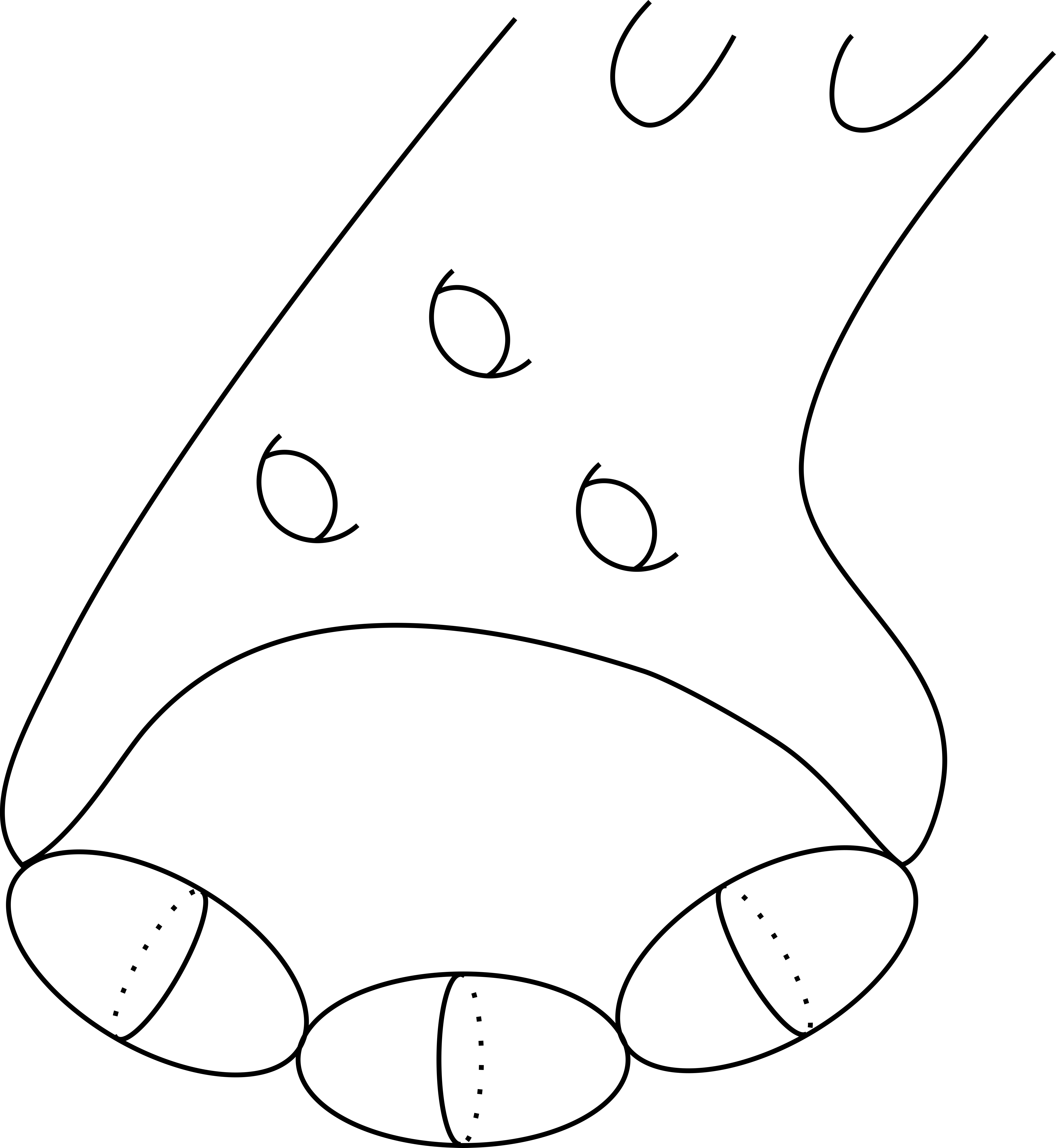}}
	\caption{A sketch of the fibre of $\wttilde_\varepsilon$ above the origin. }
	\label{fig:bpsingmf}
\end{figure}

As computed in \cite[6.3]{HabStackyCurves}, the smooth fibre is a $\gcd(p,\frac{p+q}{\ell})$-punctured curve of genus 
\begin{align*}
g(\Sigma)=\frac{1}{2\ell}(2\ell-1+(p-1)(q-1)-\gcd(\ell q,p+q)).
\end{align*}
We remove discs around the neck regions of $\Sigma$ corresponding to the critical points whose critical value is $0$, and call the resulting surface $\Sigma'$, as before. We then trivialise the fibration over a disc, identifying each fibre with $\Sigma'$. We then define the preliminary vanishing paths and cycles $\vc{\lambda_i\mu_i}{}$ and $\vcpr{0}{m,n}$ as in the loop and chain cases, taking 
\begin{align*}
\theta_{m,n}=\frac{2\pi}{\ell}\frac{np+mq}{pq-p-q}.
\end{align*}
As usual, the vanishing cycles $\vcpr{0}{0,0}$ corresponds to $(\lambda_{1,\text{crit}}^+,\mu_{1,\text{crit}}^+)$, and $\vcpr{0}{m,n}$ are obtained from this in $\Sigma'$ by \eqref{BPAction}. The full description in $\Sigma$ is given by local parallel transport considerations on the necks which were removed from $\Sigma$ to give $\Sigma'$. To yield bona fide vanishing paths, we perturb the fibration slightly and introduce long fingers which go around the critical points when $\theta_{m,n}>2\pi$ in the familiar way. As in the maximally graded case, this in fact now yields a set of transversely intersecting vanishing cycles; there is no need to isotope. \\

Analogously to the previous sections, we have $\arg \lambda_1=\arg\lambda_2=\dots=\arg\lambda_\ell$, and this argument interpolates between
\begin{align*}
-2\pi\ell\left(\frac{n(p-1)+m}{pq-p-q}\right)\quad\text{and}\quad 2\pi\ell\left(\frac{m(q-1)+n}{pq-p-q}\right).
\end{align*}
as the modulus of $\lambda_i$ increases, increasing an equal amount in each neck region. With this, it is then checked in the same way as in the previous two cases that the corresponding category $\mathcal{A}_{\widecheck{\overline{\Gamma}}}$ is presented as the quiver algebra of the quiver with relations given in \cref{figBPQuiver}, where we use the identification 
	\begin{equation}
\label{BPObjectMatch}
\begin{aligned}
\vc{0}{m,n} &\leftrightarrow \obj{0}{i,j}
\\ \vc{\lambda_k\mu_k}{} &\leftrightarrow \obj{w_r}{}[3]
\end{aligned}
\quad \text{with} \quad
\begin{aligned}
i+m&=p-1
\\ j+n&=q-1\\
k+r&=\ell.
\end{aligned}
\end{equation}
Namely, the only compositions are the ones claimed, and the grading of the manifold can again be taken such that all Floer complexes are graded in degree zero. \\

With all this, we arrive at the main theorem in the Brieskorn--Pham case. We state it for the non-maximally graded case, as the maximally graded case is established as a special case of \cite{FutakiUedaBP}. 

\begin{thm}[\cref{mainTheorem}, non-maximally graded undeformed Brieskorn--Pham polynomial case]
	For a Brieskorn--Pham polynomial $\w$ with grading group of index $\ell>1$,  the $\Z$-graded $A_\infty$-category $\mathcal{A}_{\widecheck{\overline{\Gamma}}}$ is described under the identification \eqref{BPObjectMatch}  by the quiver with relations in   \cref{figBPQuiver}, and is formal. In particular, by \cref{BPEndAlgebra}, it is quasi-equivalent to $\mathcal{B}$, and hence there is an induced quasi-equivalence
	\[
	\mathrm{mf}(\C^2, \Gamma, \w) \simeq \mathcal{FS}(\wttilde).
	\]
\end{thm}
\section{Mirror symmetry and deformations.}\label{DeformationsSection}
In this section, we observe that considering Berglund--H\"ubsch--Henningson mirror symmetry for $\ell>1$ gives rise to a genuinely new phenomenon which is not present in the maximally graded case. Namely, there are clear deformations of the categories appearing on both sides of the argument. In this section, we compute and match these deformations, completing the proof of \cref{mainTheorem}. \\

We begin with an observation about Hochschild cohomology for categories of matrix factorisations. 
\begin{lem}\label{BDeformations}
Let $\w$ be an invertible polynomial in two variables with $\Gamma\subseteq \Gamma_\w$ an admissible group of symmetries of index $\ell$. Then,
\begin{align*}
\mathrm{HH}^2(\C^2,\Gamma,\w)&=\begin{cases}
\mathrm{span}_\C\{g_1,\dots, g_{\ell-3}\}\quad \text{if } \w=x^{p}+y^p\text{ for }\ell=p\\
\mathrm{span}_\C\{g_1,\dots, g_{\ell-2}\}\quad \text{if } \w=x^p y+y^{np+1}\text{ or } \w=x^{np}+y^p\text{ for } n>1 \text{ and } \ell=p\\
\mathrm{span}_\C\{g_1,\dots, g_{\ell-1}\}\quad\text{otherwise} 
\end{cases}
\end{align*}
where the $g_i$ are all monomials in $\mathrm{Jac}\ \w$ which are of the same degree as $\w$.
Moreover, 
\begin{align*}
\mathrm{HH}^3(\C^2,\Gamma,\w)&=0,
\end{align*}
and so every infinitesimal deformation of $\mf(\C^2,\Gamma,\w)$ extends to infinite order. In particular, every deformation is realised by the category of $\Gamma$-equivariant matrix factorisations of 
\begin{align}\label{DeformedBModel}
\w_{\vec{\varepsilon}}=\w+\sum_{i=1}^{N}\varepsilon_i g_i,
\end{align}
for $\vec{\varepsilon}=(\varepsilon_1,\dots,\varepsilon_{N})\in\C^{N}$ for $N=\dim_{\C}\mathrm{HH}^2(\C^2,\Gamma,\w)$.
\end{lem}
\begin{proof}
The computations of the Hochschild cohomology vector spaces is a straightforward application of \cite[Theorem 1.2]{BallardFaveroKatzarkov} (see, for example, \cite[Section 3]{LekiliUeda} or \cite[Section 4]{HabMilnorFibreHMS} for closely related calculations). The Jacobians were computed in \cite[Section 2]{HabMilnorFibreHMS}. In particular, the second Hochschild cohomology is, in this case, a count of the number of $\Gamma$-invariant elements of $\mathrm{Jac}\ \w$ in our case, and $\mathrm{HH}^3$ can similarly be seen to vanish since $\w$ has an isolated singularity. 
\end{proof}
In the following, we shall only consider polynomials $\w_{\vec{\varepsilon}}$ which continue to have an isolated hypersurface singularity at the origin. More precisely, this excludes polynomials $\w_{\vec{\varepsilon}}$ which do not have distinct branches at the origin. It will become momentarily clear that the space of polynomials such that $\w_{\vec{\varepsilon}}$ only has an isolated hypersurface singularity at the origin is the complement of $(\ell-1)!$ generic hyperplanes in $\C^{N}$.\\

In this section, we study the loop case for concreteness; however, all arguments go through for the chain and Brieskorn--Pham cases with only cosmetic alteration. To fix notation, we let $\mathcal{E}$ be the tilting object of \cref{LoopTiltingCor} and $A=\mathrm{End}^*(\mathcal{E})(=\mathrm{End}^0(\mathcal{E}))$ the quiver algebra of \cref{figLoopQuiver}. \\
To connect this with the classical deformation theory of an algebra, we observe that the tilting object for the category of matrix factorisations of $\w_{\vec{\varepsilon}}$ is an infinite order deformation of $A$ in the usual sense. The first thing to observe is that $\w_{\vec{\varepsilon}}$ can still be $\Gamma-$equivariantly factored as 
\begin{align}\label{LoopBDeformationFactored}
xy\prod_{r=1}^\ell(x^{\frac{p-1}{\ell}}-e^{\frac{\pi i}{\ell}}\eta^r\alpha_ry^{\frac{q-1}{\ell}}), 
\end{align} 
where $\eta^\ell=1$ is the fixed root of unity as in the computations of \cref{LoopBModel}, and $\prod_{r=1}^\ell\alpha_r=1$. This last condition ensures that the expanding \eqref{LoopBDeformationFactored} yields a deformation of $\w$ of the form \eqref{DeformedBModel}. It is here that it becomes clear that the space of deformations which have an isolated singularity is the complement of $(\ell-1)!$ generic hyperplanes in $\C^{\ell-1}$. In particular, we pairwise exclude the values $\alpha_{r_1}$ and $\alpha_{r_2}$ such that $\eta^{r_1}\alpha_{r_1}=\eta^{r_2}\alpha_{r_2}$. \\
With this, we then run through the argument of \cref{LoopBModel} for $R_{\vec{\varepsilon}}=S/(\w_{\vec{\varepsilon}})$. Whilst, for example, the maps \eqref{EvenDifferential} are no longer in as clean a form as in the undeformed case, the actual computations, with the exception of those in \eqref{KijKwrHoms}, do not change. \\
The key difference between the computation of matrix factorisations for $\w_{\vec{\varepsilon}}$ compared with those of $\w$ is contained in \cref{KijKwrHoms} and its consequences for the relations of the corresponding quiver algebra. In the proof of this lemma (in the undeformed case), we showed that the kernel was one dimensional and identified a basis vector for this space. The only non-trivial part was to show that the image of the last row of $\diff_0^T$ included $w_r$ as a factor\footnote{Recall that our notation is that $\w=xyw_1\dots w_\ell$.}, and so vanished in $R/(w_r)$. Now, for $R_{\vec{\varepsilon}}$, the kernel is still one dimensional, but the basis element changes. In particular, the last row of the corresponding differential must contain $(x^{\frac{p-1}{\ell}}-e^{\frac{\pi i}{\ell}}\eta^r\alpha_ry^{\frac{q-1}{\ell}})$ as a factor. Consequently, the relation corresponding to $(iii)$ in \cref{LoopEndAlgebra} then becomes
\begin{align*}
c_r(x^{\frac{p-1}{\ell}}-e^{\frac{\pi i}{\ell}}\eta^r\alpha_ry^{\frac{q-1}{\ell}})=0.
\end{align*}

Whilst the above shows how to recover a deformation of $A$ in the classical sense from \cref{BDeformations}, and, strictly speaking, this is all that we require, we still find it valuable to explain the reverse implication, building a deformation in \cref{BDeformations} from Hochschild cocycles\footnote{We thank the anonymous referee for this suggestion.}. In particular, we consider $A_{t_r}=A\otimes_\C\C[t_{r}]/(t_r^2)$, and infinitesimally deform the product as
\begin{align*}
u\cdot_{t_r} v=u\cdot v+t_r f(u\otimes v), 
\end{align*}
where $f\in\mathrm{Hom}(A\otimes A,A)=\mathrm{CC}^2(A)$ is a Hochschild cocycle representative of a Hochschild cohomology class $[f]\in\mathrm{HH}^2(A)$ and $-\cdot-$ is the undeformed multiplication on $A$. In order to deform the relation
\begin{align*}
c_r\cdot (x^{\frac{p-1}{\ell}}-e^{\frac{\pi i}{\ell}}\eta^ry^{\frac{q-1}{\ell}})=0,
\end{align*}
we take $f(c_r\otimes(x^{\frac{p-1}{\ell}}-e^{\frac{\pi i}{\ell}}\eta^ry^{\frac{q-1}{\ell}}))=-e^{\frac{\pi i}{\ell}}\eta^rc_ry^{\frac{q-1}{\ell}}$ (and zero on the other generators of the algebra). One can see that this is a cochain by the definition of multiplication on $A$. In particular, there are three possible non-vanishing configurations for $\delta f(u\otimes v\otimes w)$. These are when two of $u,v$ or $w$ are $c_r$ and $(x^{\frac{p-1}{\ell}}-e^{\frac{\pi i}{\ell}}\eta^ry^{\frac{q-1}{\ell}})$ (appearing in that order) and the other is the appropriate idempotent. It is then straightforward to check directly that $\delta f=0$ in these cases as well. The infinitesimally deformed relation becomes 
\begin{align*}
c_r(x^{\frac{p-1}{\ell}}-e^{\frac{\pi i}{\ell}}\eta^r(1+t_r)y^{\frac{q-1}{\ell}})=0.
\end{align*}
 Since $\mathrm{HH}^3(A)=0$, we can extend this deformation to higher order defining a product on $A\otimes_\C\C[t_r]/(t_r^{n+1})$
\begin{align*}
u\cdot_\varepsilon v=u\cdot v+t_r f^{(1)}(u\otimes v)+\dots+ t_r^n f^{(n)}(u\otimes v),
\end{align*}
where the $f^{(k)}\in\Hom(A\otimes A,A)$ are similarly (representatives of) cocycles. Here, we set $f^{(k)}(c_r\otimes(x^{\frac{p-1}{\ell}}-e^{\frac{\pi i}{\ell}}\eta^ry^{\frac{q-1}{\ell}}))=-\frac{1}{k!}e^{\frac{\pi i}{\ell}}\eta^rc_ry^{\frac{q-1}{\ell}}$, and so, by extending our deformation to infinite order, our deformed relation becomes 
\begin{align*}
c_r(x^{\frac{p-1}{\ell}}-e^{\frac{\pi i}{\ell}}\eta^r\big(\sum_{k\geq 0} \frac{t_r^k}{k!}\big)y^{\frac{q-1}{\ell}})=0.
\end{align*}
The sum $\sum_{k=0}^\infty \frac{t_r^k}{k!}$ converges to $e^{t_r}$ for all values of $t_r$. Fixing $t_r$, this defines a deformed product on A. We can then repeat the above argument to deform the other relations corresponding to $(iii)$ in \cref{LoopEndAlgebra}. There are therefore $\ell$ deformation parameters; however, we note that two deformations are equivalent iff, for a fixed constant $k\in\C$, $t_r=t_r'+k$ for all $r=1,\dots \ell$. The equivalence is given by simply scaling the elements labelled by $y$ by a fixed amount which cancels the $e^k$ term in each relation corresponding to $(iii)$ in \cref{LoopEndAlgebra}. Note that the other relations are unaffected by this scaling. In particular, the deformation is trivial iff  $t_1\equiv\dots\equiv t_\ell\bmod 2\pi i$.  
We therefore explicitly see the $\ell-1$ deformation parameters of \cref{BDeformations}. Finally, since scaling each $e^{t_r}$ by a fixed constant is a trivial deformation, we can choose a distinguished representative of each deformation class by setting $\alpha_r=e^{t_r}\beta$ for any fixed constant $\beta$ such that $\beta^{-\ell}=\prod_{k=1}^\ell e^{t_k}$ so that $\prod_{k=1}^\ell \alpha_k=1$. Expanding out $xy\prod_{r=1}^\ell(x^{\frac{p-1}{\ell}}-e^{\frac{\pi i}{\ell}}\eta^r\alpha_ry^{\frac{q-1}{\ell}})$ and matching coefficients of the $g_i$ gives the deformation $\w_{\vec{\varepsilon}}$ in \cref{BDeformations}. \\

With deformations on the B--side understood, we must demonstrate that we can realise these same deformations symplectically. In what follows, we will recall some facts about B--fields and deformations more generally, before specialising to the case at hand. \\
Let $\theta_1,\dots,\theta_N$ be a collection of vanishing thimbles for a Lefschetz fibration $M\rightarrow \C$, in particular such as in \cref{LoopAModel}, \cref{ChainAmodel} or \cref{BPAmodel}. We then let $\Theta=\bigcup \theta_i$ be the union, and recall that computing the coefficients of the Floer products is equivalent to defining a homomorphism 
\begin{align}\label{WeighingHomomorphism}
H_2(M,\Theta;\Z)\rightarrow\C^*,
\end{align}
sending each contributing disc $u:\mathbb{D}^2\rightarrow M$ to an element of $\C^*$, the coefficient of the output of this disc. In principle, any such homomorphism would be permissible, but it is most geometrically meaningful to define it by considering the symplectic areas of the discs and sign contributions coming from the orientations of moduli spaces of such discs. Letting $\phi:H_2(M,\Theta;\Z)\rightarrow \C^*$ be this homomorphism, we will be looking for a homomorphism $\phi'$ such that $\phi+\phi'$ describes the deformed product. \\

Now, observe that, since $\C^*$ is an injective $\Z$-module, $\mathrm{Ext}_\Z^1(H_1(M,\Theta;\Z),\C^*)=0$, and
\begin{align}\label{HomCohomIsom}
H^2(M,\Theta;\C^*)\simeq\mathrm{Hom}_\Z(H_2(M,\Theta;\Z),\C^*)
\end{align} 
by the universal coefficient theorem. For an element which yields the undeformed Floer product, its image under the map
\begin{align*}
H^2(M,\Theta;\C^*)\rightarrow H^2(M;\C^*)
\end{align*}
coming from the long exact sequence for relative cohomology of the pair $(M,\Theta)$ is the identity. Again by the universal coefficient theorem, $H^2(M;\C^*)\simeq \mathrm{Hom}(H_2(M;\C),\C^*)$, where 
\begin{align*}
H_2(M;\C)&\rightarrow \C^*\\
A&\mapsto e^{\int_A B}
\end{align*}
for a fixed $B\in H^2(M;\C)$. In order to deform the Floer products, the idea is to pick such an element $B\in H^2(M;\C)$, consider its image in $H^2(M;\C^*)$ and then weight the Floer products by the homomorphism \eqref{WeighingHomomorphism} corresponding to (an element in) the preimage of this element. For this to make sense, we require\footnote{Different authors have different conventions, stemming from whether one considers the first Chern class of a line bundle to have a factor of $2\pi i$ or not.} that $B|_{\Theta}\in H^2(\Theta;2\pi i\Z)$, so that $B$ is in the kernel of the map
\begin{align}\label{LESMap1}
H^2(M;\C^*)\rightarrow H^2(\Theta;\C^*),
\end{align}
and it can be lifted to $H^2(M,\Theta;\C^*)$. Moreover, this guarantees the existence of a line bundle $\mathcal{L}_i\rightarrow \theta_i$ with connection $\nabla_i$ on each Lagrangian such that $c_1(\mathcal{L})=B|_{\theta_i}=F_{\nabla_i}$. Considering $u:\mathbb{D}^2\rightarrow M$ as an element of $H_2(M,\Theta;\Z)$, the lift of the map $(A\mapsto e^{\int_A B})\in H^2(M;\C^*)$ to $H^2(M,\Theta;\C^*)$ is given by 
\begin{align*}
u\mapsto \mathrm{Hol}(\partial u) e^{\int_{\mathbb{D}^2}u^*B},
\end{align*} 
where $\mathrm{Hol}(\partial u)$ is the holonomy around the boundary of the disc given by the parallel transport induced by $\nabla_i$ along each component, as well as fixed isomorphisms $\mathcal{L}_i|_p\simeq\mathcal{L}_j|_p$ for $p\in \theta_i\cap \theta_j$. It is in this way that, given an element $B\in H^2(M;\C)$ whose restriction to each Lagrangian is integral, one can deform the Floer products. Such a cohomology class is referred to as a \emph{B--field}, and was originally introduced in Floer theory by Fukaya in \cite{Fukaya} as a way of enlarging the class of objects in the Fukaya category. His construction dealt with unitary\footnote{Meaning that the cohomology class takes values in the Lie algebra of the unitary group.} B--fields $iB\in H^2(M;i\R)$, and was expanded upon by \cite{ChoNonDisplaceable} who considered non-unitary B--fields $B\in H^2(M;\C)$. It should be emphasised that we keep the symplectic form fixed, only changing how pseudoholomorphic curves are weighted, not which curves contribute. Deformations of Fukaya--Seidel categories by B--fields in the context of homological mirror symmetry have played an important role, for example in \cite{AOK, AOKDelPezzo}.  \\

We now return specifically to the case at hand, so $M=\widetilde{X}$, the crepant resolution of the $A_{\ell-1}$ singularity. We also exclude the transpose of the sporadic families in \cref{BDeformations} from the discussion. Working with a B--field and explicitly calculating the weights of the discs is the most geometric approach, although can be computational taxing. Instead, we opt to choose directly the homomorphism \eqref{WeighingHomomorphism} such that the relations amongst Floer products matches the deformations on the B--side and then define the B--field to be an element\footnote{This is a non-unique choice, stemming from the fact that the lifts of an element of $H^2(\widetilde{X};\C^*)$ is a torsor over $H^1(\Theta;\C^*)$, as well as the non-uniqueness of the complex logarithm in determining $B\in H^2(\widetilde{X};\C)$ from an element of $H^2(\widetilde{X};\C^*)$. By construction, however, the resulting deformations are the same.} which yields this deformation.\\
In order to define this homomorphism, we continue to send all triangles with boundary on exact Lagrangians to the identity (this ensures that the commutativity relation (i) of \cref{figLoopQuiver} continues to hold). Then, consider triangles whose boundary is on $\vc{0}{\frac{p-1}{\ell},0}\cup\vc{0}{0,0}\cup\vc{\lambda_r\mu_r}{}$ (of which there are two\footnote{We are defining this for the configuration of vanishing cycles which was used to compute $\mathcal{A}_{\widecheck{\overline{\Gamma}}}$. If one were to Hamiltonian isotope these Lagrangians so as to introduce new intersections, it would be necessary to alter this description.} -- one big and one small). For each $\vc{\lambda_r\mu_r}{}$, we send the big and small triangles to $e^{d_r}$ or $e^{c_r}$, respectively, for some $c_r,d_r\in \C$, such that 
\begin{align*}
e^{d_r-c_r}=\alpha_r,
\end{align*}
where $\alpha_r$ was the deformation parameter on the B--side. This can be done consistently since the $\vc{\lambda_r\mu_r}{}$ are pairwise disjoint and each triangle has at most one of these Lagrangians on its boundary. The remaining triangles which contribute to Floer products are then determined by virtue of the map being a homomorphism and the commutativity of the product between exact Lagrangians. From this, we define our B--field as explained above. By construction, the resulting Fukaya--Seidel category is presented as modules over the algebra in \cref{figLoopQuiver} with the third relations deformed by a factor of $\alpha_r$, as in the B--model. 

\begin{rmk}
The above argument might seem objectionable for its lack of geometric origin, but it is actually just cutting to the chase of what we want out of a B--field. It is also not so complicated to see which B--field a particular deformation corresponds to. For example, in \cref{fig:milnorfiberexample}, $e^{d_1}$ is the contribution from the B--field of the green (large) triangle and $e^{c_1}$ is the contribution from the light blue (small) triangle. There are similarly triangles with boundary on $\vc{0}{0,0}$, $\vc{0}{0,1}$ and $\vc{\lambda_2\mu_2}{}$ with B--field contributions $e^{d_2}$ and $e^{c_2}$. Let $Z\in H_2(\widetilde{X},\Theta;\Z)$ be the relative homology class given by the cylinder between $\vc{\lambda_1\mu_1}{}$ and $\vc{\lambda_2\mu_2}{}$. Then, since $Z=d_1-d_2-c_1+c_2$, we have
\begin{align}\label{CylinerMapping}
Z\mapsto e^{d_1-d_2-c_1+c_2}.
\end{align}
Since the two Lagrangian thimbles are symmetric about the zero section in $\widetilde{X}\simeq\mathcal{O}_{\P^1}(-2)$, the connection forms on $\nabla_1$ and $\nabla_2$ at worst differ by an exact one form. In particular, $B|_{\theta_1}$ is the same exact two form on $\theta_1$ as the two form $B|_{\theta_2}$ is on $\theta_2$. Consequently, the holonomies around the boundaries of these Lagrangians, $\int_{\theta_i}B|_{\theta_i}$, are equal. Since $\partial Z=\partial \theta_1-\partial \theta_2$, the holonomy contribution of of \eqref{CylinerMapping} is zero, meaning that
\begin{align*}
(Z\mapsto e^{d_1-d_2-c_1+c_2})\mapsto e^{d_1-d_2-c_1+c_2}\in H^2(\widetilde{X};\C^*)
\end{align*}
in the long exact sequence of the pair. Call $C$ the generator of the second cohomology of $\widetilde{X}$. Then, the B--field which corresponds to this deformation is\footnote{Recall that $C^2=-2$}
\begin{align}\label{BFieldExampleEqn}
B=-\frac{1}{2}(d_1-d_2-c_1+c_2)C.
\end{align}
More generally, one can do this for each generator of $H^2(\widetilde{X};\C)$, although it should be reiterated that this is not necessary for our argument. We point this out only to make connection with the geometric origins of B--fields. 
\end{rmk}
This process yields $\ell$ parameters which we freely deform; however, as on the B--side, there is a relation amongst them. Namely, scaling each of the contributions $d_r-c_r\mapsto \lambda +d_r-c_r$ for some $\lambda\in \C$ yields the same deformation. This can be seen algebraically in the same way as on the B--model, or one can observe that scaling the deformation parameters like this yields the same B--field. Symplectically, it is immediate in \eqref{BFieldExampleEqn} for $\widetilde{X}$ above, and the general case is achieved by iterating this construction.
\begin{rmk}\label{AlternativeDeformationsRMK}
It is worth remarking that we could have started with the undeformed case $\phi:H_2(\widetilde{X},\Theta;\Z)\rightarrow \C^*$ corresponding to counting discs with respect to the exact symplectic form and then finding a B--field such that $\phi+\phi'$ matches deformations of the B--model. We didn't do this for the following reason. In the B--model, the algebra and geometry is given to us by direct calculation, and is naturally the `undeformed' case. We believe that it is a less clear argument to then construct the mirror A--model as a deformation of something, where the B--model isn't. Particularly since we are looking to prove that mirror symmetry holds under deformations, it makes sense that the `base case' for this deformation statement corresponds to where both the A-- and B--models are undeformed. 
\end{rmk}
Putting this all together yields a proof of the the remaining statement of \cref{mainTheorem}.
\begin{thm}[\cref{mainTheorem}, deformed cases] Let $\w$ be an invertible curve singularity with admissible grading group $\Gamma\subseteq\Gamma_\w$, $\wt$ be the Berglund--H\"ubsch transpose with dual grading group $\widecheck{\overline{\Gamma}}$. Furthermore, let $\vec{\varepsilon}\in \C^{N}$ and $\w_{\vec{\varepsilon}}=\w+\sum_{i=1}^{N}\varepsilon_ig_i$ as in \cref{BDeformations} such that $\w_{\vec{\varepsilon}}$ has an isolated hypersurface singularity at the origin. Then, there exists a non-unitary B-field such that there is a quasi-equivalence of $\Z$-graded pre-triangulated $A_\infty$-categories over $\C$
\begin{align*}
\mathrm{mf}(\C^2, \Gamma,\w_{\vec{\varepsilon}})\simeq \mathcal{FS}(\wttilde;B).
\end{align*}
\end{thm}
\begin{rmk}
In the A--model corresponding to the cases where the rank of the second Hochschild cohomology is less than $\ell-1$, there are more exceptional spheres than deformation parameters. This does not affect the argument, however, since we can still add a B-field to realise any deformation on the B-side which we are considering.
\end{rmk}

\end{document}